\documentclass[12pt,twoside,reqno]{amsart}

\usepackage{amssymb,amsmath,amstext,amsthm,amsfonts,amscd,xcolor}

\usepackage{dsfont}

\usepackage{amsthm, enumerate}

\usepackage[ansinew]{inputenc} 
\usepackage{graphicx}
\usepackage[mathscr]{eucal}
\usepackage{hyperref}

\newcommand{\R}{\mathbb{R}}
\newcommand{\Su}{\mathbb{S}}
\newcommand{\C}{\mathbb{C}}
\newcommand{\N}{\mathbb{N}}
\newcommand{\Z}{\mathbb{Z}}

\newcommand{\SL}{{\rm SL}}
\newcommand{\GL}{{\rm GL}}

\newcommand{\Mat}{{\rm Mat}}

\newcommand{\tr}{\mbox{tr}}

\newcommand{\Pp}{\mathbb{P}}

\newcommand{\EE}{\mathbb{E}}
\newcommand{\Escr}{\mathscr{E}}

\newcommand{\cocycles}{\mathcal{C}}

\newcommand{\imeas}[1]{d (#1, \diags^k)}

\newcommand{\Diag}{\mathscr{D}{\rm iag}}
\newcommand{\diags}{\mathscr{D}{\rm iag}}
\newcommand{\Diagast}{\mathscr{D}{\rm iag}^\ast}

\newcommand{\pcal}{\mathscr{P}}
\newcommand{\osc}{\mathrm{osc}}

\theoremstyle{plain}
\newtheorem{theorem}{Theorem}[section]
\newtheorem{proposition}{Proposition}[section]

\newtheorem{lemma}[proposition]{Lemma}
\theoremstyle{definition}
\newtheorem{definition}{Definition}[section]

\theoremstyle{definition}
\newtheorem{remark}{Remark}[section]

\numberwithin{equation}{section}

\newcommand{\sabs}[1]{\left| #1 \right|} 
\newcommand{\abs}[1]{\bigl| #1 \bigr|} 
\newcommand{\norm}[1]{\lVert#1\rVert} 
\newcommand{\normtwo}[1]{
{\left\vert\kern-0.25ex\left\vert\kern-0.25ex\left\vert #1 
    \right\vert\kern-0.25ex\right\vert\kern-0.25ex\right\vert} } 

\newcommand{\medp}{\hat{\mathfrak{v}}}

\newcommand{\mostexpuv}{\mathfrak{v}}

\newcommand{\avg}[1]{\left< #1 \right>} 

\newcommand{\bigo}{\mathcal{O}}

\newcommand{\transl}{{T}} 

\newcommand{\medir}[1]{\mostexp^{({#1})}}
\newcommand{\mostexp}{\overline{\mathfrak{v}}}

\newcommand{\less}{\lesssim}

\newcommand{\ep}{\epsilon} 
  
\newcommand{\la}{\lambda}
\newcommand{\ga}{\gamma}

\newcommand{\om}{\omega}

\newcommand{\GLtwoR}{\GL_2(\R)}
\newcommand{\SLtwoR}{\SL_2(\R)}

\makeatletter
\newsavebox\myboxA
\newsavebox\myboxB
\newlength\mylenA

\newcommand*\xoverline[2][0.75]{%
    \sbox{\myboxA}{$\m@th#2$}%
    \setbox\myboxB\null
    \ht\myboxB=\ht\myboxA%
    \dp\myboxB=\dp\myboxA%
    \wd\myboxB=#1\wd\myboxA
    \sbox\myboxB{$\m@th\overline{\copy\myboxB}$}
    \setlength\mylenA{\the\wd\myboxA}
    \addtolength\mylenA{-\the\wd\myboxB}%
    \ifdim\wd\myboxB<\wd\myboxA%
       \rlap{\hskip 0.5\mylenA\usebox\myboxB}{\usebox\myboxA}%
    \else
        \hskip -0.5\mylenA\rlap{\usebox\myboxA}{\hskip 0.5\mylenA\usebox\myboxB}%
    \fi}
\makeatother

\newcommand{\nldt}{\xoverline{n}}
\newcommand{\Nbar}{\xoverline{N}}

\newcommand{\LE}[1]{L^{({#1})}}  
\newcommand{\An}[1]{A^{({#1})}}  
\newcommand{\Bn}[1]{B^{({#1})}}  
\newcommand{\Dn}[1]{D^{({#1})}}  

\newcommand{\comp}{^{\complement}}

\newcommand{\FF}{\mathscr{F}}

\newcommand{\ind}{\mathds{1}}

\newcommand{\B}{\mathscr{B}}
\newcommand{\Bt}{\tilde{\mathscr{B}}}

\newcommand{\I}{\mathscr{I}}
\newcommand{\J}{\mathscr{J}}

\newcommand{\filt}{\mathcal{E}}

\newcommand{\Qscr}{\mathscr{Q}}
\newcommand{\Bscr}{\mathscr{B}}
\newcommand{\Nscr}{\mathscr{N}}

\newcommand{\Pz}{\Sigma_0}
\newcommand{\Po}{\Sigma_1}
\newcommand{\Pt}{\Sigma_2}
\newcommand{\measA}{E}
\newcommand{\measB}{F}

\newcommand\restr[2]{{
  \left.\kern-\nulldelimiterspace 
  #1 
  \vphantom{\big|} 
  \right|_{#2} 
  }}

\newcommand{\nirred}{\widehat{m}}
\newcommand{\ndiag}{m}

\newcommand{\Xfrk}{\mathfrak{X}}
\newcommand{\one}{{\bf 1}}

\newcommand{\Hscr}{\mathscr{H}}

\title[Large deviations for products of random matrices]{Large deviations for products of random  two dimensional matrices}

\date{}

\begin{document}

\author[P. Duarte]{Pedro Duarte}
\address{Departamento de Matem\'atica and CMAF\\
Faculdade de Ci\^encias\\
Universidade de Lisboa\\
Portugal 
}
\email{pduarte@ptmat.fc.ul.pt}

\author[S. Klein]{Silvius Klein}
\address{Departamento de Matem\'atica, Pontif\'icia Universidade Cat\'olica do Rio de Janeiro (PUC-Rio), Brazil}
\email{silviusk@impa.br}

\begin{abstract}   
We establish large deviation type estimates for i.i.d. pro\-ducts of two dimensional random matrices with finitely suppor\-ted probability distribution. The estimates are stable under pertur\-bations and require no irreducibility assumptions. In consequence, we obtain a uniform local modulus of continuity for the correspon\-ding Lyapunov exponent regarded as a function of the support of the distribution. This in turn has consequences on the modulus of continuity of the integrated density of states and on the localization properties of random Jacobi operators.   
\end{abstract}

\maketitle


\section{Introduction and statements}\label{intro}
Consider a multiplicative random process; that is, let $\{g_n\}_{n\in\Z}$ be a sequence of i.i.d. random matrices relative to a compactly suppor\-ted probability measure $\mu$ on the general linear group $\GL_d(\R)$, and let $g^{(n)} := g_{n-1} \ldots g_1 g_0$ denote the partial products process. The Furstenberg-Kesten theorem, the multiplicative analogue of the law of large numbers, implies that the average process
$\frac{1}{n} \, \log \norm{g^{(n)}}$ converges almost surely to a number $L^+ (\mu)$, called the (maximal) Lyapunov exponent of the process.

A natural problem in this context is to obtain a more {\em quantitative} version of the convergence $\frac{1}{n} \, \log \norm{g^{(n)}} \to L^+ (\mu)$, or in other words, to establish large deviation type estimates or other statistical properties for such multiplicative processes.  

A second problem concerns the continuity of the limit quantity, the Lyapunov exponent $L^+ (\mu)$, with respect to the input data (e.g. the measure $\mu$, or just its support), as the data varies in an appropriate sense.

Both of these problems are highly relevant in the spectral theory of the discrete random (e.g. Schr\"odinger or Jacobi) operators in mathematical physics. 

These topics were studied in the 80s by H. Furstenberg and Y. Kifer and by E. Le Page in a {\em generic} setting (i.e. under irreducibility\footnote{Irreducibility refers to the {\em non} existence of proper subspaces invariant under the closed semigroup $T_\mu$ generated by the support of the  measure $\mu$. There are different versions of this notion.} and contractive\footnote{Contractivity  refers to the existence in $T_\mu$ of matrices with arbitrarily large gaps between consecutive singular values.} assumptions on the support of the measure). Their results have served well the mathematical physics community, as they apply to the Anderson tight binding model (i.e. the discrete random   Schr\"odinger operator on the integer lattice or on a strip lattice). They do not, however, apply to all Jacobi operators, as the corresponding eigenvalue equation may lead to a multiplicative process whose underlying probability measure has {\em reducible} (non generic) support.

More recently, the continuity of the Lyapunov exponent(s) has been considered in the {\em general} (not necessarily generic) setting  by C. Bocker and M. Viana and by A. Avila, A. Eskin and M. Viana. Their results do not provide a modulus of continuity, and in consequence are not applicable to spectral theory problems.\footnote{A more detailed review of such results\textemdash including another new quantitative statement\textemdash follows.}

\vspace{5pt}

We have previously established a link between the two problems formulated above, in the general setting of linear cocycles over an ergodic base dynamical system. More precisely, we have shown that if a given cocycle satisfies certain large deviation type (LDT) estimates, which are {\em uniform} in the data, then necessarily the corresponding Lyapunov exponent (LE) varies continuously with the data, and with a modulus of continuity that depends explicitly on the strength of the LDT estimates.

In this paper we establish such uniform LDT estimates for locally constant 
$\GLtwoR$ cocycles over a full Bernoulli shift in a finite number of symbols. This corresponds to a $\GLtwoR$-valued random multiplicative process whose probability distribution is a finite sum of point masses.\footnote{We note that the case of an absolutely continuous probability distribution was already studied by E. Le Page.}  As a consequence of our general theory, we derive a local modulus of continuity (namely weak-H\"older)  for the Lyapunov exponent regarded as a function of the support of the measure. Furthermore, this implies the same local modulus of continuity for the integrated density of states (IDS) of a random Jacobi operator, near every energy level with positive Lyapunov exponent.   

\vspace{5pt}

This project was in part motivated by a question posed to the second author by G. Stolz, related to his work with J. Chapman on (dynamical) localization for certain disordered quantum spin systems (see~\cite{Chapman-Stolz}). The results in this paper partly address that question, and the approach we use has the potential to completely solve it in the future.

\vspace{5pt}

We are grateful to G. Stolz for his question, as it helped us channel our attention to a simpler yet still interesting model, and thus to a more attainable goal. We would also like to thank S. Griffiths for his help simplifying a probabilistic argument in the prison break section. 

\subsection*{A case study} Let  $\{\omega_n\}_{n\in\Z}$ be an i.i.d. sequence of random variables with common distribution $\mu$, such that $\textrm{ supp } \mu = [a, b] \subset (0, \infty)$ and  $\int  \log t \, d \mu (t)  \neq 0$. 
A toy model for the system considered in~\cite{Chapman-Stolz}  is the Jacobi operator $H_\om$ on $l^2 (\Z)$ defined as follows. If $\psi = \{\psi_n\}_{n\in\Z} \in l^2 (\Z)$,
\begin{equation}\label{toy-Stolz}
[ H_\om \, \psi ]_n := \begin{cases}
\omega_n \psi_{n-1} + \psi_{n+1} & \text{ if } n \text{ is even} \\
\psi_{n-1} + \omega_n \psi_{n+1} & \text{ if } n \text{ is odd} . 
\end{cases}
\end{equation}

The operator $H_\om$ is thus an infinite, tridiagonal, selfadjoint random matrix,  with zero diagonal and symmetric off-diagonals of the form $ \ldots, 1, \om_{-1}, 1, \om_0, 1, \om_1, 1, \ldots$

It is convenient to consider the {\em two-step} transfer matrices of the corresponding eigenvalue equation $H_\om \psi = E \, \psi$, i.e. the matrices $g_n^E$ satisfying
$$
\left[\begin{array}{@{}c@{}}
    \psi_{2n+2} \\
    \psi_{2 n + 1}  
\end{array} \right]
= g_n^E \, 
\left[\begin{array}{@{}c@{}}
    \psi_{2n} \\
    \psi_{2 n - 1}  
\end{array} \right] .
$$
 
Then  
\begin{equation}\label{toy process}
g_n^E = \begin{bmatrix}  
(E^2-1) \frac{1}{\om_n} & - E \om_n  \\
E \frac{1}{\om_n} & - \om_n 
\end{bmatrix} \, ,
\end{equation}
so for every energy parameter $E$, $\{g_n^E\}_{n\in\Z}$  is an $\SL_2(\R)$-valued i.i.d. multiplicative random process. 

Also, the maximal LE of  the operator~\eqref{toy-Stolz} at energy $E$ is half the Lyapunov exponent $L^+ (E)$ of this process.

If $E \neq 0$ it can be shown that the  support of the underlying probability measure satisfies  the irreducibility and contraction properties. This implies the positivity  (by Furstenberg's theorem) and the H\"older continuity (by Le Page's theorem) of the maximal Lyapunov exponent. 

However, when $E=0$, the  process~\eqref{toy process} is {\em diagonal}:
\begin{equation*}
g_n^0 = \begin{bmatrix}  
- \frac{1}{\om_n} & 0 \\
0 & - \om_n 
\end{bmatrix} \, ,
\end{equation*}
hence reducible, while $L^+ (0) = \abs{  \EE ( \log \om_0 ) } =  \abs{ \int \log t \, d \mu (t) }> 0 $. 

\vspace{5pt}

Therefore, in the vicinity of the energy level $E=0$,  the study of the operator~\eqref{toy-Stolz} falls outside the scope of Le Page's theorem. However, when the probability distribution $\mu$ has {\em finite support}, the results in this paper are immediately applicable and they provide a local modulus of continuity of the LE, and thus of the IDS.  This in turn implies the Wegner estimates used in the multiscale analysis that leads to the localization of the operator~\eqref{toy-Stolz}.

\subsection*{The setting} Let $\Sigma=\{1,\ldots, k\}$ be a finite space of symbols and let $p=(p_1,\ldots,p_k)$ with $p_i > 0$ for all $i=1,\ldots, k$ and $p_1+\ldots+p_k=1$ be a pobability vector.
Consider the space $X=\Sigma^\Z$ of all sequences $x=\{x_n\}_{n\in\Z}$ of symbols from $\Sigma$, endowed with the product probability measure $\Pp=p^\Z$. Then the map $T \colon X \to X$, $T \{x_n\}_{n\in\Z} := \{x_{n+1}\}_{n\in\Z}$ is an ergodic transformation called the (full) Bernoulli shift. 

Any function from $\Sigma$ to the general linear group $\GLtwoR$, that is, any $k$-tuple $(A_1,\ldots,A_k)$ of two by two invertible matrices determines the locally constant function $A\colon X \to \GLtwoR$,
$$A \{x_n\}_{n\in\Z} := A_{x_0} \, .$$
The corresponding {\em random cocycle}, or linear cocycle over the Bernoulli shift $T$ is the transformation
$$F_A \colon X\times\R^2\to X\times\R^2, \quad F_A (x, v) := (T x, A(x) v) \,. $$

We identify the transformation $F_A$ with the function $A$ and furthermore with the $k$-tuple $(A_1, \ldots, A_k)$, and refer to either of them as a cocycle. We denote by $\GLtwoR^k$ the space all such cocycles/$k$-tuples. 

%

We endow the space $ \GLtwoR^k$ with the uniform distance 
$$d (A, B) := \sup_{x \in X}  \norm{A (x) - B (x)} = \max_{1\le j \le k} \norm{A_j - B_j} ,$$
where $A=(A_1, \ldots, A_k)$ and $B=(B_1, \ldots, B_k)$.

The $n$-th iterate of a cocycle $A \in \GLtwoR^k$ is given by
\begin{align*}
\An{n} (x) &= A(T^{n-1} x) \ldots A(T x) \, A (x) \\ 
&= A_{x_{n-1}} \ldots A_{x_1} \, A_{x_0} \, ,
\end{align*}
and it encodes a product of $n$ i.i.d. random matrices. 

Given any cocycle $A$, by the Furstenberg-Kesten theorem there exist two numbers, $L^+(A)$ and $L^-(A)$, called the Lyapunov exponents of $A$, such that $\Pp$-almost surely,
$$
\lim_{n\to\infty} \frac{1}{n} \, \log \norm{\An{n}(x)} = L^+(A) \ \text{ and } \
\lim_{n\to\infty} \frac{1}{n} \, \log \norm{\An{n}(x)^{-1}}^{-1} = L^-(A) \,.
$$ 



This is the multiplicative analogue of the law of large numbers.

\vspace{5pt}

It is clear that $L^+(A) \ge L^- (A)$. When $L^+(A) > L^- (A)$, by the Oseledets multiplicative ergodic theorem, there exists a  (proper) measurable decomposition $\R^2 =\filt_A^+(x) \oplus \filt_A^-(x)$ such that $\Pp$-almost surely  $A(x)\,\filt_A^\pm(x) = \filt_A^\pm(T x)$ and
\begin{enumerate}
\item[]  
\;  $\displaystyle  \lim_{n\to\pm\infty} \frac{1}{n}\,\log \norm{\An{n}(x)\,v} =
 L^+(A)$ \ for all $v \in \filt_A^+(x) \setminus\{0\}$
\item[]  
\;  $\displaystyle  \lim_{n\to\pm\infty} \frac{1}{n}\,\log \norm{\An{n}(x)\,v} = 
 L^-(A)$ \ for all $v \in \filt_A^-(x) \setminus\{0\}$.
\end{enumerate}  

The projective line, denoted by
$\Pp=\Pp(\R^2)$, is the space of all lines ($1$-dimensional linear subspaces) in $\R^2$.
Henceforth whenever $\hat p\in\Pp$ the letter `$p$'
will stand for any non-zero vector $p\in \hat p$.
A natural metric is defined in  $\Pp$ by
$$ d(\hat p, \hat q):=\frac{\norm{p\wedge q}}{\norm{p}\,\norm{g}}
= \abs{\sin \angle (p,q)}. $$

We identify the lines $\filt^\pm_A (x)$ with points in the projective space $\Pp(\R^2)$, so the components of the Oseledets decomposition are regarded as functions $\filt_A^\pm \colon X \to \Pp (\R^2)$. 

Let $L^1 (X, \Pp (\R^2) )$ be the space of all Borel measurable functions $\filt \colon X \to \Pp (\R^2)$.  On this space we consider the distance 
$$
d (\filt_1, \filt_2) := \EE_x  \, d (\filt_1 (x), \filt_2 (x) ) \,.
$$

Note that if $A \in \GLtwoR^k$ is a cocycle with $L^+(A) > L^- (A)$, then $\filt^\pm_A \in L^1 (X, \Pp (\R^2) )$.  

\subsection*{Statements} Let us formally introduce the concept of large deviations for iterates of a linear cocycle.

\begin{definition}\label{def:LDT}
A cocycle $A \colon X \to \GL_2(\R)$ satisfies an exponential LDT estimate if there is a constant $c>0$ and for every small enough $\ep > 0$ there is $\nldt = \nldt (A, \ep) \in \N$ such that for all $n \ge \nldt$,
\begin{equation}\label{eq:LDT}
\Pp \left[  \, \left| \frac{1}{n} \log \norm{\An{n}} -  L^+ (A)  \right| > \ep \right] < e^{-c \ep^2 n} \, .
\end{equation}
\end{definition}

We say that a cocycle $A \in \GLtwoR^k$ satisfies a {\em  uniform} exponential LDT estimate if the constants\footnote{We refer to the constants $c$ and $\nldt$ as the LDT parameters of $A$. They depend on $A$, and in general they may blow up as $A$ is perturbed.} $c$ and $\nldt$ above are stable under small perturbations of $A$. We formulate this more precisely below.

\begin{definition}\label{def:unif LDT}
A cocycle $A \in \GLtwoR^k$ satisfies a uniform exponential LDT estimate if there are constants $\delta > 0$, $c > 0$ and for every small enough $\ep > 0$ there is $\nldt = \nldt (A, \ep) \in \N$ such that
\begin{equation}\label{eq:unif LDT}
\Pp \left[  \, \left| \frac{1}{n} \log \norm{\Bn{n}} - L^+  (B) \right| > \ep \right] < e^{-c \ep^2 n} 
\end{equation}
for all cocycles $B \in \GLtwoR^k$ with $d(B,A)<\delta$ and for all $n\ge \nldt$.
\end{definition}

Below we define a weaker version of the LDT estimate, where instead of the exponential, we have a sub-exponential decay of the probability of the tail event.
 
\begin{definition}\label{def:weakLDT}
A cocycle $A \colon X \to \GL_2(\R)$ satisfies a sub-exponential LDT estimate if there are constants $\nldt \in \N$ and $a, b > 0$ such that for all $n \ge \nldt$,
\begin{equation}\label{eq:weakLDT}
\Pp \left[  \, \left| \frac{1}{n} \log \norm{\An{n} } - L^+ (A)  \right| > n^{-a} \right] < e^{- n^b} \, .
\end{equation}

Such an estimate is called uniform if the LDT parameters $\nldt, a, b$ are stable under perturbations of the cocycle $A$ in $\GLtwoR^k$.
\end{definition}

\begin{definition} A cocycle $A=(A_1, \ldots, A_k) \in \GLtwoR^k$ is called diagonali\-zable if the matrices $A_1, \ldots, A_k$ are simultaneously diagonali\-zable over $\R$. 
\end{definition}

We are now ready to formulate the main result of this paper, from which everything else follows.

\begin{theorem}\label{main-uniform-ldt}
Consider the space $\GLtwoR^k$ of locally constant cocycles over a full Bernoulli shift in $k$ symbols. Let $A \in \GLtwoR^k$ be a cocycle with $L^+ (A) > L^- (A)$.

If  $A$ is diagonalizable, then it satisfies a uniform sub-exponential LDT estimate.

If $A$ is not diagonalizable, then it satisfies a uniform exponential LDT estimate.
\end{theorem}

In order to formulate our result on the continuity of the LE and that of the Oseledets decomposition, we  define some specific  moduli of continuity.

\begin{definition}\label{def:modcont}
Let $(M,d)$ and $(N,d)$ be two metric spaces and let $f \colon M \to N$ be a function.

We say that $f$ is locally H\"older continuous near $a \in M$ if there are a neighborhood $\Nscr_a$ of $a$ in $M$ and constants $C<\infty$ and $\alpha>0$ such that for all $b_1, b_2 \in \Nscr_a$  we have
\begin{equation*}
d \left(f(b_1), f(b_2)\right) \le C \, d (b_1, b_2)^\alpha \,.
\end{equation*}

Moreover, $f$ is locally weak-H\"older continuous near $a \in M$ if there are a neighborhood $\Nscr_a$ of $a$ in $M$ and constants $C<\infty$, $\alpha>0$ and $\sigma \in (0,1]$ such that for all $b_1, b_2 \in \Nscr_a$  we have
\begin{equation*}
d \left(f(b_1), f(b_2)\right) \le C \, \exp \left(-\alpha \left( \log \frac{1}{d (b_1, b_2)} \right)^\sigma \right) \,.
\end{equation*}

Note that the case $\sigma=1$ corresponds to H\"older continuity.
\end{definition}

The following continuity statements are direct consequences of Theorem~\ref{main-uniform-ldt} and the abstract continuity theorem in \cite{DK-book} or~\cite{DK-31CBM}.

\begin{theorem}\label{main-continuity} Consider the space $\GLtwoR^k$ of locally constant cocycles over a full Bernoulli shift in $k$ symbols. Then the Lyapunov exponents functions $L^\pm \colon \GLtwoR^k \to \R$ are continuous at all points.\footnote{Thus we obtain another proof  of the result of C. Bocker and M. Viana on the continuity of the LE on the whole space of cocycles, in the finite support setting.}

\vspace{5pt}

Furthermore, let $A \in \GLtwoR^k$ be a cocycle with $L^+ (A) > L^- (A)$. 

If  $A$ is diagonalizable, then locally near $A$, the Lyapunov exponents $L^+$ and $L^-$ are weak-H\"older continuous. 

Moreover, locally near $A$, the maps $B \mapsto \filt^\pm_B \in  L^1 (X, \Pp (\R^2) )$ are well defined and weak-H\"older continuous.

If  $A$ is not diagonalizable, then the same results hold but with the stronger, H\"older modulus of continuity.
\end{theorem}

We note that if $A$ is a reducible cocycle with $L^+ (A) = L^- (A)$, then near $A$ the modulus of continuity of the LE may drastically deteriorate (below weak-H\"older), as shown in our recent work~\cite{DKS} with M. Santos.

\subsection*{Applications to mathematical physics}
We now describe some consequences of the continuity result above to the spectral theory of random Jacobi operators. 

Let $\{v_n\}_{n\in\Z}$ and $\{w_n\}_{n\in\Z}$ be two i.i.d. sequences of real-valued,  random variables (with possibly different distributions) such that moreover $\{v_n\}$ and $\{w_n\}$ are independent from each other. Assume also that $v_1, w_1$ are almost surely bounded and $\EE ( \log \sabs{w_1} ) > - \infty$.

The corresponding random Jacobi operator is the operator $H_\om$ on $l^2 (\Z)$ defined as follows: 
if  $\psi = \{\psi_n\}_{n\in\Z} \in l^2 (\Z)$, then for all $n\in\Z$,
\begin{equation}\label{Jacobi-op}
\big[ H_\om \psi \big]_n  := - (w_{n+1} (\om) \psi_{n+1} + w_n (\om) \psi_{n-1}) 
+ v_n (\om)  \psi_n  \, .
\end{equation}

Denote by $P_n\colon l^2(\Z)\to \R^{n}$ the coordinate projection to the range $\{0,1, 2, \ldots, n-1\} \subset \Z$, and let 
\begin{equation}
\label{H truncation}
H_\om^{(n)}  := P_n \, H_\om  \, P_n^\ast,
\end{equation}
where $P_n^\ast$ stands for the adjoint of $P_n$.
$H_\om^{(n)}$ is called the finite volume truncation of $H$.
  By ergodicity, the following almost sure limit exists
$$N (E) := \lim_{n\to\infty} \, \frac{1}{n} \, \# \big( (- \infty, E] \cap \text{ Spectrum of } H_\om^{(n)}  \big) \,,$$
and the function $E \mapsto N (E)$ is called the {\em integrated density of states} (IDS) of the ergodic operator $H_\om$ (see~\cite{David-survey}).

The LE and the IDS are related via the Thouless formula:
$$L (E) = \int_\R \log \abs{E-E'} d N (E')\,,$$ 
which essentially describes one function as the Hilbert transform of the other.
Thus while the IDS is known to always be $\log$-H\"older continuous, a stronger modulus of continuity can be derived from that of the LE.

\vspace{5pt}

We recall some notions of localization for a random operator. More details can be found in the self contained introduction~\cite{Stolz-introAL} to the topic of Anderson localization.

\begin{definition} Let $\I \subset \R$ be an open interval of energies such that $\I$ intersects the almost sure spectrum of $H_\om$.

An operator $H_\om$ satisfies {\em spectral localization} in $\I$ if almost surely $H_\om$ has pure point spectrum in $\I$ (i.e. $\I$ contains no continuous spectrum). 
 
A stronger (and also more physically relevant) property is that of {\em dynamical localization}, which ensures that solutions of the time dependent Schr\"odinger-type  equation $H_\om \psi (t) = i \, \partial_t \psi (t)$ stay localized in space, uniformly for all times.

An operator $H_\om$ satisfies (a strong form of) dynamical localization in $\I$ if for every compact subinterval $\J \subset \I$, there are constants $C < \infty$ and $\ga, \sigma >0$ such that for all $j, k \in \Z$, 
\begin{equation*}
\EE \left( \sup_{t\in\R} \, \abs{ \avg{ e_j, e^{- i t H_\om} \,  \ind_\J (H_\om) \, e_k} } \right) \le C \, e^{- \ga \sabs{j-k}^\sigma} \, ,  
\end{equation*}
where $\{e_j\}_{j\in\Z}$ is the canonical orthonormal basis in $l^2(\Z)$ and  $\ind_\J (H_\om)$ denotes the spectral projection of $H_\om$ onto the interval $\J$. Both
$e^{- i t H_\om}$ and $\ind_\J (H_\om)$ are defined via functional calculus for selfadjoint operators.
\end{definition}

\begin{theorem}\label{cont-IDS} Consider the random Jacobi operator $H_\om$ defined in~\eqref{Jacobi-op}, and assume moreover that the distributions of the random variables $\{v_n\}$, $\{w_n\}$ have finite support.

Let $E_0$ be any energy level such that $L^+ (E_0) > 0$. There exists an open interval $\I$ containing $E_0$ such that on $\I$, the Lyapunov exponent $L^+$ is positive, the integrated density of states $N$ is weak-H\"older conti\-nuous and almost surely, the operator $H_\om$ satisfies dynamical localization (provided $\I$ intersects the almost sure spectrum of $H_\om$).
\end{theorem}

We remark that the case study~\eqref{toy-Stolz} does not exactly fit the above theorem, since the sequence of `weights' $\{w_n\}_{n\in\Z}$ is given by $$ \ldots, \om_{-1}, 1, \om_0, 1, \om_1, \ldots \, ,$$ thus it is not independent. It is however Markov, hence the Thouless formula relating the LE to the IDS still holds (it holds for every ergodic system); moreover, as the LE of $H_\om$ can be obtained from the i.i.d. two-step transfer matrices $g_n^E$ in~\eqref{toy process}, its weak-H\"older continuity is ensured, and with it, all the other properties (modulus of continuity of the IDS and localization near the energy $E=0$).

\subsection*{Related results} We first review previous work on limit theorems for multiplicative random processes. E. Le Page proved  central limit theorems, as well as a large deviation principle~\cite{LP} for generic (strongly irreducible and contracting) 
$\GL_d(\R)$-cocycles.
Later P. Bougerol extended Le Page's approach, proving similar
results for Markov type random  cocycles~\cite{Bou}. Their results are asymptotic and do not provide uniformity of the large deviation parameters in the data.  In~\cite{DK-book} (see also~\cite{DK-31CBM}) we established, under similar (but weaker) assumptions, finitary (i.e. non asymptotic) and  
{\em uniform} LDT estimates for strongly mixing Markov  (so in particular also Bernoulli) cocycles.

Regarding  the continuity of the LE of random linear cocycles, the first result was obtained by 
 H. Furstenberg and Y. Kifer ~\cite{FKi}
for generic (irreducible and contracting) $\GL_d(\R)$-cocycles. In~\cite{LePage} E. Le Page proved   the H\"older continuity of the maximal LE for a one-parameter family of 
random $\GL_d(\R)$-cocycles satisfying similar generic assumptions. A new proof of Le Page's result (with a formulation for the whole space of quasi-irreducible cocycles)  was recently obtained by A. Baraviera and P. Duarte~\cite{BD}. 

The continuity in the whole space of random $\GL_2(\R)$-cocycles (without any generic assumption),
was established by C. Bocker-Neto and M. Viana in~\cite{BV}. The analogue of this result for random $\GL_d(\R)$-cocycles ($d\ge 2$) has been announced by A. Avila, A. Eskin and M. Viana (see~\cite[Note 10.7]{Viana-book}). An extension of~\cite{BV} to a particular type of cocycles
over Markov systems (particular in the sense that the cocycle still depends on one
coordinate, as in the Bernoulli case) was obtained by E. Malheiro and M. Viana in~\cite{Viana-M}. 

Very recently, and using completely different methods from ours, E.Y. Tall and M. Viana~\cite{Tall-Viana} considered  the problem of establishing a {\em pointwise} modulus of continuity for the LE in the  whole space of random $\GL_2(\R)$-cocycles. They established H\"older continuity at cocycles A with $L^+ (A) > L^- (A)$ and a weak ($\log$-H\"older) minimum modulus of continuity  at all points. Moreover, compared to Theorem~\ref{main-continuity} in this paper, the results in~\cite{Tall-Viana} apply to any compactly supported probability distribution (and not just to finitely supported ones). However, they lack uniformity (providing a pointwise rather than a local modulus of continuity), and in particular are not suitable for mathematical physics applications.

An interesting question arising from this comparison is whether a uniform (rather than pointwise) H\"older modulus of continuity holds in the vicinity of a cocycle with simple Lyapunov exponents, or else whether our weak-H\"older result is optimal. 

Other natural problems related to the results in this paper are the finite state Markov setting (which we treat in a forthcoming paper), the general compact (rather than finite) support distribution case, and the much more challenging higher dimensional cocycles setting.   

\subsection*{The proof structure} In~\cite{DK-book}, we introduced a general method\footnote{This method is based on ideas introduced by M. Goldstein and W. Schlag~\cite{GS-Holder} in their study of quasi-periodic Schr\"odinger operators.} for establishing a modulus of continuity for the LE (and for the Oseledets decomposition) of linear cocycles, that employs as a black box the availabi\-lity of uniform LDT estimates. Such estimates are to be obtained separately, using methods specific to the model of linear cocycles considered (we treat the irreducible random and the quasi-periodic models in the book). Our more recent monograph~\cite{DK-31CBM} provides a self-contained and less technical introduction to this theory, as it is concerned with the $\SL_2(\R)$ setting, which suffices for this paper.

The continuity Theorem~\ref{main-continuity} is thus a direct consequence of the large deviation estimates in Theorem~\ref{main-uniform-ldt}, while the modulus of continuity of the IDS in Theorem~\ref{cont-IDS} follows via the Thouless formula from that of the LE. Furthermore, a good enough modulus of continuity of the IDS implies a sharp enough Wegner estimate which, via multiscale analysis (a standard technique in the theory of random Schr\"odinger type operators) leads to the localization statement formulated above.       

\vspace{5pt}

We now describe the structure of the proof of the uniform LDT estimates. 

In~\cite{DK-31CBM} we derived uniform LDT estimates for {\em quasi-irreducible} cocycles, that is, for cocycles satisfying a rather weak form of irreducibility. It turns out that any random $\SL_2(\R)$ cocycle is diagonalizable or else, itself or its inverse is quasi-irreducible. This will allow us to reduce the problem to  $\SL_2(\R)$-valued diagonalizable cocycles (see Section~\ref{diagred}). 

A diagonalizable cocycle trivially satisfies LDT estimates (or any other limit theorems), and it does so uniformly within the set $\Diag$ of all diagonalizable cocycles. The problem is that this set is not open. Nearby quasi-irreducible cocycles will satisfy LDT estimates, but there is no a priori reason why their LDT parameters do not blow up while approaching $\Diag$, and herein lies the challenge of the proof. 

In fact, referring to Definition~\ref{def:weakLDT}, we will show that $a, b$ are some universal constants, and the problem of stability under perturbations of the LDT parameters will reduce to the stability of $\nldt$, the minimum number of iterates required for the LDT to start applying.
The idea of the proof is to introduce a {\em quantitative measurement} of irreducibility, and relate it to that minimum number of iterates $\nldt$.
There are different ways of measuring the irreducibility of a cocycle (see Section~\ref{irredm}, where these measurements are introduced and related).

For the purpose of this introduction, given a cocycle $B$, consider its irreducibility measurement to be the distance  $d (B, \Diag)$ to the set of diagonalizable cocycles. If the cocycle $B$ is close enough to a diagonalizable cocycle $B_\flat$, then {\em up to} a certain finite  number of iterates $\ndiag (B)$ that depends on $d (B, B_\flat)$, the LDT estimate of $B_\flat$ transfers over to $B$ by proximity. On the other hand, {\em after}  a large enough number  of iterates $\nirred (B)$, the irreducibility kicks in and it provides an LDT for $B$ (see Section~\ref{irredldt}).  
There are then two issues left to address. 

The first is to derive an explicit relationship between the irreducibi\-lity measurement $d (B, \Diag)$ of a cocycle and the minimum number of iterates $\nirred (B)$ required so that the (quantitative) LDT for quasi-irreducible cocycles applies. A key ingredient in this derivation is a probabilistic (random walk) argument which we refer to as the prison break (see Section~\ref{prison}). We note that this is the only part in this paper requiring the finite support condition on the probability distribution. 

The second  is to {\em bridge} the range of values between the iterates $\ndiag (B)$ and $\nirred (B)$. This is achieved in Section~\ref{bridging}, by performing an `almost linearization' of the multiplicative process by means of the Avalanche Principle\textemdash  a deterministic result on the product of a long chain of matrices satisfying some geometric conditions (see~\cite{DK-book, DK-31CBM}).

\section{Reduction to diagonalizable cocycles}\label{diagred}
In this section we show that establishing uniform LDT estimates for $\GLtwoR$ random cocycles and a modulus of continuity for their LE can be reduced to the setting of $\SLtwoR$ valued diagonalizable cocycles. 

\subsection*{Special linear group reduction} If $g\in\GLtwoR$, let 
$$g_\star := \frac{1}{\sqrt{\sabs{\det g}}} \, g \in \SL_2^\pm (\R) \, ,$$
where $ \SL_2^\pm (\R)$ is the group of two dimensional matrices with determinant $\pm 1$. As in this context it makes no difference whether the determinant of a matrix is $1$ or $-1$, we will disregard the sign of the determinant and simply write $\SL_2(\R)$.
 
Moreover, for a cocycle $A\in\GLtwoR^k$, denote by $A_\star$ the induced $\SLtwoR^k$ cocycle given by
$$A_\star (x) :=  \frac{1}{\sqrt{\sabs{\det A (x)}}} \, A (x) \, .$$

Note that since for any matrix $g_\star \in \SLtwoR$ the operator norm satis\-fies $\norm{g_\star^{-1}} = \norm{g_\star} \ge1$, if $A_\star\in\SLtwoR^k$, then 
$$L^+ (A_\star) \ge 0 \ge L^- (A_\star) \quad \text{and} \quad L^-(A_\star) = - L^+(A_\star) \, .$$

Because of this, from now on, if $A_\star\in\SLtwoR^k$, we will denote its maximal Lyapunov exponent $L^+(A_\star)$ by $L(A_\star)$ and refer to it as {\em the} Lyapunov exponent of $A_\star$.

\begin{proposition}\label{SL reduction}
Let $A\in\GLtwoR^k$ be a cocycle. If $A_\star$ satisfies a uniform LDT estimate in $\SLtwoR^k$, then $A$ satisfies a similar LDT estimate in $\GLtwoR^k$. Moreover, if $L \colon \SLtwoR^k \to \R$ has a certain (at most Lipschitz) modulus of continuity near $A_\star$, then $L^+ \colon \GLtwoR^k\to\R$ has the same modulus of continuity near $A$.
\end{proposition}

\begin{proof}
It is easy to verify that the maps
$$\GLtwoR \ni g \mapsto \log \sabs{\det g} \in \R \quad \text{and} \quad \GLtwoR \ni g \mapsto g_\star \in \SLtwoR$$
are locally Lipschitz continuous.

This then implies that the maps
$$\GLtwoR^k \ni A \mapsto \xi_A := \log \sabs{\det A} \in L^\infty(X, \R) \quad \text{and}$$
$$\kern-9.2em \GLtwoR^k \ni A \mapsto A_\star \in \SLtwoR^k$$
are locally Lipschitz as well.

The continuity statement in the proposition is then evident. To derive the LDT statement, note that
\begin{equation}\label{SL reduction LDT}
\frac{1}{n} \, \log \norm{\An{n}} = \frac{1}{n} \, \log \norm{\An{n}_\star} + \frac{1}{2 n} \sum_{i=0}^{n-1} \xi_A \circ T^i \, .
\end{equation}

It is then enough to establish a uniform LDT estimate for the Birkhoff sum on the right hand side. 

Let us recall the Hoeffding inequality from classical probabilities (see for instance  \cite{Tao-2012}), which will be used again later.

\begin{lemma}\label{Hoeffding}
Let $\xi_0, \xi_1, \ldots, \xi_{n-1}$ be a (finite) random process and denote by  $S_n := \xi_0 + \xi_1 + \ldots + \xi_{n-1}$ the corresponding sum process.

Assume that for $0 \le j \le n-1$ the random variables $\xi_j$ are independent and that almost surely 
 $\abs{\xi_j} \le K$. Then for all $\ep > 0$ we have
$$\Pp \, \left[ \left| \frac{1}{n} S_n - \EE \left( \frac{1}{n} S_n \right) \right| > \ep \right] \le e^{- \frac{\ep^2}{2 K^2} \, n} .$$
 \end{lemma}
 
 What makes this estimate important for our purposes (more so than say, Cram\'er's large deviation principle) is its {\em uniformity} in the input data: the parameters in the estimate depend only on the almost sure bound on the process.

If in~\eqref{SL reduction LDT} we put $\xi_i := \xi_A \circ T^i$, then Hoeffding's inequality applies, and it does so uniformly for all nearby cocycles, since $B\mapsto \xi_B$ is Lipschitz, so the almost sure bound on $\xi_B$ is uniform in $B$.
\end{proof}

From now on we will only consider $\SLtwoR$ cocycles.

\subsection*{Reduction to diagonalizable cocycles} We first define a weak form of irreducibility.

Let $A \in \SLtwoR^k$ and assume that the line $\ell \subset \R^2$ is {\em invariant} under all matrices (regarded as linear transformations) $A_j$ with $j \in \Sigma$.  
In other words, $A_j \ell = \ell$ for all $j \in \Sigma$, and we can consider the restriction $\restr{A}{\ell}$ of the cocycle $A$ to the one dimensional subspace $\ell$. This restriction may be described as $\restr{A}{\ell}  \, v = \la_A  \, v$ for some function $\la_A \colon \Sigma \to \R$ and unit vector $v\in\ell$. The process $\log \norm{\restr{\An{n}}{\ell}} $ is thus additive, and by Birkhoff's ergodic theorem, the Lyapunov exponent of $\restr{A}{\ell}$ is
$L (\restr{A}{\ell}) =   \int_\Sigma \log \sabs{\la_A} \, d p = \EE \left( \log \norm{\restr{A}{\ell}} \right)$. 

Note also that from the Oseledets theorem, $L (\restr{A}{\ell}) = L^+(A) = L(A)$ or $L (\restr{A}{\ell}) = L^-(A) = - L(A)$.

\begin{definition}
A cocycle $A \in \SLtwoR^k$ is called {\em quasi-irreducible}
if
there is no line  $\ell \subset \R^2$ such that $\ell$ is invariant under all matrices $A_j$ with $j \in \Sigma$ and $L(\restr{A}{\ell})<L(A)$. 
\footnote{We later prove (see Lemma~\ref{qirred def})  that when $L(A)>0$, this is equivalent to the concept of quasi-irreducibility introduced in~\cite{Bou} (see D\'{e}finition 2.7).}
\end{definition}

\begin{definition}
A cocycle $A\in\SL_2(\R)^k$
is called {\em diagonalizable} if there exist two transversal invariant lines 
lines $\ell$ and $\ell'$, that is, 
$A_j  \ell=\ell$ and $A_j \ell'=\ell'$ for all $j\in\Sigma$. 

We denote by $\Diag$ the set of all diagonalizable cocycles in $\SLtwoR^k$ and by $\Diag^\ast$ the set of cocycles $A\in \Diag$ with $L(A)>0$.
\end{definition}

Given a cocycle $F_A\colon X\times\R^2\to X\times\R^2$,
$F_A(x,v)=(T x, A(x) v)$, determined by a function $A\colon X\to \SL_2(\R)$, its {\em inverse}   is the map 
$F_{A}^{-1}\colon X\times\R^2 \to X\times\R^2$,  \, $F_{A}^{-1}(x,v)=(T^{-1} x, A(T^{-1} x)^{-1}\,v)$.
\index{linear cocycle!inverse}
The iterates of the inverse cocycle $A^{-1}\colon X\to \SL_2(\R)$ satisfy for all $n\in\N$ and $x\in X$,
$$ (A^{-1})^{(n)}(x) = \An{n}(T^{-n} x)^{-1} =: \An{-n}(x)\;.$$

It is then clear that $L (A^{-1}) = L(A)$.

Moreover, if the line $\ell \subset \R^2$ is $A$-invariant then it is also invariant for $A^{-1}$ and
$L (\restr{A^{-1}}{\ell}) =  \int_\Sigma \log \sabs{\la^{-1}_A} \, d p = - L (\restr{A}{\ell})$.

\begin{lemma}\label{dichotomy}
For a cocycle $A\in\SLtwoR^k$, the following dichotomy holds: $A$ is diagonalizable or else, $A$ or $A^{-1}$ is quasi-irreducible. 
\end{lemma}

\begin{proof} Let $A\in\SLtwoR^k$ be a non diagonalizable cocycle. Then either $A$ has no invariant lines, so in particular it is quasi-irreducible, or it has exactly one invariant line $\ell$. In this case, either $L(\restr{A}{\ell}) = L(A)$, so $A$ is quasi-irreducible, or $L(\restr{A}{\ell})= - L(A)$. But then $\ell$ is also invariant for $A^{-1}$ and $L(A^{-1}) = L(A) = - L(\restr{A}{\ell}) = L (\restr{A^{-1}}{\ell})$, so $A^{-1}$ is quasi-irreducible.
\end{proof}

\begin{lemma}\label{inverse ldt} If a cocycle $A\in\SLtwoR^k$ satisfies a uniform LDT estimate, then $A^{-1}$ satisfies the same uniform LDT estimate.
\end{lemma}

\begin{proof}
It is easy to see that the map $\SLtwoR^k \ni B \mapsto B^{-1} \in \SLtwoR^k$ is locally bi-Lipschitz. 
Then for any cocycle $B^{-1}$ near $A^{-1}$ (which ensures the proximity of $B$ to $A$), since
$$ (B^{-1})^{(n)}(x) = \Bn{n}(T^{-n} x)^{-1} \quad \text{ and } \quad L(B^{-1}) = L(B) \, ,$$
it follows that the event
\begin{align*}
 \left[  \, \left| \frac{1}{n} \log \norm{(B^{-1})^{(n)}} - L (B^{-1})  \right| > n^{-a} \right]  & = \\
&  \kern-4em   T^{-n} \, \left[  \, \left| \frac{1}{n} \log \norm{\Bn{n} } - L (B)  \right| > n^{-a} \right]  .
 \end{align*}

We then conclude that
\begin{align*}
\Pp \,  \left[  \, \left| \frac{1}{n} \log \norm{(B^{-1})^{(n)}} - L (B^{-1})  \right| > n^{-a} \right]  & = \\
&  \kern-4em   \Pp \, \left[  \, \left| \frac{1}{n} \log \norm{\Bn{n} } - L (B)  \right| > n^{-a} \right] ,
 \end{align*}
which shows that $B^{-1}$ satisfies the same LDT as $B$.
\end{proof}

%
%
%

\begin{remark}\label{reduction diag remark}
We have already established exponential uniform LDT estimates for quasi-irreducible cocycles (see Theorem 4.1 in~\cite{DK-31CBM}). Thus the two lemmas above reduce the proof of Theorem~\ref{main-uniform-ldt} to establishing (sub-exponential) uniform (relative to the whole space $\SLtwoR^k$) LDT estimates for diagonalizable $\SLtwoR$ cocycles. 
\end{remark}

\subsection*{LDT in the set of diagonalizable cocycles}
 
%
%
%

We derive a uniform  LDT estimate {\em within} the set $\Diag^\ast$ of diagonalizable cocycles with posi\-tive Lyapunov exponent.


\begin{theorem}\label{ldt diag}
Let $A \in \Diag$ with $L(A)>0$. There are constants $\delta=\delta(A)>0$, $c=c(A)>0$ and $C=C(A)<\infty$ such that if $D\in\Diag$ is any diagonalizable cocycle with $d(D,A)<\delta$, then
\begin{equation*}\label{ldt diag eq}
\Pp \left[  \left| \frac{1}{n} \, \log \norm{D^{(n)} } - L(D)  \right| > \ep  \right] \le e^{- c(A)\,\ep^2 \, n}  
\end{equation*}
for all $\ep >0$ and $n\ge \frac{C(A)}{\ep}$. 
\end{theorem}

\begin{proof}
Let $A=(A_1,\ldots,A_k)\in\Diagast$. This is equivalent to the exis\-tence of a matrix $p_A\in\GL_2(\R)$ and a diagonal cocycle $D_A\colon X\to\SL_2(\R)$, such that 
\begin{equation}\label{diag formula}
A(x) = p_A \, D_A (x) \, p_A^{-1} \quad \text{for all } \ x\in X \, .
\end{equation}

We write
$$D_A (x) = \begin{bmatrix}  \theta_A (x) & 0 \\ 0 & \theta_A (x)^{-1} \end{bmatrix} \, ,$$
where $\theta_A\colon X\to\R\setminus\{0\}$ is a locally constant function (in other words, it can be identified with a vector $(\theta_{A,1}, \ldots, \theta_{A,k}) \in\R^k$).

The numbers $\theta_A (x)$ and $ \theta_A (x)^{-1}$ are the eigenvalues of the matrix $A(x)$, while  the columns of the matrix $p_A$ (which simultaneously  diagona\-lizes $A(x), x\in X$) are corresponding eigenvectors. We show that we can choose $p_A$ and $D_A$ so that $A\mapsto p_A$ and for all $x\in X$, $A\mapsto D_A(x)$ are locally continuous (in fact, locally analytic).

First note that
\begin{equation}\label{L diag}
L (A) = \abs{ \EE \, \log \sabs{\theta_A} } = \Big|  \sum_{j=1}^k  p_j  \, \log \sabs{\theta_{A, j}} \big| \, .
\end{equation}

Indeed, for all $x\in X$ we have $\An{n} (x) = p_A \, D_A^{(n)} (x) \, p_A^{-1}$, and 
$$D_A^{(n)} (x) = \begin{bmatrix}  \theta^{(n)}_A (x) & 0 \\ 0 & \theta^{(n)}_A (x)^{-1} \end{bmatrix} \quad \text{where} \quad \theta^{(n)}_A (x) := \prod_{j=n-1}^{0} \theta_A (T^j x) .$$

Thus $$\norm{ D_A^{(n)} (x)  } = \max \big\{  \abs{ \theta^{(n)}_A (x) }, \abs{ \theta^{(n)}_A (x)^{-1} } \big\}$$ and it follows that for all $x\in X$,
\begin{equation}\label{diag bound}
\left| \frac{1}{n} \, \log \norm{\An{n} (x) }  - \Big| \frac{1}{n} \, \sum_{j=0}^{n-1} \log \sabs{\theta_A (T^j x) } \Big|  \right| \le \frac{1}{n}  \left( \log \norm{p_A} + \log \norm{p_A^{-1}} \right) \, .
\end{equation}

Applying the Furstenberg-Kesten theorem to the cocycle $A$ and the Birkhoff ergodic theorem to the observable $\log \sabs{\theta_A}$, we derive~\eqref{L diag}.

Then since $L(A)>0$ and $L (A) = \Big|  \sum_{j=1}^k  p_j  \, \log \sabs{\theta_{A, j}} \big|$, there is an index $j$ such that $\sabs{\theta_{A,j}}\neq 1$. Therefore, at least one of the matrices $A_j$, say $A_1$ is hyperbolic. 

In particular, its eigenvalues $\theta_{A,1}$ and $\theta_{A,1}^{-1}$ are {\em simple}. As a consequence of the general perturbation theory of linear operators (see~\cite{Kato}, or the more modern approach~\cite{Benoit-lin-op}) we have the following. As we perturb the cocycle $A$ in a small enough neighborhood in $\Diagast$, its first component $A_1$ is perturbed in a small neighborhood in $\GLtwoR$, hence the corresponding eigenvalues $\theta_{A,1}$ and $\theta_{A,1}^{-1}$ remain simple (and they can be chosen to depend analytically on $A_1$). Furthermore, we can choose the corresponding eigenvectors $u_{A,1}^+$ and $u_{A,1}^-$ to (locally) depend analytically on $A_1$, hence in particular they depend Lipschitz continuously on $A$.  
The matrix $p_A$ whose column vectors are $u_{A,1}^+$ and $u_{A,1}^-$ diagonalizes $A_1$, hence it diagonalizes $A_1, \ldots, A_k$. 

We conclude that in~\eqref{diag formula} we can choose the matrix $p_A$ in such a way that $\Diagast \ni A\mapsto p_A$ is locally Lipschitz continuous, and in particular it is locally bounded. Combined with the fact that the function $\GLtwoR\ni g \mapsto g^{-1} \in \GLtwoR$ is locally Lipschitz continuous, we derive the same property for  $\Diagast \ni A\mapsto p_A^{-1}$. 

Finally, since $D_A = p_A^{-1} \, A \, p_A$, the map $\Diagast \ni A\mapsto D_A$ is also locally Lipschitz continuous, hence $\Diagast \ni A\mapsto \log \sabs{\theta_A}$ are locally Lipschitz continuous and bounded as well.

From the above it follows in particular that there is a locally bounded function $C \colon \Diagast\to (0, \infty)$ such that $\left( \log \norm{p_A} + \log \norm{p_A^{-1}} \right) \le C(A)$ and $\abs{\log \sabs{\theta_A} } \le C(A)$. 

Therefore, ~\eqref{diag bound} implies that for all $x\in X$, 
\begin{equation}\label{diag eq2}
\left| \frac{1}{n} \, \log \norm{\An{n} (x) }  - \Big| \frac{1}{n} \, \sum_{j=0}^{n-1} \log \sabs{\theta_A (T^j x) } \Big|  \right| \le \frac{C(A)}{n} \le \frac{\ep}{2} 
\end{equation}
if $n \ge \frac{2 C(A)}{\ep}$.

Furthermore, the i.i.d. random variables $\xi_{A,j} := \log \sabs{\theta_A \circ T^j}$ are bounded by $C(A)$, hence Hoeffding's inequality in Lemma~\ref{Hoeffding} implies 
\begin{equation}\label{diag eq3}
\Pp \left[ \, \left|   \frac{1}{n} \, \sum_{j=0}^{n-1} \log \sabs{\theta_A (T^j x) }   -  \EE \, \log \sabs{\theta_A}  \right| \ge \frac{\ep}{2} \right] \le  
 e^{- \frac{\ep^2}{8 C(A)^2} \, n} \, .
\end{equation}

The conclusion follows by combining \eqref{L diag}, \eqref{diag eq2} and \eqref{diag eq3}.
\end{proof}

\begin{remark}\label{continuity LE in Diag}
From the above considerations we also conclude that the restriction $L \colon \Diagast \to \R$ of the  Lyapunov exponent  is a locally Lipschitz continuous function.
That is because $L(A) =  \abs{ \EE \, \log \sabs{\theta_A} }$ and $A\mapsto \log \sabs{\theta_A}$ is locally Lipschitz. Furthermore, $L$ is continuous on $\Diag$, as continuity at $\SLtwoR$ cocycles with zero LE is automatic.
\end{remark}

\section{Irreducibility measurements}
\label{irredm}

\newcommand{\Vscr}{\mathscr{N}}
\newcommand{\Dscr}{\mathscr{D}}

In this section we introduce two {\em irreducibility measurements} $\rho(B)$ and $N(B)$ for a cocycle $B$,
which we then relate to the distance
$d(B,\Diag)$ of the cocycle $B$ to the space of diagonalizable cocycles.


\begin{lemma}\label{qirred def}
Let $A$ be an $SL_2(\R)$ valued cocycle with $L(A)>0$. The following statements are equivalent:
\begin{enumerate}[(i)]
\item $A$ is quasi-irreducible.
\item $\displaystyle \frac{1}{n} \, \log \norm{\An{n} v} \to L(A)$ \ $\Pp$-almost surely, for all $v\in\R^2$, $v\neq 0$.
\end{enumerate}
\end{lemma}

\begin{proof} The implication (ii)$\implies$(i) is obvious. If there is an $A$-invariant line $\ell$, picking $v\in\ell$, $v\neq 0$, the following holds $\Pp$-almost surely:
$L (\restr{A}{\ell}) = \lim_{n\to\infty} \, \frac{1}{n} \, \log \norm{\An{n} (x) v} = L (A)$,
so $L (\restr{A}{\ell}) = L(A)$.

Now we prove the implication (i)$\implies$(ii).  

Let $F\subset X$ be a $T$-invariant Borel set with full probability,
$\Pp (F)=1$, consisting of Oseledets regular points. Thus for all $x\in F$ we have the Oseledets decomposition 
$\R^2= \filt^+(x)\oplus \filt^-(x)$ into proper subspaces, which are invariant under the cocycle action, i.e., $A(x) \filt^\pm(x)=\filt^\pm(T x)$. Moreover, given any $x\in F$
and  $v\in\R^2$, $v\neq 0$, 
$$\lim_{n\to\infty} \, \frac{1}{n} \, \log \norm{\An{n} (x) v} = 
\begin{cases}
- L(A) & \text{ if } v \in \filt^-(x) ,\\
\ \  L(A) &   \text{ if } v \notin \filt^-(x) .
\end{cases}
$$
 
 On the other hand, by Theorem A in~\cite{FKi}, for every $v\in\R^2$, $v\neq 0$, there is a constant $\la_v$ and a set $F_v\subset X$ with $\Pp (F_v)=1$ such that for all $x\in F_v$,
 $$\lim_{n\to\infty} \, \frac{1}{n} \, \log \norm{\An{n} (x) v} = \la_v \, .$$
 
 Hence for all nonzero vectors $v$, $\la_v = L(A)$ or $\la_v = - L(A)$  and 
\begin{align*}
\text{either}  & \  \lim_{n\to\infty} \, \frac{1}{n} \, \log \norm{\An{n} (x) v} = \ L(A)   \quad \ \text{ for all } x \in F_v \, , \\
\text{or }  &  \ \lim_{n\to\infty} \, \frac{1}{n} \, \log \norm{\An{n} (x) v} = - L(A)   \quad \text{ for all } x \in F_v \, .
\end{align*}

Consider the set
\begin{align*}
S & :=\left\{ v\in \R^2\colon \, v\in \filt^-(x) \quad \Pp\text{-a.s.} \,\right\} \\
& = \Big\{ v\in \R^2\colon \, \lim_{n\to\infty} \, \frac{1}{n} \, \log \norm{\An{n} (x) v} = - L(A)  \quad \Pp\text{-a.s.} \, \Big\} \cup \{0\} \, .
\end{align*}

It is easy to see that $S$ is a linear subspace. We prove that $S$ is invariant under the cocycle. 
Because locally constant random cocycles over full Bernoulli shifts factor through the one-sided shift,  without loss of generality we may assume that  $T\colon X\to X$ is the lateral shift  with $X=\Sigma^\N$.

Given $i\in\Sigma$ consider the conditional probability measure, $\Pp_i(F):= \Pp( F\vert x_0=i)$,
defined for Borel sets $F\subset X$. By definition the family 
of probability measures $\{\Pp_i\}_{i\in\Sigma}$ is the desintegration of $\Pp$ through the canonical
projection $\pi:X\to \Sigma$, $\pi(x):=x_0$.
Because $\Pp$ is the product measure $p^\N$, it follows that $\Pp_i(F)=\Pp(F_i)$, where  for any Borel set $F\subset X$ we define
$F_i:=\{ x\in X\colon i x\in F\}$.

Now, given $v\in S$ with $v\neq 0$, consider the full probability  event
$F^v:=\{x\in X\colon v\in \filt^-(x) \}$.
If $i\in \Sigma$, since
$$1=\Pp(F^v)=\int_{\Sigma}  \Pp_j(F^v)\,dp(j)  = \int_{\Sigma} \Pp(F^v_j)\,dp(j)   $$
the event
$F^{v}_i =\{ x\in X\colon v\in \filt^-(ix)  \}$ has also probability one. By invariance of the Oseledets splitting, if
$v\in \filt^-(ix) $  then
$A_i v \in \filt^-(x) $, which implies that 
$A_i v\in \filt^-(x)$  $\Pp$-almost surely, thus
proving that $A_i\, v\in S$.
Therefore, $S$ is an $A$-invariant linear subspace of $\R^2$.

If $\dim S = 2$, then $\Pp$-almost surely, $\filt^- (x) = \R^2$, which contradicts the fact that $L(A)>0$. If $\dim (S)=1$ then $S$ is an $A$-invariant line and 
picking $v\in S$, $v\neq 0$, $\Pp$-almost surely we have
$$L (\restr{A}{S}) = \lim_{n\to\infty} \, \frac{1}{n} \, \log \norm{\An{n} (x) v} = - L (A) .$$
Thus $L (\restr{A}{S}) = - L(A)$, which contradicts the quasi-irreducibility of $A$. \\
It follows that $S=\{0\}$, so for every $v\neq 0$,
$\frac{1}{n} \log \norm{\An{n} (x) v} \to - L (A)$
does not happen $\Pp$-almost surely. But then, from the preceding dichotomy, $ \frac{1}{n} \, \log \norm{\An{n} (x) v} \to L (A)$ must happen $\Pp$-almost surely.
\end{proof}

\begin{lemma}\label{qirred unif conv}
Let $A \in \SL_2 (\R)^k$ be a cocycle. If $A$ is quasi-irreducible and $L(A)>0$, then
\begin{equation*}
\lim_{n\to\infty} \, \min_{ v \in \Su^1} \, \EE  \left[ \frac{1}{n} \, \log \norm{\An{n} v} \right]  = L (A) .
\end{equation*}
\end{lemma}

\begin{proof} Fix any vector $v\in\Su^1$.
Using the previous lemma, we have that
$$\frac{1}{n} \, \log \norm{\An{n} v} \to L(A) \quad \Pp\text{-almost surely} .$$
Since the functions $\frac{1}{n}\,\log \norm{A^{(n)} v}$ are uniformly bounded by $\log \norm{A}_\infty<\infty$, by the dominated convergence theorem,
$ \frac{1}{n}\,\EE \left[ \,
\log \norm{A^{(n)} \,v}
\,\right]$ converges to $L(A)$.

Assume now that this convergence is not uniform (in $v\in\Su^1$),
in order to get a contradiction. This assumption implies the existence of a sequence of unit vectors $v_n\in \R^2$ and a positive number $\delta>0$ such that for all large $n$,
\begin{equation}\label{contrad irred}
 \frac{1}{n}\,\EE\left[ \,
\log \norm{A^{(n)} \,v_n}
\,\right] \leq L(A)-\delta . 
\end{equation}

We claim now (and postpone the motivation for the end of the proof) that $\Pp$-almost surely,
\begin{equation}\label{liminf bound}
\liminf_{n\to\infty} 
\frac{\norm{ A^{(n)}\,v_n}}{\norm{ A^{(n)}}} >0 .
\end{equation}

Then $\Pp$-almost surely,
$ \frac{1}{n}\, \log 
\frac{\norm{ A^{(n)}\,v_n}}{\norm{ A^{(n)}}} \to 0$.
Therefore, using again the  dominated convergence theorem, we get
\begin{align*}
\lim_{n\to \infty} \frac{1}{n}\,\EE\left[ \,
\log \norm{ A^{(n)} \,v_n}
\,\right] &= \lim_{n\to \infty} \frac{1}{n}\,\EE\left[ \,
\log \norm{ A^{(n)} }
\,\right]\\
&\quad + \frac{1}{n}\,\EE\left[ \,\log 
\frac{\norm{ A^{(n)}\,v_n}}{\norm{ A^{(n)}}}\,\right]\\
& = L(A)+0 =L(A) \, ,
\end{align*}
which contradicts~\eqref{contrad irred}.

To complete the proof we verify the claim~\eqref{liminf bound}. The argument depends upon certain geometrical considerations related to our proof of the Oseledets multiplicative ergodic theorem in \cite[Chapter 4]{DK-book}.\footnote{This claim, in a more general setting, is also proven in~\cite[Proposition 2.8]{Bou}, using ingredients in the proof of the Oseledets multiplicative ergodic theorem given by Ledrappier~\cite{Ledrappier-MET}.}\\
We review below these considerations (see also Section 2.2 in ~\cite{DK-book}). 

Given a matrix $g\in\SL_2(\R)$, if $\norm{g} > 1$ then its singular values are distinct, so its singular directions (most and least expanding) are  well defined. The same holds also for the transpose matrix $g^*$. Then there are two orthonormal singular bases of $\R^2$, $\{ \mostexpuv_+ (g), \mostexpuv_- (g) \}$ and $\{ \mostexpuv_+ (g^*), \mostexpuv_- (g^*) \}$ such that
$$g \, \mostexpuv_+ (g) = \norm{g} \, \mostexpuv (g^*) \quad \text{and} \quad g \, \mostexpuv_- (g) = \norm{g}^{-1} \, \mostexpuv (g^*) \, .$$

If $w \in \R^2$ is any other vector, then using the Pythagorean's theorem, 
$$\frac{\norm{g \, w}}{\norm{g}} \ge \abs{ w \cdot  \mostexpuv_+ (g)} \, .$$

Denote also by $\medp_+ (g) \in \Pp (\R^2)$ the projective point corresponding to the unit vector $\mostexpuv_+ (g)$.
With these notations, given the cocycle $A$, 
we define the sequence of partial functions
 $\mostexp^{(n)}(A)\colon X\to\Pp(\R^2)$, $n\ge1$ by
$$ \mostexp^{(n)}(A)(x):=
\left\{ \begin{array}{lll}
\medp_+ (A^{(n)}(x)) & & \text{ if }  \norm{A^{(n)}(x)}>1 \\
\text{undefined} &  &  \text{otherwise}. 
\end{array} \right. $$

Since $L(A)>0$, by Proposition 4.4 in~\cite{DK-book}, this sequence converges
$\Pp$-almost surely to a (total) measurable function  
$\mostexp^{(\infty)}(A)\colon X\to\Pp(\R^2)$,
$$  \mostexp^{(\infty)}(A)(x):= \lim_{n\to +\infty} \mostexp^{(n)}(A)(x) . $$ 

As it turns out from our proof of the Oseledets theorem (see the beginning of the proof of Theorem 4.4 in~\cite{DK-book}), the Oseledets subspace corresponding to the second Lyapunov exponent $- L(A)$ is the orthogonal complement of the most expanding direction of the cocycle $A$, that is, $\Pp$-almost surely, $\filt^- (x) = \medir{\infty}(A) (x)^\perp$.

Claim~\eqref{liminf bound} now follows immediately. By the compactness of the unit circle we can assume that
the sequence $\{v_n\}$ converges to a unit vector $v\in\R^2$.
Then for $\Pp$-almost every $x$, and for $n$ large enough,
$$ \frac{\norm{A^{(n)} (x) \,v_n}}{\norm{A^{(n)} (x)}} \geq  \abs{ v_n \cdot \mostexpuv_+ \left( \An{n} (x) \right) }
\to  \abs{ v \cdot \medir{\infty}(A) (x)} .$$

But if $v\in\Su^1$, then 
$ \displaystyle v \, \cdot \, \medir{\infty}(A) (x) >0$, since otherwise we would have $v \in \medir{\infty}(A) (x)^\perp=\filt^-(x)$, which happens with probability zero. 
\end{proof}


\begin{definition}
\label{def N(A)}
Let $A$ be a  quasi-irreducible cocycle  with $L(A)>0$.
We define $N (A)$ to be the least $n \in \N$ such that for all    $v \in \Su^1$ one has
$$\EE  \left[ \frac{1}{n} \, \log \norm{\An{n} v} \right]  > \frac{1}{2} \, L (A) \, .$$
\end{definition}

Next we show that if $A$ is instead a diagonalizable cocycle, then Lemma~\ref{qirred unif conv} does not hold, and as $B\to A$, where $B$ is quasi-irreducible, the measurement $N (B) \to \infty$.

\begin{proposition}\label{N goes to infinity}
Let $A$ be a diagonalizable cocycle with $L(A)>0$. 
\begin{enumerate}
\item There is $v_A\in\Su^1$ such that for all $n\ge1$,
\begin{equation*}
 \EE  \left( \frac{1}{n} \, \log \norm{\An{n} v_A} \right)  = - L(A)  \, .
\end{equation*}
\item There are a constant $c = c (A)>0$ and a neighborhood $\Vscr_A$ of $A$ such that if $B \in \Vscr_A$ is a quasi-irreducible cocycle with $L(B)>0$, then
\begin{equation*}\label{N to infty}
N (B) \ge c \, \log \left( d (B, \Diag)^{-1} \right) \, .
\end{equation*}

In particular,
$N (B) \ge c \, \log \left( d (B, A)^{-1} \right)  \to 0$ as the neighborhood $\Vscr_A$ shrinks. 
\end{enumerate}
\end{proposition}

\begin{proof} Following the same approach used in the proof of Theorem~\ref{ldt diag}, there are unit vectors $u^+, u^-$ and there is a locally constant function $\theta \colon X \to \R\setminus \{0\}$ such that for all $x\in X$,
$$A (x) \, u^+ = \theta (x) \, u^+ \quad \text{and} \quad  A (x) \, u^- = \theta (x)^{-1} \, u^- \, .$$ 

Then for all $n\ge1$,
\begin{align*}
\frac{1}{n} \, \log \norm{\An{n} (x) \, u^+} & = \frac{1}{n} \, \sum_{j=0}^{n-1} \log \, \abs{\theta (T^j x) } \ \text{and}\\
\frac{1}{n} \, \log \norm{\An{n} (x) \, u^-} & = - \frac{1}{n} \, \sum_{j=0}^{n-1} \log \, \abs{\theta (T^j x) } .
\end{align*}

By the Oseledets theorem and the law of large numbers, as $n\to\infty$, one of the two quantities above converges to $L(A)= \abs{ \EE \left( \log \sabs{\theta} \right)}$ and the other to $- L(A)$. With no loss of generality, we may assume that it is the former quantity that converges to $- L(A)$, and so put $v_A := u^+$. 

By the invariance of the measure, for all $n\ge1$,
$$\EE \left( \frac{1}{n} \, \sum_{j=0}^{n-1} \log \, \abs{\theta \circ T^j  } \right) = \EE \left( \log \sabs{\theta} \right) \, .$$

We conclude that for all $n\ge 1$ we have
$$\EE \left( \frac{1}{n} \, \log \norm{\An{n}  \, v_A} \right) =  \EE \left( \log \sabs{\theta} \right) = - L(A) \,,$$
which completes the proof of the first item.

To prove the second item, we fix a neighborhood $\Vscr_A$ of $A$ such that, by Remark~\ref{continuity LE in Diag}, we have that $L(D) > \frac{1}{2} L(A)$ for all diagonalizable cocycles $D\in\Vscr_A$. 
Let $1<L<\infty$ be such that for all $B\in\Vscr_A$,
$$\max_{x\in X} \, \max \left\{ \norm{B(x)}, \norm{B(x)^{-1}} \right\} < L$$
and put $K := \log L$.

We decrease, if necessary, the neighborhood $\Vscr_A$, to ensure that, if $\delta$ denotes its diameter, then 
$\delta <  \left(L(A)/2\right)^2$.

Fix a quasi-irreducible cocycle $B\in\Vscr_A$, with $L(B)>0$. Let $D$ be {\em any} diagonalizable cocycle in $\Vscr_A$.

Then $\Pp$-almost surely and for every unit vector $v$ and for all $n\ge 1$, by the mean value theorem we have
\begin{align*}
\left| \frac{1}{n} \, \log \norm{\Bn{n} (x) \, v} -  \frac{1}{n} \, \log \norm{D^{(n)} (x) \, v}    \right| & \le
\frac{1}{n} \, \frac{ \left| \norm{\Bn{n} (x) \, v }  -  \norm{D^{(n)} (x) \, v} \right| }{ \min \{ \norm{\Bn{n} (x) \, v}, \norm{D^{(n)} (x) \, v}   \} }  \, . 
\end{align*}

Moreover,
\begin{align*}
 \norm{\Bn{n} (x) \, v -  D^{(n)} (x) \, v }  & \\
&   \kern-8em  \le \sum_{j=0}^{n-1} \, \norm{B (T^{n-1} x)} \ldots  \norm{ B (T^j x) - D (T^j x)}  \ldots \norm{D (x)} \\ 
&   \kern-8em < n \,  L^n \, d (B, D) \, ,
\end{align*}
while 
\begin{align*}
\norm{\Bn{n} (x) \, v } & \ge \norm{\Bn{n} (x)^{-1}}^{-1} \ge L^{-n} \quad  \text{and} \\
 \norm{D^{(n)} (x) \, v} & \ge \norm{\Dn{n} (x)^{-1}}^{-1} \ge L^{-n} \, .
\end{align*}

Therefore,
$$\left| \frac{1}{n} \log \norm{\Bn{n} (x) \, v} -  \frac{1}{n} \log \norm{D^{(n)} (x) \, v}  \right| \le L^{2 n} \, d (B, D) = e^{2 K n} \, d (B, D) \, .$$

Choosing $v=v_D$, taking expectations and using item 1, we have that for all $n\ge 1$,
\begin{align*}
\EE \left( \frac{1}{n} \, \norm{\Bn{n}  \, v_D } \right) & \le \EE \left( \frac{1}{n} \, \norm{\Dn{n} \, v_D } \right)  + e^{2 K n} \, d (B, D) \\
& = - L(D)  + e^{2 K n} \, d (B, D) \le 0 \, ,
\end{align*}
provided 
$$e^{2 K n} < d (B, D)^{-1/2} < \frac{L(A)}{2} \, d (B, D)^{-1} \, .$$

Therefore, we conclude that for all integers $n \le \frac{1}{4 K} \, \log d (B, D)^{-1}$,
$$\EE \left( \frac{1}{n} \, \norm{\Bn{n}\, v_D } \right) \le 0 < \frac{L(B)}{2} \, .$$

This shows that $N(B) >  \frac{1}{4 K} \, \log d (B, D)^{-1}$. But since $D$ was an arbitrary diagonalizable cocycle in a neighborhood of $A$, we conclude:
$$N (B) \ge \frac{1}{4 K} \, \log \left( d (B, \Diag)^{-1} \right) \, ,$$
which completes the proof.
\end{proof}

\bigskip

Given $A=(A_1,\ldots, A_k) \in \Diagast$, we denote by $\hat e_\pm (A)\in\Pp$ the points corresponding to the invariant lines where the Lyapunov exponent of $A$ is positive, respectively negative.
Moreover, $e_\pm (A)$ will stand for unit vectors
representing these projective points $\hat e_\pm(A)$.
Because $A\in\Diag$ there is a family of non-zero numbers $\{a_j\colon j=1,\ldots, k\}$  such that 
$A_j \, e_+(A)= a_j\, e_+(A)$ and 
$A_j \, e_-(A)= a_j^{-1}\, e_-(A)$, for all $j=1,\ldots, k$.

Then we define
$$\Sigma_H(A):=\left\{ i\in \{1,\ldots, k\}\,\colon \, \sabs{a_i} \geq e^{L(A)} \, \right\} .$$
This set is non-empty because
$L(A)=\sum_{j=1}^k p_j\,\log \sabs{a_j}>0$.

Likewise, given a hyperbolic matrix $g\in\SL_2(\R)$,
we denote by $\hat e_\pm(g)\in\Pp$ the eigen-directions
of $g$  associated to eigenvalues of $g$ with absolute value larger than $1$, respectively less than $1$.

Given $A\in \Diagast$, we define the projective cones\footnote{
With this definition,
$\overline{\Dscr_\pm(\delta)}=\Dscr_\mp(\delta^{-1})\comp$,\, for all $\delta>0$.}
\begin{align}
\label{Dscr def}
\Dscr_\pm (\delta) =\Dscr_\pm(A, \delta) :=\{ \hat x\in\Pp\colon
d(\hat x, \hat e_\pm(A)) <\delta\, d(\hat x, \hat e_\mp(A)) \,\} 
\end{align}
where the radius $\delta=\delta(A)>0$ is chosen sufficiently small so that
\begin{align}
\label{Dscr radius}
 \max_{1\leq i \leq k} \sup_{\hat v, \hat v'\in \Dscr_\pm(2\delta) } \abs{\log \norm{A_i\,v} - \log \norm{A_i\,v'}}
<\frac{1}{20}\, L(A) .
\end{align}

Take $L<\infty$ such that $\max\{ \norm{A}_\infty,\norm{A^{-1}}_\infty\}<L$. Next we fix a neighborhood $\Vscr_A:=\{B\in\cocycles \colon \norm{A-B}_\infty<\varepsilon \}$ of $A$  
in the space of cocycles $\cocycles=\SL_2(\R)^k$  such that for all $B\in\Vscr_A$:
\begin{enumerate} 
\item[N0:]  
\begin{enumerate}
	\item[(a)]  $\displaystyle  d(A,B)<  \frac{\delta}{L}\, e^{-\frac{L(A)}{4}}\left( e^{\frac{L(A)}{12}}-1 \right) $,
	\item[(b)]  $\displaystyle  d(A,B)< \frac{L^{-1}\, \sqrt{ e^{\frac{L(A)}{12}}-1}}{1+ \sqrt{e^{\frac{L(A)}{12}}-1}} $,
	\item[(c)]  $\displaystyle  d(A,B)< \delta\,L^{-1}\, {e^{-\frac{L(A)}{8}}} \,\left( 1-e^{-\frac{L(A)}{24}} \right) . $ 
\end{enumerate}

\item[N1:]  For all  $i\in \Sigma_H(A)$, the matrix $B_i$ is hyperbolic. This holds because
$B_i\approx A_i$ for all $i$ and $A_i$ is hyperbolic for $i\in \Sigma_H(A)$.
In particular the directions
$\hat e_\pm(B_i)$ are well-defined for all 
$B\in\Vscr_A$ and $i\in\Sigma_H(A)$.
\item[N2:]  $\frac{4}{3}\, L(A) > L(B)$ for all $B\in\Vscr_A$. This holds because the LE is an upper-semicontinuous function.
\item[N3:] $\norm{B}_\infty <  L$ and
$\norm{B^{-1}}_\infty <  L$.
\item[N4:] For all  $i\in \Sigma_H(A)$,  $d( \hat e_+(B_i), \hat e_-(B_i)) > \frac{1}{2}\, d( \hat e_+(A), \hat e_-(A))=:c$.
\item[N5:] 
$\displaystyle \max_{1\leq i \leq k} \sup_{\hat v, \hat v'\in \Dscr_\pm(A,2\delta)  } \abs{\log \norm{A_i\,v} - \log \norm{B_i\,v'}}
<\frac{1}{12}\, L(A) .$  This holds provided $\norm{A-B}_\infty$ is smaller than a certain fraction of $L(A)$,  because of~\eqref{Dscr radius} and the following inequality
$$ \abs{\log \norm{A_i\,v'} -\log\norm{B_i\,v'} }\leq L\, \norm{A_i-B_i}\,\norm{v'} , $$
which follows from  N3.
\end{enumerate}

Other constraints on $\Vscr_A$ will be imposed below. The statements of this section and the next refer to the neighborhood $\Vscr_A$.

Given $g\in\SL_2(\R)$ we 
denote by $\hat g\colon \Pp\to\Pp$ the projective
action of this matrix,  $\hat g \hat x := \widehat{g x}$.

Then for all $B\in\Vscr_A$ we define
\begin{align*}
 \rho_+(B) &:= \max_{i\in \Sigma_H(A)}
\max_{1\leq j\leq k} d( \hat B_j \hat e_+(B_i), \hat e_+(B_i)) \\
 \rho_-(B) &:= \max_{i\in \Sigma_H(A)}
\max_{1\leq j\leq k} d( \hat B_j \hat e_-(B_i), \hat e_-(B_i)) \\
\rho(B) &:=\max\{ \rho_-(B), \rho_+(B)\}.
\end{align*}

\begin{lemma} For all $B\in \Vscr_A$,  $\rho(B)=0$  if and only if  $B$ is diagonali\-zable. 
Moreover,
$\rho_+(B^{-1})\asymp \rho_-(B)$ and $\rho_-(B^{-1})\asymp \rho_+(B)$.
\end{lemma}

\begin{proof}
If $\rho(B)=0$ then for some $i_0\in \Sigma_H(A)$
and all $j=1,\ldots, k$, 
$\hat B_j \hat e_+(B_{i_0}) = \hat e_+(B_{i_0})$ and
$\hat B_j \hat e_-(B_{i_0}) = \hat e_-(B_{i_0})$, which implies that $B$ is diagonalizable. Conversely, if the matrices $B_j$ are simultaneously diagonalizable there exists a pair of distinct projective points $\hat e_1$ and $\hat e_2$ which are fixed under
the action of all the matrices $B_j$. Taking $i_0\in \Sigma_H(A)$,
since the matrix $B_{i_0}$ is hyperbolic these two points must coincide with $\hat e_+(B_{i_0})$ and $\hat e_-(B_{i_0})$.
Hence $\rho(B)=\rho_-(B)=\rho_+(B)=0$.

By N3, 
 $\norm{B_j}^2\norm{B_j^{-1}}^2 \leq  L^4$ for all $B\in\Vscr_A$, 
 and hence by Remark~\ref{d(g x, g y)=L^2 d(x,y)} (see also Lemma 2.11 in~\cite{DK-book}) for every projective point $\hat x$
  $$ \frac{1}{L^4}\leq  \frac{d(  \hat B_j \hat x, \hat x)}{d(\hat x, \hat B_j^{-1} \hat x)}\leq  L^4 . $$
  Applying these inequalities to the points 
  $\hat x= \hat e_{\pm}(B_i)$ with $i\in \Sigma_H(A)$ we get
  $\rho_+(B^{-1})\leq L^4\, \rho_-(B)$, \,
  $\rho_-(B)\leq L^4\, \rho_+(B^{-1})$, \,
  $\rho_+(B)\leq L^4\, \rho_-(B^{-1})$ and
  $\rho_-(B^{-1})\leq L^4\, \rho_+(B)$, which completes the proof.
\end{proof}

\smallskip

\begin{proposition}\label{diag rho} 
\label{prop d(B, Diag) < rho(B)}
Given $A\in\Diag$ with $L(A)>0$,
there exists a constant $C=C(A)<\infty$ such that
for all $B \in\Vscr_A$,
$$ d(B,\Diag) \leq C\, \rho(B) \quad \text{ and }\quad
\rho(B)\leq C\,d(B,A) .$$
In particular,   for every $B\in\Vscr_A$, 
$$\text{either} \quad d(B,\Diag) \le C\, \rho_-(B) \quad \text{or} \quad d(B,\Diag) \le C\, \rho_-(B^{-1}) \,.$$
\end{proposition}

\begin{proof}
Let $L$ and $c$ be the constants definded in assumptions N3-N4 and take the constant  $C=C(L,c)$ provided by Proposition~\ref{proximity prop} below. For all $B\in\Vscr_A$ and $i\in \Sigma_H(A)$,
$d( \hat e_+(B_i), \hat e_-(B_i))> c$ and  $\norm{B}_\infty< L$. Choose the sign $\ep= \pm$ so that $\rho(B)=\rho_\ep(B)$,
and take the maximizing index  $i_0\in\Sigma_H(A)$ such that for some
$j=1,\ldots, k$,
$$ \rho (B) =  d( \hat B_{j} \hat e_\ep (B_{i_0}), \hat e_\ep(B_{i_0}))  .$$
Let $\hat p^+ := \hat e_+( B_{i_0})$ and  $\hat p^- := \hat e_-( B_{i_0})$. By Proposition~\ref{proximity prop}
for every $i=1,\ldots, k$ there exists $B_i^\ast\in \SL_2(\R)$ such that
$\hat B_i^\ast \hat p^+ = \hat p^+$, $\hat B_i^\ast \hat p^- = \hat p^-$   and
$\norm{B_i^\ast-B_i}\leq C\, \rho(B)$.
Hence $B^\ast=(B_1^\ast, \ldots, B_k^\ast)\in \cocycles=\SL_2(\R)^k$ is a diagonalizable cocycle such that
$$ d(B,\Diag) \leq \norm{B-B^\ast}_\infty \leq C\, \rho(B). $$
Finally we need to prove that $\rho(B)\lesssim d(A,B)$.
For $i\in \Sigma_H$, since $A_i$ is hyperbolic its eigenvalues are simple and  the function $B_i\mapsto \hat e_\pm(B_i)$ is analytic on a neighborhood of $A_i$. Hence there exists $C=C(A)<\infty$ such that 
$d(\hat e_\pm(B_i), \hat e_\pm(A_i))\leq C\, d(A,B)$, for all $i\in\Sigma_H$ and $B\in\Vscr_A$.
Using that the projective action of any matrix $A_j$ has Lipschitz constant $\mathrm{Lip}(\hat A_j)\leq \norm{A_j}^2$, we obtain for all 
$i\in\Sigma_H$ and $1\leq j\leq k$,
\begin{align*}
d(\hat B_j \hat e_\pm(B_i), \hat e_\pm(B_i))  \leq \, & \
d(\hat B_j \hat e_\pm(B_i), \hat A_j \hat e_\pm(B_i)) 
+
d(\hat A_j \hat e_\pm(B_i),  \hat e_\pm(A_i)) \\
& +  d( \hat e_\pm(A_i),  \hat e_\pm(B_i)) \\
\leq \, &
\norm{ B_j   e_\pm(B_i) -  A_j   e_\pm(B_i)} 
+
d(\hat A_j \hat e_\pm(B_i),  \hat A_j \hat e_\pm(A_i)) \\
& +  d( \hat e_\pm(A_i),  \hat e_\pm(B_i)) \\
\leq \, &
\norm{ B_j - A_j  } 
+ (\norm{A_j}^2+1)\,  d( \hat e_\pm(A_i),  \hat e_\pm(B_i)) \\
\leq \, & (1+(L^2+1)\, C)\,d(A,B) .
\end{align*}
Taking the maximum in $i\in\Sigma_H$ and $1\leq j\leq k$, this inequality  implies that
$\rho(B)\lesssim d(A,B) $.
\end{proof}

\begin{proposition}
\label{proximity prop}
Given $c>0$ and $1\leq L<\infty$ there exists $C=C(L,c)<\infty$  with the following property:

For any $g\in\SL_2(\R)$ and any
two points $\hat p^\pm \in\Pp$ such that $\norm{g}\leq L$ and  $d(\hat p^+, \hat p^-)\geq c$ there exists  $g^\ast \in\SL_2(\R)$ such that
$\hat g^\ast\,\hat p^+ =\hat p^+ $,
$\hat g^\ast\,\hat p^- =\hat p^- $   and
$$ \norm{g^\ast -g}\leq C\,\max\{d(\hat g\, \hat p^+, \hat p^+), 
d(\hat g \, \hat p^-, \hat p^-) \} .$$
\end{proposition}

\begin{proof}
Take unit vectors $p^\pm\in \hat p^\pm$ such that 
$\angle(p^+, p^-)$ is not obtuse.
Let $m\in\GL_2(\R)$ be the matrix with columns $g\,p^+/\norm{g\, p^+}$ and $g\,p^-/\norm{g\, p^-}$. Similarly, let $m'\in\GL_2(\R)$ be the matrix with columns $p^+$ and $p^-$. 
Consider the matrix $g'=m'\,m^{-1}\,g\in \GL_2(\R)$.
By construction this matrix satisfies
$	\hat g'\,\hat p^+ =\hat p^+ $ and
$\hat g'\,\hat p^- =\hat p^- $.

Next notice that
\begin{align*}
\norm{g'-g} &=  \norm{ (m'-m)\, m^{-1}\,g}
\leq \norm{ m'-m}\,\norm{ m^{-1}}\, \norm{g} .
\end{align*}
Since $\norm{g}\leq L$, we now need to derive bounds for $\norm{m'-m}$
and $\norm{m^{-1}}$.

For any matrix $a\in \Mat_2(\R)$ and in particular
for $a=m'-m$, 
$$ \norm{a}\leq 2\, \max\{\norm{a \,e_1},\norm{a \, e_2}\}  
$$
where $\{e_1,e_2\}$ stands for the canonical basis of $\R^2$. On the other hand, given $\hat x, \hat y\in \Pp$,
if we choose the unit vectors $x\in \hat x$ and $y\in \hat y$ to form a non obtuse angle then
$$ \norm{x-y} \leq \sqrt{2}\, d(\hat x, \hat y). $$
Combining these two facts we get
$$ \norm{m'-m}\leq 2\,\sqrt{2}\, \max\{d(\hat g\, \hat p^+, \hat p^+), 
d(\hat g \, \hat p^-, \hat p^-) \} .$$
For any matrix $g\in \SL_2(\R)$,
the projective map $\hat g\colon \Pp\to\Pp$ has
Lipschitz constant ${\rm Lip}(\hat g)\leq \norm{g}^2$ 
and the same is true about $g^{-1}$ because
$\norm{g^{-1}} = \norm{g}$. Hence, since in our case $\norm{g}\leq L$,
\begin{align*}
\sabs{ \det m } &= \frac{\norm{g\, p^+ \wedge g\, p^-}}{\norm{g\, p^+}\, \norm{g\, p^-} } = d(\hat g\, \hat p^+, \hat g\, \hat p^- )\geq \frac{1}{L^2}\, d(\hat p^+, \hat p^-) \geq \frac{c}{L^2}.
\end{align*}
Therefore 
$$\norm{m^{-1}}\leq \frac{\norm{m}}{\sabs{\det m}}\leq 
\frac{2}{\sabs{\det m}}\leq
\frac{2\, L^2}{c} .$$
Together these bounds imply that
$$   \norm{g'-g}\leq   4\,\sqrt{2}\, L^3\,c^{-1}\, \max\{d(\hat g\, \hat p^+, \hat p^+), 
d(\hat g \, \hat p^-, \hat p^-) \} .$$
Finally, by Lemma~\ref{g*g' lemma} below
we can replace  $g'\in \GL_2(\R)$ by
  $g^\ast:=g_\ast'\in\SL_2(\R)$ and still have 
$$   \norm{g^\ast-g}\lesssim   \max\{d(\hat g\, \hat p^+, \hat p^+), 
d(\hat g \, \hat p^-, \hat p^-) \} .$$
We finish explaining why this lemma can be applied.
First notice that the previous estimate on $\sabs{\det m}$ gives
$$ \frac{c}{L^2}\leq  \sabs{\det m}=d(\hat g \, \hat p^-, \hat g\, \hat p^-) \leq 1 .$$
Constructing a smooth deformation $m_t$ from $m_0=m'$ to $m_1=m$ by matrices with unitary columns we have
$$ \norm{m_t}\leq 2 \quad \text{ and } \quad
\frac{c}{L^2}\leq \sabs{\det m_t}\leq 1 .$$
Hence $g_t:= m'\, m_t^{-1}\, g$ is a smooth deformation
from $g$ to $g'$ such that $\norm{g_t}\leq 4\, L^3\, c^{-1}$ and 
$$\sabs{\det g_t}=\frac{\sabs{\det m'}}{\sabs{\det m_t}}
\geq d(\hat p^+, \hat p^-) \geq c . $$
This shows that Lemma~\ref{g*g' lemma}  can be applied with an appropriate new constant $\tilde L=\tilde L(L,c)<\infty$.
\end{proof}

\begin{lemma}
\label{g*g' lemma}
Given $1\leq L<\infty$ there exist $C=C(L)<\infty$  such that for all $g\in \SL_2(\R)$ with $\norm{g}< L$ and all
$g'\in \GL_2(\R)$ with $\norm{g'}<L$, $\sabs{\det g'}\geq L^{-1}$ and such that these bounds for $g'$ hold along a smooth deformation from $g$ to $g'$,  then 
$g_\ast' := \sabs{\det g'}^{-\frac{1}{2}}\, g'\in\SL_2(\R)$ satisfies
$$ \norm{g-g_\ast'}\leq C\, \norm{g-g'}. $$
\end{lemma}

\begin{proof}
By Jacobi's formula for the derivative of the determinant,
$$ (D \det)_g(h)= \tr( h \,g^{-1})\, \det g,\qquad
h\in \Mat_2(\R) . $$
Hence, for all $g\in\GL_2(\R)$,
$\norm{(D \det)_g}\leq 2\,\norm{g^{-1}}\,\abs{\det g}=2\,\norm{g}$. Using the assumptions and  the Mean Value Theorem we get
$$ \abs{\sabs{\det g}^{-1/2} - \sabs{\det g'}^{-1/2} }  \leq L^{5/2}\, \norm{g-g'} . $$
 Therefore
\begin{align*}
 \norm{g-g_\ast'} &\leq
 \norm{g -  \sabs{\det g'}^{-\frac{1}{2}}\, g}
 + \sabs{\det g'}^{-\frac{1}{2}}\,\norm{g-g'} \\
 &\leq
 \abs{\sabs{\det g}^{-\frac{1}{2}} -  \sabs{\det g'}^{-\frac{1}{2}}}\,\norm{g} 
 + \sabs{\det g'}^{-\frac{1}{2}}\,\norm{g-g'} \\
 &\leq (L^{5/2}\, L + L^{1/2})\, \norm{g-g'}
\end{align*}
which proves this lemma with $C(L):= L^{7/2}+L^{1/2}$.
\end{proof}

 The next proposition establishes a relation between the
 quasi-irreducibility measurements  $\rho(B)$ and  $N(B)$.
 We postpone its proof to the next section.

\begin{proposition}
\label{N-rho comparison proposition}
Given $A\in\Diag$ with $L(A)>0$ there is a constant 
$c=c(A)>0$ such that if the neighborhood $\Vscr_A$ of $A$
is small enough then 
for all $B\in\Vscr_A$,
$$ \rho_-(B) \leq e^{-c\,N(B) }\; \text{ and }\;
\rho_+(B) \leq e^{-c\,N(B^{-1}) }  . $$
In other words,
$$ N(B)\leq    c^{-1}\, \log ( \rho_-(B)^{-1} )\;
\text{ and } \;
 N(B^{-1})\leq   c^{-1} \, \log  (\rho_+(B)^{-1})   . $$
\end{proposition}

\begin{remark}
Combining Propositions~\ref{N goes to infinity},~\ref{diag rho} and~\ref{N-rho comparison proposition}, we conclude that all three irreducibility measurements introduced in this section, namely $N(B)$, $\rho (B)$ and $d (B, \Diag)$ are essentially equivalent. 

More precisely, let us put $N(A) := \infty$ if $A$ is diagonalizable (which,  by Proposition~\ref{N goes to infinity}, is a reasonable convention). Moreover, for any cocycle $B$, define 
$\Nbar (B) := \max \{ N(B), N (B^{-1}) \}$.

Then for any diagonalizable cocycle $A$ with $L(A)>0$, there are constants $c_1, c_2, c_3 >0$ and a neighborhood $\Vscr_A$ of $A$ such that for all $B\in\Vscr_A$, the following hold:
$$c_1 \, \Nbar (B) \le \log \rho (B)^{-1} \le c_2 \, \log d (B, \Diag)^{-1}  \le c_3 \, \Nbar (B) \, .$$
\end{remark}

\section{Prison break}
\label{prison}

Any  cocycle $A\in \cocycles=\SL_2(\R)^k$
induces a random walk on the projective space.
If the cocycle $A$ is diagonalizable with $L(A)>0$
then the random walk starting at a point $\hat x\neq \hat e_-(A)$
converges rapidly to $\hat e_+(A)$. This is not the case for  $\hat x= \hat e_-(A)$ where the random walk gets permanently stuck.
For a quasi-irreducible cocycle $B\in\cocycles$ very close to $A$ the corresponding random walk may also wander for a while near the
point $\hat e_-(A)$ after which it  rapidly moves to somewhere near $\hat e_+(A)$. It is natural to ask for how long one needs to wait until it becomes likely
to see the random walk moving in a neighborhood of $\hat e_+(A)$. The answer  depends of course on the distance from $B$ to $A$. Proposition~\ref{prison break main}
makes  a precise quantitative statement which answers this question. 

\bigskip

The next subsection provides an abstract scheme for the proof
of Proposition~\ref{prison break main}. We illustrate it with a {\em prison break metaphor}.

Imagine a prisoner whose movement in the world $\Sigma$ is modeled by a random  walk $\xi_n$. 
The confinement constraints on the prisoner's movements are
encoded in the transition probabilities of this random walk. Let $\Pz$ be the prisoner's cell,
$\Po$ be the prison area and $\Pt$ be the state where the prisoner serves his sentence.

Assume the following about the prisoner.
\begin{enumerate}
\item[(1)] The probability that he escapes from the cell $\Sigma_0$ within a time $n_0$ is small  but positive.
\item[(2)] He has a large  probability of evading the prison $\Po$ within a time $n_1$ once he is outside the cell. 
\item[(3)] The probability of him fleeing the state $\Pt$ within a time $n_2$ is again large
once he is out of prison.     
\item[(4)]  The chances of him ever being caught again in the prison $\Po$ are very slim after he  gets abroad. 
\end{enumerate}
From these assumptions one concludes that after a long enough time, of the form $q\, n_0+n_1+n_2$  for some possibly large $q\in\N$, the prisoner will very likely stay  permanently out of the jail $\Po$. 

The next subsection   quantifies this statement
in Proposition~\ref{abstract prison break}.

\newcommand{\rhoB}{\rho_B}
\newcommand{\Prob}{{\rm Prob}}
\newcommand{\Mscr}{\mathscr{M}}

\subsection*{Abstract scheme}

Let  $K$ be a stochastic kernel on a compact space of symbols $\Sigma$ and denote by $\FF$ the Borel $\sigma$-algebra of $\Sigma$. Let $\Prob(\Sigma)$ denote the space of 
probability measures on $(\Sigma,\FF)$.
Let  $\{\xi_n\}_{n\geq 0}$ be the process on $X=\Sigma^\N$  
defined by $\xi_n\{x_j\}_{j\in\N}:= x_n$.

Given  $\nu\in\Prob(\Sigma)$  there exists a unique probability measure $\Pp_\nu$  on $X=\Sigma^\N$ such that   $\{\xi_n\}_{n\geq 0}$  is a Markov
process with transition kernel $K$ and initial distribution $\nu$, in the sense that for all $\measA\in\FF$ and $x\in \Sigma$ 
\begin{enumerate}
\item $\Pp_\nu[\, \xi_n\in \measA\,\vert \, \xi_{n-1}=x \, ]= K_x(\measA)$ \, ( $n\geq 1$),
\item $\Pp_\nu[ \, \xi_0\in \measA \, ]= \nu(\measA)$.
\end{enumerate} 
To emphasize the measure $\Pp_\nu$
underlying  the process $\xi_n$ 
we will write $\xi^\nu_n$, or simply $\xi^x_n$
when $\nu=\delta_x$, instead of $\xi_n$.
We will also write $\Pp_x$ instead of $\Pp_{\delta_x}$.

Given $x\in\Sigma$ and $\measA \in\FF$ with $x\in \measA$ define
the {\em probability of escaping from $\measA$ in time  $n$}
$$ \pcal_n(x,\measA) := \Pp_x[ \text{ for some } \, 0\leq j\leq n,\; \xi^x_j\notin \measA \,]   $$
as well as the complementary {\em probability of remaining in $\measA$ for time $n$}
$$ \pcal_n^\ast(x,\measA):= \Pp_x[ \text{ for all } \, 0\leq j\leq n,\; \xi^x_j\in \measA \,] . $$

Note that  
$$  \pcal_n(x,\measA)+  \pcal_n^\ast(x,\measA)   = 1 . $$

\begin{proposition}
\label{prop pcal multip}
Given $x_0\in \measA\subset \measB\subset \Sigma$ in $\FF$,
$$ \pcal_{n+m}(x_0, \measB)\geq   \pcal_m(x_0,\measA)\,  \left(\inf_{x\in \Sigma
\setminus \measA} \pcal_n(x,\measB) \right)\;. $$
\end{proposition}

\begin{proof}
Define a stopping time $\tilde{\j}\colon X\to \N\cup\{\infty\}$,
$$ \tilde{\j} := \min \{ l\geq 0\colon \xi^{x_0}_l\notin \measA\}$$
together with a random variable
$\tilde{x}\colon X\to\Sigma$ such that
$\tilde{x} = \xi^{x_0}_{\tilde{\j}}$ whenever $\tilde \j<\infty$. By construction
$\tilde{x}$ takes values in $\Sigma\setminus \measA$ if
$\tilde \j<\infty$.

\begin{align*}
\pcal_{n+m}(x_0, \measB)&\geq  \Pp_{x_0}\left[ \tilde{\j}\leq m \, \text{ and for some }\, 0\leq j\leq n,\; \xi^{\tilde{x}}_j\notin \measB \right]\\
&= \sum_{l=0}^m  \Pp_{x_0}\left[ \text{ for some }\, 0\leq j\leq n,\; \xi^{\tilde{x}}_j\notin \measB \, \vert\, \tilde{\j}=l \right]\, \Pp_{x_0}[\tilde{\j}=l] \\
&\geq \left(\inf_{x\in \Sigma
\setminus \measA} \pcal_n(x,\measB) \right)  \sum_{l=0}^m  \Pp_{x_0}[\tilde{\j}=l] \\
&=   \left(\inf_{x\in \Sigma
\setminus \measA} \pcal_n(x,\measB) \right)\, \pcal_m(x_0,\measA)\;. 
\end{align*}
\end{proof}

Given $\measA\in\FF$ we also define
$$ \pcal_n^\ast(\measA):=\sup_{x\in \measA} \pcal_n^\ast(x,\measA).$$

\begin{proposition}
\label{P* sub-multip}
$\pcal_{n+m}^\ast(\measA)\leq 
\pcal_{n}^\ast(\measA)\,
\pcal_{m}^\ast(\measA)$ for all $n,m\geq 0$.
\end{proposition}

\begin{proof}
Let  $X_n= \Sigma^{n+1}$ and denote by
$\pi_n\colon X\to X_n$ the canonical  projection
$\pi_n\{x_j\}_{j\in\N}=(x_0,x_1,\ldots, x_n)$.
Let $\Pp_x^{(n)}:=(\pi_n)_\ast \Pp_x$ and note that
$\Pp_x^{(n)}(\{x\}\times E^n)=\pcal_n^\ast(x,E)$.

%

Define the family of events
$$  E_{n,m}^x:=\left[ \, \xi_j^x\in E, \quad \forall\, n+1\leq j\leq n+m  \,   \right] . $$

Using the law of total probability and then the fact that $\{\xi_n^x\}$ is a Markov chain, we obtain the following:
\begin{align*}
\pcal_{n+m}^\ast(x,E) &= \Pp_x[ \text{ for all } \, 0\leq j\leq n+m,\; \xi^x_j\in \measA \,]\\
& \kern-2.5em = \int_{\{x\}\times E^n}
\Pp_x\left[  E_{n, m}^x \, \vert \, \xi_j^x=x_j, \, \forall  \, 0\le  j \leq n  \right]\, d\Pp_x^{(n)}(x_0, x_1,\ldots, x_n)\\
& \kern-2.5em = \int_{\{x\}\times E^n}
\Pp_x\left[  E_{n, m}^x \, \vert \, \xi_n^x=x_n  \right]\, d\Pp_x^{(n)}(x_0, x_1,\ldots, x_n) .
\end{align*}

But given any $y\in E$, since $\{\xi_n^x\}$ is a Markov chain,
\begin{align*}
\Pp_x\left[  E_{n, m}^x \, \vert \, \xi_n^x=y  \right] & = \Pp_x\left[  \xi_{n+1}^x\in E, \ldots,  \xi_{n+m}^x\in E \, \vert \, \xi_n^x=y  \right] \\
& =  \Pp_y\left[  \xi_{1}^y\in E, \ldots, \xi_{m}^y\in E   \right] = \pcal_{m}^\ast(y, E) \le \pcal_{m}^\ast(E) . 
\end{align*}

Combining this with the previous identity we have:
$$\pcal_{n+m}^\ast(x,E) \le \pcal_{m}^\ast(E) \, \Pp_x^{(n)}( \{x\}\times E^n) =  \pcal_{m}^\ast(E) \, \pcal_{n}^\ast(x, E) ,$$
and the conclusion follows by taking the supremum over $x\in E$.
\end{proof}

\bigskip

Consider now the prisoner's context:
three measurable sets 
$\Pz\subset \Po\subset \Pt\in\FF$,
the cell, the prison and the state, respectively.

\begin{proposition}
\label{abstract prison break}
Given $r>0$  assume that:
\begin{enumerate}
\item[(A1)] $\pcal_{n_0}(x,\Pz)\geq b_0>0$ for all $x\in \Pz$,
\item[(A2)] $\pcal_{n_1}(x,\Po)\ge 1-r/4$ for all $x\in \Po\setminus \Pz$,
\item[(A3)] $\pcal_{n_2} (x,\Pt)\ge 1-r/4$  for $x\in \Pt\setminus \Po$,
\item[(A4)] $\pcal_n^\ast(x, \Po\comp)\ge 1-r/4$ for all $x\in \Pt\comp$ and $n\geq 0$.
\end{enumerate}
If 
$(1-b_0)^{q} <r/4$   for some    $q \in\N$ 
then setting $ N:=q\, n_0+ n_1+  n_2$,
one has that for all $x\in\Sigma$, 
$$ \Pp_x\left[\; \exists \,  j\geq N, \; \xi^x_j\in \Po  \; \right] <r . $$
\end{proposition}

\begin{proof}
From assumption  (A1) we get
$\pcal_{n_0}^\ast(x,\Pz)<1-b_0$ for all $x\in\Pz$.
Taking the sup this implies that $\pcal_{n_0}^\ast(\Pz)<1-b_0$.
By Proposition~\ref{P* sub-multip},
$$ \pcal_{q\, n_0}^\ast(\Pz) < (1-b_0)^q<r/4 . $$
Thus 
$\pcal_{q\, n_0}(x,\Pz)\geq 1-r/4$ for all $x\in \Pz$.

Combining this fact with assumptions (A2)-(A3)  and applying Proposition~\ref{prop pcal multip} we get that for all $x\in \Pt$,
\begin{equation}
\label{last escape}
 \pcal_{N}(x,\Pt)\geq (1-r/4)^3 .
\end{equation}

Finally,  the claim in the proposition reduces to showing that
$$ \Pp\left[\; \xi^{x_0}_j \in  \Po\comp,\; \forall \,  j\geq N  \;  \right] \ge \left(1-\frac{r}{4}\right)^4 > 1-r $$
for all $x_0\in\Sigma$.

If $x_0\in \Pt\comp$ this follows from (A4) because
 $1-\frac{r}{4} > 1-r$.

Otherwise, let $x_0\in \Pt$ and consider the stopping time $\tilde{\j}\colon X\to\N$ defined by
$$ \tilde{\j}:= \min \{ l\geq 0\colon \xi^{x_0}_l\notin \Pt\} .$$
Define also the random variable
$\tilde{x}\colon X\to\Sigma$,
$\tilde{x}:= \xi^{x_0}_{\tilde{\j}}$. By construction
$\tilde{x}$ takes values in $\Pt\comp$. 
Then,  using  assumption (A4) and~\eqref{last escape},
\begin{align*}
\Pp\left[\, \xi^{x_0}_j \in  \Po\comp,\, \forall   j\geq  N \,  \right] 
&\geq  \Pp_{x_0}\left[ \tilde{\j}\leq N \, \text{ and for all }\,  j\geq 0,\; \xi^{\tilde{x}}_j\in \Po\comp \right]\\
&= \sum_{l=1}^N  \Pp_{x_0}\left[ \text{ for all }\, j\geq 0 ,\; \xi^{\tilde{x}}_j\in \Po\comp \, \vert\, \tilde{\j}=l \right]\, \Pp_{x_0}[\tilde{\j}=l] \\
&\geq \left(\inf_{l\in\N } \,\inf_{x\in \Pt\comp} \pcal_l^\ast(x,\Po\comp) \right)  \sum_{l=1}^N  \Pp_{x_0}[\tilde{\j}=l] \\
&=   \left( \inf_{l\in\N } \, \inf_{x\in \Pt\comp } \pcal_l^\ast(x,\Po\comp) \right)\, \pcal_N(x_0,\Pt)\\
&\geq (1-r/4)\, (1-r/4)^3 = (1-r/4)^4>1-r\;. 
\end{align*}

\end{proof}

\subsection*{The setting}
A  given probability vector $p=(p_1,\ldots, p_k)\in \R_+^k$
and a space of symbols $\Sigma=\{1,\ldots, k\}$ determine the Bernoulli measure $\Pp_p:=p^\Z$
in the space of sequences $X=\Sigma^\Z$, which  is invariant under the full shift map
$T\colon X\to X$.
Assume now that $A=(A_1,\ldots, A_k)\in\cocycles=\SL_2(\R)^k$ is a diagonali\-za\-ble cocycle with $L(A)>0$ and  $B=(B_1,\ldots, B_k)\in \Vscr_A$ is a nearby cocycle such that
$\rho(B)= \rho_-(B)>0$,
which henceforth will be denoted by $\rho_B\footnote{The case where $\rho(B)= \rho_+(B)>0$ reduces to the previous one applied to the inverse cocycle $B^{-1}$ and will not addressed here.}$.

Given $\hat x\in\Pp$, consider the random walk $\{ \xi^{\hat x}_n=\xi^{\hat x}_n(B)\colon X\to\Pp\}_{n\geq 0}$  defined by
$$ \xi^{\hat x}_n \{\omega_j\}_{j\in\Z}  :=
\hat B_{\omega_{n-1}}\,\ldots \, \hat B_{\omega_1}\, \hat B_{\omega_0}\, \hat x .$$

\smallskip

Let $C=C(A)$ be $\frac{10}{9}$ of the homonymous constant in Proposition~\ref{prop d(B, Diag) < rho(B)}. Take a diagonalizable cocycle   $D=(D_1,\ldots, D_k)\in\Diag$ which nearly minimizes the distance to $B$, say $d(B,D)<\frac{10}{9}\,d(B,\diags)$.
By Proposition~\ref{prop d(B, Diag) < rho(B)}, 
$d(B,D)\leq C \,\rho_B$ and $\rho_B\leq C\,d(A,B)$.
We can assume that $d(B,D)<d(A,B)$ for otherwise we would take $D=A$. Hence
\begin{equation}
\label{d(A,D)}
d(A,D)\leq d(A,B)+d(B,D)< 2\, d(A,B) . 
\end{equation} 
Figure~\ref{pic} illustrates the relative positions of  $A$, $B$ and $D$ in $\Vscr_A$.

Up to a conjugation  we may assume that  $D$ is  diagonal with
$$ D_i=\begin{bmatrix}
d_i & 0 \\ 0 & d_i^{-1}
\end{bmatrix}\qquad \text{ for }\; 1\leq i\leq k $$
and so $\hat e_+(D)=(1:0)$ and $\hat e_-(D)=(0:1)$.

We will be using the following two projective coordinate systems
$\psi_\pm:\Pp\to \R\cup\{\infty\}$ defined respectively by
$\psi_-(x:y):= {x/y}$ and $\psi_+(x:y):= {y/x}$. 
Since $D$ is diagonal we have the following. With respect to $\psi_+$, the projective point $\hat e_+ (D)$ corresponds to $0$, the projective point $\hat e_- (D)$ corresponds to $\infty$, while  with respect to $\psi_-$, the r\^oles are reversed. Moreover, $\psi_\pm \left( \Dscr_\pm(D, r) \right) =(-r, r)$,
where the projective cones $\Dscr_\pm(D, r)$ were defined in~\eqref{Dscr def}.

\begin{figure}[h]
\begin{center}
	\includegraphics[scale=2.0]{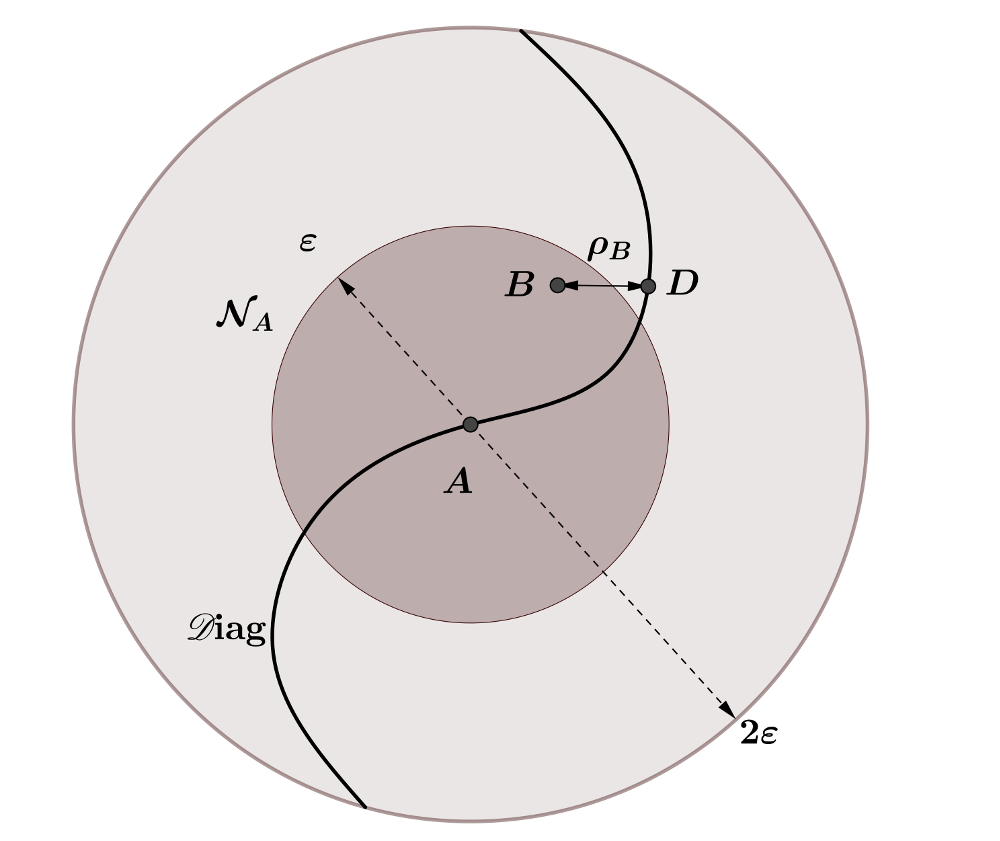}
\end{center}
\caption{Cocyles $A$, $B$ and $D$ in the neighborhood $\Vscr_A$ }
\label{pic}
\end{figure}


Proposition~\ref{N-rho comparison proposition} of the previous section, relating the measurements $N(B)$ and $\rho(B)$,
 will be proved through the  abstract prison break argument encapsulated in  Proposition~\ref{abstract prison break}.
Before specifying the sets $\Pz$, $\Po$ and $\Pt$,
and for clarity, we collect below a list of constants (depending only on $A$)  and conditions on the neighborhood $\Vscr_A$ that will be needed  in the argument (as well as to define the sets $\Sigma_i$). The following details may be skipped be  at first reading.

\smallskip

Besides the LE $L(A)$ of $A$, let  $\delta=\delta(A)$ be  the size of the cones in assumption N5, let  $L$ be the upper bound  in assumption N3 and let $C=C(A)$ be the constant introduced above, which is $\frac{10}{9}$ of the homonymous constant in Proposition~\ref{prop d(B, Diag) < rho(B)}.
Set the probability threshold
\begin{equation}
\label{def r}
r=r(A):= \frac{L(A)}{ 42 L(A) + 60 \log L }   
\end{equation} 
for the application of Proposition~\ref{abstract prison break}.
Let
\begin{equation}
\label{def M0} 
M_0:=  \frac{C\,e^{-\frac{5}{3}\, L(A)}}{1-e^{-\frac{5}{3}\, L(A)}}  + \frac{C}{1-e^{-\frac{5}{3}\, L(A)}}
 + 1  ,
\end{equation}
the parameter in Lemma~\ref{lemma d(Bj e-D, e-D)} below. 
For all $1\leq j\leq k$ set
\begin{align}
\label{def lambda tilde j}
\tilde{\lambda}_j&:= e^{2\,\log \sabs{ a_j } -\frac{L(A)}{3 }} ,\\
\label{def lambda * j}
 \lambda^\ast_j&:= e^{2\,\log \sabs{ a_j } -\frac{L(A)}{2 }} .
\end{align}
These families of numbers will be used to define two random walks
for comparison with $\xi^{\hat x}_n(B)$, see lemmas~\ref{hatBj>B*j} and~\ref{hatBj<B*j} below. Next choose $s=s(A)<\infty$ such that
\begin{align}
\label{def s}
s > \max \{ \,3\,(\tilde{\lambda}_j-\lambda_j^\ast)^{-1}\colon \,1\leq j\leq k \, \} ,\\
\label{def s1}
s > \max \{ \,(1+(\lambda_j^\ast)^{-1})\,((\lambda_j^\ast)^{-1} -\tilde{\lambda}_j^{-1})^{-1}\colon \,1\leq j\leq k \, \} . 
\end{align} 
A simple calculation gives
\begin{align*}
\lambda_j-\lambda^\ast_j \quad & \leq L^2 e^{-{L(A)\over 3}} (1- e^{-{L(A)\over 6}}) ,\\
(\lambda_j^\ast)^{-1} -\tilde{\lambda}_j^{-1})^{-1} &\leq  \frac{ L^2 e^{-{L(A)\over 3}}}{  e^{{L(A)\over 6}}-1 }  
\end{align*} 
and these bounds can be used to derive an explicit expression for $s(A)$. This  parameter $s$ is used in lemmas~\ref{hatBj>B*j} and~\ref{hatBj<B*j}.
Define also
\begin{equation}
\label{def hat c}
\hat c=\hat c(A):= \frac{1}{18}\,
\left(  \frac{L(A)}{ 2\log L + L(A)/2} 
\right)^2.
\end{equation}
The parameter $\hat c$ will be the rate of exponential decay in Hoeffding's
inequality for the random walk determined by the numbers $\lambda^\ast_j$. Next let $l_0=l_0(A)\in\N$ be the smallest integer such that
\begin{equation}
\label{def l0}
  \frac{e^{-\hat c\, l_0}}{ 1-e^{-\hat c}} 
<\frac{r}{8}  
\end{equation}
and let
\begin{equation}
\label{def kappa}
\kappa=\kappa(A):= l_0 \,\max\{-\log \lambda^\ast_j \colon 1\leq j\leq k \} .
\end{equation}
These parameters, $\hat c$ and $\kappa$, appear in the proof
of Proposition~\ref{step 2 prop}.
Finally set
\begin{align}
\label{def M}
M=M(A)&:= e^{\kappa(A)}\, s(A)\, M_0(A)   .
\end{align}


We now impose several assumptions that will further restrict
the size of the  neighborhood $\Vscr_A$. In all statements N6-N11,
 $B$ is an arbitrary cocycle in $\Vscr_A$  while $D\in\diags$ is taken, as above, near $B$.
 When we say that a certain assumption holds we mean that there exists a sufficiently small neighborhood $\Vscr_A$ 
 of $A$ such that all cocycles $B\in\Vscr_A$ satisfy that assumption.

\begin{enumerate} 
\item[N6:] $M\,(\rho_B)^{1\over 2}<\delta$. Holds because  $\rho_B\leq C\,d(A,B)$ (Proposition~\ref{prop d(B, Diag) < rho(B)}).
\item[N7:] $d(B^{-1},D^{-1}) \leq C \,\rho_B$ and  $L(D)>\frac{9}{10}\, L(A)$.
The following inequality 
$$\norm{B^{-1}-D^{-1}}_\infty\leq L^2\norm{B-D}_\infty$$
shows that if we replace $C$ by $C L^2$ then the first inequality holds without changing $\Vscr_A$. Because $D$ and $A$ are diagonalizable cocycles, in view of Remark~\ref{continuity LE in Diag}, the second assumption on $L(D)$ holds as well.
\item[N8:] $\sabs{d_i}\geq e^{L(A)/2}$ for all  
$i\in \Sigma_H(A)$, where $D=(D_1,\ldots, D_k)$ and some
coordinates are fixed where
$ D_i=\begin{bmatrix}
d_i & 0 \\ 0 & d_i^{-1}
\end{bmatrix}$. This holds by~\eqref{d(A,D)} because
 $\sabs{a_i}\geq e^{L(A)}$.

\item[N9:] $\Dscr_\pm(D,\delta)\subset \Dscr_\pm(A,2\delta)$.
Holds because $D$ is near $A$, by~\eqref{d(A,D)}.
\end{enumerate}

Given $i\in \Sigma_H(A)$, 
one has $D_i^{ \pm 1}\Dscr_\pm(\delta)\subset  \Dscr_\pm(e^{- L(A)}\delta )$. For instance if  $\hat x\in \Dscr_-(\delta)$,
using N8,
\begin{align*}
d( \hat D_i^{-1} \hat x, \hat e_-(D))&= d( \hat D_i^{-1} \hat x, \hat D_i^{-1}  \hat e_-(D)) =
\frac{d(\hat x, \hat e_-(D))}{\norm{D_i^{-1} x}\,\norm{D_i^{-1} e_-(D)}}\\
& =
\frac{d(\hat x, \hat e_-(D))}{\norm{D_i^{-1} x}\,\sabs{d_i}}\leq \delta\,
\frac{d(\hat x, \hat e_+(D))}{\norm{D_i^{-1} x}\,\sabs{d_i}}\\
&\leq \frac{\delta}{\sabs{d_i}^2}\, \frac{d(\hat x, \hat e_+(D))}{\norm{D_i^{-1} x}\,\sabs{d_i}^{-1}} \leq \frac{\delta}{\sabs{d_i}^2}\, \frac{d(\hat x, \hat e_+(D))}{\norm{D_i^{-1} x}\,\norm{D_i^{-1} e_+(D)} }\\
&= \frac{\delta}{e^{L(A)}}\, d(\hat D_i^{-1}\hat x, \hat e_+(D))
\end{align*}
which implies that $\hat D_i^{-1}\hat x\in \Dscr_-( e^{- L(A)}\,\delta)$.

\begin{enumerate}
\item[N10:] $\hat B_i^{\pm 1}\Dscr_\pm(D,\delta)\subset \Dscr_\pm(D,\delta)$ for all  
$i\in \Sigma_H(A)$. This follows 
by continuity from the previous considerations.
\item[N11:] $\log(\rhoB^{-1})> \max\{2\,\kappa, l_0\,L(A)\}$.
This holds by~\eqref{d(A,D)}. Note that $\rho_B\searrow  0$
as the size of the neighborhood $\Vscr_A$ decreases.
\end{enumerate}

\smallskip

Next consider the (projective arc) sets w.r.t. the cocycle $D$
$$\Pz:= \Dscr_-( M \rhoB) , \  \Po:= \Dscr_-(\delta^{-1}) \ \text{ and } \ \Pt:=  \Dscr_-((M \rhoB)^{-1}) .$$
 By  N6,  $M\,\rho_B<\delta$  which implies that
$$ \Pz\subset \Po \subset \Pt \subset \Pp.$$


In the prisoner's metaphor these sets
are the `cell', the `prison' and the `state'.
Throughout the rest of this section the projective cones
$\Dscr_\pm(a)$  always refer to the diagonal cocycle $D$.

\subsection*{Establishing Assumption (A1)}

\begin{proposition} 
\label{step 1 prop}
There exist constants $b_0>0$ and $n_0\in\N$,
depending on  $A$ and on the probability vector $p$, such that for all
$\hat x\in \Pz$,
$$ \Pp_{\hat x}\left[ \, \exists \, j\leq n_0\, \text{ such that }\, \xi^{\hat x}_j(B) \notin \Pz\, \right]\geq b_0. $$
\end{proposition}

\smallskip

The proof of this proposition requires some lemmas.
 
\begin{remark} 
\label{d(g x, g y)=L^2 d(x,y)}
Given $g\in \GL_2(\R)$ and unit vectors
$x,y\in\R^d$,
$$ d(\hat g\, \hat x, \hat g \, \hat y )
=\frac{\norm{g x \wedge g y}}{\norm{g x}\, \norm{g y}}
=\frac{\sabs{\det g} \norm{x \wedge y}}{\norm{g x}\, \norm{g y}}
=  \frac{\sabs{\det g}}{\norm{g x}\,\norm{g y}}\,
d(\hat x, \hat y) . $$
\end{remark}

\begin{lemma} 
\label{contraction factor near Dscr+-}
Given $i\in \Sigma_H(A)$ for all $\hat x, \hat y\in \Dscr_\pm(\delta)$,
$$ d\left( B_i^{\pm 1} \,\hat x, B_i^{\pm 1} \, \hat y \right)\leq e^{-\frac{5}{3}\, L(A)} \, d(\hat x, \hat y)    .$$
\end{lemma}

\begin{proof}
Given $i\in \Sigma_H(A)$, by N5 and N9, for any $\hat v\in \Dscr_-(\delta)$,
$$ \norm{B_i\,v} \leq e^{-\log \sabs{a_i} +\frac{L(A)}{12}}
<  e^{-\frac{5}{6}\,L(A)} . $$
Because of N10,
we  have for all $\hat v\in \Dscr_-(\delta)$,
$$ \norm{B_i^{-1}\,v} >  e^{\frac{5}{6}\,L(A)}\, \norm{B_i B_i^{-1}\, v} 
\geq   e^{\frac{5}{6}\,L(A)} $$
which in turn implies that for all $\hat x, \hat y\in\Dscr_-(\delta)$
$$ d( \hat B_i^{-1}\,\hat x, \hat B_i^{-1}\,\hat y)=\frac{d(\hat x,\hat y)}{\norm{ B_i^{-1} x}\, \norm{B_i^{-1} y}}\leq e^{-\frac{5}{3}\,L(A)}\, d(\hat x,\hat y) .$$
This proves one inequality. The other is analogous.
\end{proof}

\begin{lemma} \label{lemma d(Bj e-D, e-D)}
For all\,  $1\leq j\leq k$,
\begin{equation*}
d(\hat B_j^{\pm 1}\,\hat e_\pm (D), \hat e_\pm(D))\leq M_0\,\rhoB. 
\end{equation*} 
\end{lemma}

\begin{proof}
Because of N10, which   applies to the cocycle  $A$ as well, we have
$\hat e_\pm(A)\in \Dscr_\pm(\delta)$  and $\hat e_\pm(B_i)\in \Dscr_\pm(\delta)$ for all $i\in\Sigma_H(A)$.  We claim that for every $i\in\Sigma_H(A)$,
\begin{equation}
\label{d(Bi e-Bi, e-D)}
d( \hat e_\pm(B_i), \hat e_\pm(D))
\leq  \frac{C}{1-e^{-\frac{5}{3}\, L(A)}}\, \rhoB  . 
\end{equation} 
By Lemma~\ref{contraction factor near Dscr+-}
we have
\begin{align*}
d( \hat e_\pm(B_i), \hat e_\pm(D)) &\leq
d( \hat B_i^{\pm 1}\,  \hat e_\pm(B_i), \hat B_i^{\pm 1}\,  \hat e_\pm(D)) + 
d(\hat B_i^{\pm 1}\, \hat e_\pm(D),   \hat D_i^{\pm 1}\,\hat e_\pm(D)  )\\
&\leq  e^{-\frac{5}{3}\, L(A)}\,d( \hat e_\pm(B_i), \hat e_\pm(D)) +
\norm{B^{\pm 1}-D^{\pm 1}}_\infty,
\end{align*}
which implies by N7 that
$$ (1-e^{-\frac{5}{3}\, L(A)})\,d( \hat e_\pm(B_i), \hat e_\pm (D))\leq  C\,\rhoB  $$
and proves the claim.

Therefore,  using~\eqref{d(Bi e-Bi, e-D)},
\begin{align*}
d( \hat B_j^{\pm 1} \, \hat e_\pm(D), \hat e_\pm(D)) &\leq
d( \hat B_j^{\pm 1} \, \hat e_\pm(D) , \hat B_j^{\pm 1} \,  \hat e_\pm(B_i))\\
&\quad +
d( \hat B_j^{\pm 1} \,  \hat e_\pm(B_i),  \hat e_\pm(B_i))+
d( \hat e_\pm(B_i),   \hat e_\pm(D)  )\\
&\leq e^{-\frac{5}{3}\, L(A)}  \, d(  \hat e_\pm(D) ,  \hat e_\pm(B_i))  + \rhoB +
\frac{C}{1-e^{-\frac{5}{3}\, L(A)}}\,\rhoB ,\\
&\leq \frac{C\, e^{-\frac{5}{3}\, L(A)}}{1-e^{-\frac{5}{3}\, L(A)}}\,\rhoB  + \rhoB +
\frac{C}{1-e^{-\frac{5}{3}\, L(A)}}\,\rhoB =M_0\, \rhoB.
\end{align*}
\end{proof}

\begin{proof}[Proof of Proposition~\ref{step 1 prop}]
Fix $i\in\Sigma_H(A)$ and $1\leq j\leq k$
such that
$$ \rhoB = d(\hat B_j\, \hat e_-(B_i), \hat e_-(B_i) ) .
$$
Set $b_0= p_i^{n_0}\, p_j >0$, where ${n_0}\in\N$ is large enough but yet to be chosen depending only on $A$. We claim that for any $\hat x\in\Pz$
either $\hat B_i^{n_0}\, \hat B_j\,\hat x\notin\Pz$,
if $\hat x$ is close enough to $\hat e_-(B_i)$, or else
$\hat B_i^{n_0}\, \hat x\notin\Pz$ otherwise.
In either case the probability of these events is at least  $b_0$.

Next we need to explain how to fix ${n_0}\in\N$, depending on $A$, so that one of the above alternative claims  holds.

Given $\hat x\in \Pz$, the condition 
$d(\hat x, \hat e_-(B_i))<\frac{\rhoB}{L^2+1}$
describes the first case, in which we have
\begin{align*}
d( \hat B_j\,\hat x, \hat e_-(B_i)) &\geq 
d( \hat B_j\,\hat e_-(B_i), \hat e_-(B_i))
- d(  \hat B_j\,\hat e_-(B_i), \hat B_j\,\hat x)\\
&\geq \rhoB - L^2\, d(   \hat e_-(B_i),  \hat x)\geq
\rhoB - L^2\,\frac{\rhoB}{L^2+1} = \frac{\rhoB}{L^2+1}.
\end{align*}
Note that by  N3 and Remark~\ref{d(g x, g y)=L^2 d(x,y)}, 
$ d(\hat B_j\,\hat x, \hat B_j\,\hat y)\leq L^2\, d(\hat x, \hat y) $ 
for all $\hat x, \hat y\in \Pp$.
The previous inequality  shows that for any $\hat x\in\Pz$,
either $d(\hat x, \hat e_-(B_i))\geq \frac{\rhoB}{L^2+1}$
or else 
$d(\hat B_j\, \hat x, \hat e_-(B_i))\geq \frac{\rhoB}{L^2+1}$ and hence reduces the proof to the first case.
Assuming $d(\hat x, \hat e_-(B_i))\geq \frac{\rhoB}{L^2+1}$ we have to find ${n_0}\in\N$ such that
\begin{equation}
\label{d(Bi^q x, ei)}
d(\hat B_i^{n_0}\, \hat x, \hat e_-(B_i))\geq M'\,\rhoB
\end{equation}
with $M':= M+\frac{C}{1-e^{-\frac{5}{3}\, L(A)}}$. 
Using~\eqref{d(Bi e-Bi, e-D)} this  
implies  
\begin{align*}
\frac{d(\hat B_i^{n_0}\, \hat x, \hat e_-(D))}{d(\hat B_i^{n_0}\, \hat x, \hat e_+(D))} &\geq d(\hat B_i^{n_0}\, \hat x, \hat e_-(D))  \geq 
d(\hat B_i^{n_0}\, \hat x, \hat e_-(B_i)) - d(\hat e_-(B_i), \hat e_-(D))\\
&\geq M'\,\rhoB -\frac{C}{1-e^{-\frac{5}{3}\, L(A)}}\,\rhoB = M\,\rhoB 
\end{align*}
thus proving that $\hat B_i^{n_0}\, \hat x\notin \Pz$.

To finish, we   find ${n_0}\in\N$ such that~\eqref{d(Bi^q x, ei)} holds.
Since by N6 $\Pz\subset \Dscr_-(\delta)$,  using Lemma~\ref{contraction factor near Dscr+-} we have that
\begin{align*}
d(\hat B_i^{n_0}\, \hat x, \hat e_-(B_i))\geq
e^{ \frac{5\, {n_0}}{3}\,L(A)}\,d(  \hat x, \hat e_-(B_i))
\geq e^{ \frac{5\, {n_0}}{3}\,L(A)}\, \frac{\rhoB}{L^2+1}
\geq M'\,\rhoB
\end{align*}  
provided we take
${n_0} > \frac{3}{5}\,\frac{\log [M'\,(L^2+1)]}{L(A)}$.
Thus  we can pick ${n_0}\in\N$ depending only
on $A$ such that~\eqref{d(Bi^q x, ei)} holds
whenever $d(\hat x, \hat e_-(B_i))\geq \frac{\rhoB}{L^2+1}$.
\end{proof}

\bigskip

\subsection*{Establishing Assumption (A2)}

\begin{proposition}
\label{step 2 prop}
If $n_1(B)  :=  \lceil \frac{2}{L(A)}\,\log(\rhoB^{-1})\rceil $ then for all
$\hat x\in \Po\setminus \Pz$,
$$ \Pp_{\hat x}\left[ \, \exists \, j\leq n_1\,, \;\xi^{\hat x}_j(B) \notin \Po\, \right]\geq 1-r/4. $$
\end{proposition}

Consider the projective coordinate
system $\psi=\psi_-$ introduced above, in which $\hat e_-(D)=0$ 
and $\hat e_+(D)=\infty$.

We will keep denoting by $\hat B_j$
the action of  $B_j$ expressed in the previous projective coordinate, i.e., we write $\hat B_j(x):= \psi( \hat B_j (x:1))$.
Consider the family of numbers $\{\lambda^\ast_j\}_{1\leq j\leq k}$ defined in~\eqref{def lambda * j}.
A simple calculation gives
\begin{equation}
\label{average log lambda*}
\sum_{j=1}^k p_j\,\log \lambda^\ast_j =\frac{3\,L(A)}{2} >0 . 
\end{equation}

\begin{lemma}
\label{hatBj>B*j}
For all $x\in [ s\,M_0\,\rhoB, \delta^{-1}]$ and $1\leq j\leq k$,
\begin{equation}
\label{E lambda*}
 \hat B_j(x)\geq M_0\,\rhoB + \lambda^\ast_j\,( x-M_0\,\rhoB) =: \hat B_j^\ast (x) . 
 \end{equation}
\end{lemma}

\begin{proof}
By Lemma~\ref{lemma d(Bj e-D, e-D)}, the projective point
$\hat B_j(\hat e_-(D))$ has coordinate
$\vert \hat B_j(0)\vert \leq 2\, M_0\,\rhoB$.

Writing $B_j=\begin{bmatrix}
b_j & c_j \\ d_j & e_j 
\end{bmatrix}$ one has
$\hat B_j(x)= (b_j\,x+ c_j)\,(d_j\, x+ e_j)^{-1}$.
Since
$B_j\in\SL_2(\R)$, $\sabs{d_j}\leq \norm{B-D}_\infty<d(A,B)$ and $0\leq x \leq \delta^{-1}$
we have
\begin{align*}
\hat B_j'(x) = \frac{1}{(d_j\, x+ e_j)^2}
\geq \frac{1}{(\sabs{e_j} + \delta^{-1}\,d(A,B))^2} .
\end{align*} 
Because  $\hat e_\pm(A)\in \Dscr_\pm(\delta)$, by N10, and also by N5 and N9,
$$ \sabs{e_j} \leq \norm{B_j\,e_-(D)}\leq e^{-\log \sabs{a_j} + \frac{L(A)}{12}}= \tilde{\lambda}_j^{-\frac{1}{2}}\,e^{-\frac{L(A)}{12}}
< \tilde{\lambda}_j^{-\frac{1}{2}} .$$
Thus for all $j=1,\ldots, k$ and $x\in [0,\delta^{-1}]$,
\begin{equation*}
\hat B_j'(x)\geq \frac{1}{(\tilde \lambda_j^{-\frac{1}{2}} e^{-\frac{L(A)}{12}}+\delta^{-1} d(A,B))^2}\geq   \frac{\tilde{\lambda}_j \, e^{\frac{L(A)}{6}}}{(1+\delta^{-1} L e^{\frac{L(A)}{4}} d(A,B))^2} .    
\end{equation*}
By N0 (a)
$$ d(A,B) < \frac{\delta}{L}\, \frac{ e^{\frac{L(A)}{12}}-1 }{e^{\frac{L(A)}{4}}}   $$
which is equivalent to
$$ \gamma= \frac{e^{\frac{L(A)}{6}}}{(1+\delta^{-1} L e^{\frac{L(A)}{4}} d(A,B))^2}>1 .$$
Therefore $\hat B_j'(x)\geq \gamma\,\tilde{\lambda}_j > \tilde{\lambda}_j $ and by the mean value theorem for all $0\leq x\leq \delta^{-1}$,
$$ \hat B_j(x) \geq \hat B_j(0) +\tilde{\lambda}_j\, x
\geq - 2\, M_0\,\rhoB +\tilde{\lambda}_j\, x.$$
Since $\hat B_j'(x)\geq \tilde{\lambda}_j >\lambda^\ast_j =
(\hat B^ \ast_j)'(x)$, to see that
$\hat B_j(x)\geq  
\hat B^ \ast_j(x)$ for all $x\in [ s\,M_0\,\rhoB, \delta^{-1}]$, it is now enough to check that
\begin{align*}
 - 2\, M_0\,\rhoB +\tilde{\lambda}_j\, (s\, M_0\,\rhoB) &\geq \hat B^\ast_j(s\, M_0\,\rhoB) \\
 &= M_0\,\rhoB + \lambda^\ast_j\,( s\, M_0\,\rhoB-M_0\,\rhoB) ,
\end{align*}
which is equivalent to
\begin{align*}
 - 2 +\tilde{\lambda}_j\, s   \geq   1 + \lambda^\ast_j\,( s-1) .
\end{align*}
Hence, because~\eqref{def s} is enough  for this, by the definition of $s=s(A)$  inequality~\eqref{E lambda*} holds for all 
$x\in [ s\,M_0\,\rhoB, \delta^{-1}]$.
\end{proof}

Consider the i.i.d. process
$\eta_n\colon X\to \R_+$,
$\eta_n\{\omega_j\}_{j\in \Z}:= \log \lambda^\ast_{\omega_j}$
and its associated sum process 
$S_n=\sum_{j=0}^{n-1}\eta_j = \sum_{j=0}^{n-1}\eta_0\circ T^j$. Note that the  processes $\{\eta_n\}$ and $\{S_n\}$ are
completely determined by the cocycle $A$.

 The probability of escaping in assumption (A2) is estimated comparing the random walks $\xi^{\hat x}_n$ and $S_n$ by means of~\eqref{E lambda*}.

\begin{proof}[Proof of Proposition~\ref{step 2 prop}]
By~\eqref{average log lambda*} one has $\EE[\eta_0]=\frac{3\,L(A)}{2}$.
By Hoeffding's inequality (Lemma~\ref{Hoeffding})  the constant
$\hat c= \hat c(A)$   in~\eqref{def hat c} is such that the following LDT estimate holds for all $n\in\N$,
\begin{equation}
\label{P[Sn<nL(A)]}
 \Pp_p[\, S_n < n\,L(A)\,] <e^{-\hat c\,n} .
 \end{equation}
 Indeed the i.i.d. process $\eta_j$ satisfies
 $$ \sabs{\eta_j} = 
 \sabs{\log \lambda^\ast_{\omega_j}}
 =\sabs{2\,\log \sabs{a_{\omega_j}}- L(A)/2  } \leq 2\log L + L(A)/2=: K$$
 while the event $[ S_n<n L(A)]$ is contained in
 the deviation set 
 $$\left[ \,\abs{\ n^{-1} S_n -\EE(n^{-1} S_n)}>\epsilon\, \right] \quad \text{ with }\quad \epsilon:=\frac{L(A)}{3}.$$
Since 
$\hat c = \frac{\epsilon^2}{2 K^2}$, the previous LDT estimate follows from Hoeffding's inequality.
 
Consider now $l_0\in \N$ and $ \kappa<\infty$ respectively
defined by~\eqref{def l0} and in~\eqref{def kappa}.
Then \,
$\sum_{j=l_0}^\infty e^{-\hat c\,j} =\frac{e^{-\hat c l_0}}{1-e^{-\hat c}}<\frac{r}{8}$ \, and 
\begin{equation}
\label{kappa def}
 \Pp_p[\, S_j\geq -\kappa,\; \forall\, j=1,\ldots, l_0\,]=1 .
 \end{equation}
This implies that the event
$\Omega= [\, S_j\geq -\kappa,\; \forall\, j\geq 1\,]$ has probability
$\Pp_p(\Omega)\geq 1-\frac{r}{8}$.


Given $\hat x\in \Po\setminus \Pz$, we have
$t=\sabs{\psi(\hat x)}\in [M\,\rhoB, \delta^{-1}]$.
Inductively, and while $\xi^{\hat x}_n\in \Po$, one obtains by successive applications of inequality~\eqref{E lambda*} that for all $\omega=\{\omega_j\}_{j\in\Z}\in \Omega$,
\begin{align*}
\sabs{\psi( \xi^{\hat x}_n(\omega))} &= \hat B_{\omega_{n-1}}\,\ldots\, 
\hat B_{\omega_{1}}\,\hat B_{\omega_{0}}(t)
\geq B^\ast_{\omega_{n-1}}\,\ldots\,  \hat B^\ast_{\omega_{1}}\,\hat B^\ast_{\omega_{0}}(t)\\
&\geq M_0\,\rhoB + e^{S_n}\,(t-M_0\,\rhoB) \\
&\geq e^{-\kappa}\, M\,\rhoB = s\, M_0\,\rhoB 
\end{align*}
which in particular justifies that we  can keep  inductively applying inequality~\eqref{E lambda*}.

Consider now the event
$\Bscr=[\,  S_{n_1} < n_1\, L(A)\,]$
which by~\eqref{P[Sn<nL(A)]} has probability
$< e^{-\hat c\,n_1}$.
By N11, $l_0 < n_1$ and by~\eqref{def l0},
$e^{-\hat c\, l_0}<\frac{r}{8}$. Then the event $\Omega'=\Omega\setminus \Bscr$ 
has probability 
$$\Pp_p(\Omega') \geq \Pp_p(\Omega)-\Pp_p(\Bscr)\geq  1- \frac{r}{8} - e^{-\hat c \, n_1} > 1- \frac{r}{4}.$$
Assume now, by contradiction, that for some $\omega\in \Omega'$ one has for all $j\leq n_1$, $\xi^{\hat x}_n=\xi^{\hat x}_n(\omega)\in\Po$.
This implies  that
$\sabs{\psi( \xi^{\hat x}_{j})}\in [M\,\rhoB,\delta^{-1}]$
for all $j\leq n_1$.  We can  assume that
the number $s=s(A)$ defined by~\eqref{def s} and~\eqref{def s1}
satisfies  $s\geq 2$. Because
$M_0\geq 1$, it follows that $M\geq M_0+1$.
Then the previous inductive chain of inequalities implies that
\begin{align*}
\sabs{\psi( \xi^{\hat x}_{n_1})} &\geq M_0\,\rhoB + e^{S_{n_1}}\,(t-M_0\,\rhoB)   \\
&\geq  e^{{n_1}\, L(A)}\,(t-M_0\,\rhoB)\geq
e^{{n_1}\, L(A)}\,(M-M_0)\,\rhoB  \\
&\geq
e^{{n_1}\, L(A)}\, \rhoB  = \frac{1}{\rhoB^2}\,\rhoB =
\rhoB^{-1} > \delta^{-1}  ,
\end{align*} 
which contradicts  $\xi^{\hat x}_{n_1}\in\Po$. 
This concludes the proof.
\end{proof}

\bigskip

\subsection*{Establishing Assumption (A3)}

\begin{proposition}
\label{step 3 prop}
If $n_2(B)  :=  \lceil\frac{1}{L(A)}\,\log(\rhoB^{-1})\rceil$ then for all
$\hat x\in \Pt\setminus \Po$,
$$ \Pp_{\hat x}\left[ \, \exists \, j\leq n_2\,, \;\xi^{\hat x}_j \notin \Pt\, \right]\geq 1-r/4. $$
\end{proposition}

Consider the projective coordinate
system $\psi=\psi_+$ introduced above, in which  
$\Po\comp=\Dscr_-(\delta^{-1})\comp=[-\delta,\delta]$ and $\Pt\comp=\Dscr_-((M \rhoB)^{-1})\comp=[-M \rhoB,M \rhoB]$.

We will keep denoting by $\hat B_j$
the action of  $B_j$ expressed in the previous projective coordinate, i.e., we write $\hat B_j(x):= \psi( \hat B_j (1:x))$. 

Consider again the parameters $\tilde \lambda_j$,
$\lambda_j^\ast$ defined in~\eqref{def lambda tilde j} and
~\eqref{def lambda * j}.

\begin{lemma}
\label{hatBj<B*j}
For all $x\in [ s \,M_0\,\rhoB, \,(e^{\kappa}+1)\,\delta ]$ and $1\leq j\leq k$,
\begin{equation}
\label{EE lambda*}
 \hat B_j(x)\leq M_0\,\rhoB + (\lambda^\ast_j)^{-1}\,( x-M_0\,\rhoB)  =: \hat B_j^\ast (x)  . 
 \end{equation}
\end{lemma}

\begin{proof}
By Lemma~\ref{lemma d(Bj e-D, e-D)}, the projective point
$\hat B_j(\hat e_+(D))$ has coordinate
$\vert \hat B_j(0)\vert \leq 2\, M_0\,\rhoB$.

Writing $B_j=\begin{bmatrix}
b_j & c_j \\ d_j & e_j 
\end{bmatrix}$ one has
$\hat B_j(x)= (d_j+ e_j\, x)\,(b_j+ c_j\,x)^{-1}$.
Since
$B_j\in\SL_2(\R)$, $\sabs{c_j}\leq \norm{B-D}_\infty <d(A,B)$ and $0\leq x \leq \delta^{-1}$
we have
\begin{align*}
\hat B_j'(x) = \frac{1}{(b_j+ c_j\,x)^2}
\leq \frac{1}{(\sabs{b_j} - \delta^{-1}\,d(A,B))^2} .
\end{align*} 
Because $\sabs{d_j}\leq d(A,B)$ and  
$\sabs{b_j}\geq \sabs{a_j}-d(A,B)\geq L^{-1}-d(A,B)$,
$$  \frac{\sabs{b_j}}{\sqrt{b_j^2+d_j^2}}
= \frac{1}{\sqrt{1+\frac{d_j^2}{b_j^2}}} > 
\frac{1}{\sqrt{1+\frac{d(A,B)^2}{(L^{-1}-d(A,B))^2}}} >e^{-\frac{L(A)}{24}} , $$
where the last inequality follows from N0(b).
Therefore
\begin{align*}
 \sabs{b_j} &\geq e^{-\frac{L(A)}{24}}\,\sqrt{b_j^2+ d_j^2} = e^{-\frac{L(A)}{24}}\norm{B_j\,e_+(D)}\\
 &\geq e^{-\frac{L(A)}{24}}\, e^{\log \sabs{a_j} - \frac{L(A)}{12}} =  \tilde{\lambda}_j^{\frac{1}{2}}\, e^{\frac{L(A)}{24}}> \tilde{\lambda}_j^{\frac{1}{2}}.
\end{align*}
Thus for all $j=1,\ldots, k$ and $ x\in [0,\delta^{-1}]$,
\begin{equation*}
\hat B_j'(x)\leq 
 \frac{1}{(\tilde \lambda_j^{\frac{1}{2}} e^{\frac{L(A)}{24}} - \delta^{-1} d(A,B))^2}\leq   
 \frac{\tilde{\lambda}_j^{-1} \, e^{-\frac{L(A)}{12}}}{(1-\delta^{-1} L e^{\frac{L(A)}{8}} d(A,B))^2} .    
\end{equation*}
By N0 (c)
$$  d(A,B)< \delta\,L^{-1}\, {e^{-\frac{L(A)}{8}}} \,\left( 1-e^{-\frac{L(A)}{24}} \right)  $$
which is equivalent to
$$ \gamma= \frac{e^{-\frac{L(A)}{12}}}{(1-\delta^{-1} L e^{\frac{L(A)}{8}} d(A,B))^2}<1 .$$
Therefore $\hat B_j'(x)\leq \gamma\,\tilde{\lambda}_j^{-1}< \tilde{\lambda}_j^{-1} $ and by  the mean value theorem for all $0\leq x\leq \delta^{-1}$,
$$ \hat B_j(x) \leq \hat B_j(0) +\tilde{\lambda}_j^{-1} \, x
\leq  2\, M_0\,\rhoB +\tilde{\lambda}_j^{-1} \, x.$$
Since $\hat B_j'(x)\leq \tilde{\lambda}_j^{-1} <(\lambda^\ast_j)^{-1} =
(\hat B^ \ast_j)'(x)$, inequality
$\hat B_j(x)\leq  
\hat B^ \ast_j(x)$, for all $x\in [ s\,M_0\,\rhoB, \delta^{-1}]$, follows from
\begin{align*}
 2\, M_0\,\rhoB +\tilde{\lambda}_j^{-1}\, (s\, M_0\,\rhoB) &\leq \hat B^\ast_j(s\, M_0\,\rhoB) \\
 &= M_0\,\rhoB + (\lambda^\ast_j)^{-1}\,( s\, M_0\,\rhoB-M_0\,\rhoB) ,
\end{align*}
which is equivalent to
\begin{align*}
2  +\tilde{\lambda}_j^{-1}\, s   &\leq  1 + (\lambda^\ast_j)^{-1}\,( s -1) .
\end{align*}
Hence, because~\eqref{def s1} is enough  for this, by the definition of $s=s(A)$,
inequality~\eqref{EE lambda*} holds for all 
$x\in [ s\,M_0\,\rhoB,  \delta^{-1} ] $.
Finally, because $\delta$ can be made small enough
we may assume  that $(e^\kappa+1)\,\delta<\delta^{-1}$.
\end{proof}

 The probability of escaping in assumption (A3) is estimated comparing again the random walks $\xi^{\hat x}_n$ and $S_n$, now by means of~\eqref{EE lambda*}. The argument is analogous to the one used for assumption (A2).

\begin{proof}[Proof of Proposition~\ref{step 3 prop}]

Consider the constants $\hat c=\hat c(A)>0$,  $l_0=l_0(A)$ and $\kappa=\kappa(A)<\infty$,
respectively defined in~\eqref{def hat c},~\eqref{def l0} and~\eqref{def kappa} and consider the event
$$\Omega'= [\, S_{n_2} > n_2\, L(A) \, \text{ and }\; S_j\geq -\kappa,\; \forall\, j\geq 1\,]$$ which, as we have seen in the proof of Proposition~\ref{step 2 prop},  has probability
$\Pp_p(\Omega)\geq 1-\frac{r}{4}$.

Given $\hat x\in \Pt\setminus \Po$, we have
$t=\sabs{\psi(\hat x)}\in [M\,\rhoB, \delta ]$.
Inductively, and while $\xi^{\hat x}_n\in \Pt$, one obtains by successive applications of inequality~\eqref{EE lambda*} that for all $\omega=\{\omega_j\}_{j\in\Z}\in \Omega'$,
\begin{align*}
\sabs{\psi( \xi^{\hat x}_n(\omega))} &= \hat B_{\omega_{n-1}}\,\ldots\, 
\hat B_{\omega_{1}}\,\hat B_{\omega_{0}}(t)
\leq B^\ast_{\omega_{n-1}}\,\ldots\,  \hat B^\ast_{\omega_{1}}\,\hat B^\ast_{\omega_{0}}(t)\\
&\leq M_0\,\rhoB + e^{-S_n}\,(t-M_0\,\rhoB) \\
&\leq M_0\,\rhoB  + e^{\kappa}\,\delta < \delta\, (e^\kappa +1) 
\end{align*}
which in particular justifies that we  can keep inductively applying~\eqref{EE lambda*}. The last inequality above follows from N6.

Assume now, by contradiction, that for some $\omega\in \Omega'$ one has for all $j\leq n_2$, $\xi^{\hat x}_n=\xi^{\hat x}_n(\omega)\in\Pt$.
This implies  that
$\sabs{\psi( \xi^{\hat x}_{j})}\in [M\,\rhoB,\delta]$
for all $j\leq n_2$.  Because
$M_0+\delta < M_0+1 < M$, the previous inductive chain of inequalities implies that
\begin{align*}
\sabs{\psi( \xi^{\hat x}_{n_2})} &\leq M_0\,\rhoB + e^{-S_{n_2}}\,(t-M_0\,\rhoB)   \\
&\leq   M_0\,\rhoB + e^{-{n_2}\, L(A)}\,(t-M_0\,\rhoB)\leq M_0\,\rhoB +
e^{-{n_2}\, L(A)}\,\delta  \\
&\leq
 M_0\,\rhoB +
\rhoB\,\delta  <M\,\rhoB ,
\end{align*} 
which contradicts  $\xi^{\hat x}_{n_1}\in\Pt$. 
This concludes the proof.
\end{proof}

\bigskip

\subsection*{Establishing Assumption (A4)}

\begin{proposition}
\label{step 4 prop}
For all
$\hat x\in \Pt\comp$,
$$ \Pp_{\hat x}\left[ \, \forall \, j\geq 0 \,, \;\xi^{\hat x}_j(B) \notin \Po\, \right]\geq 1-r/4. $$
\end{proposition}

\begin{proof}
Consider the projective coordinate $\psi=\psi_+$ introduced in the previous subsection.
Through this coordinate system we make the identifications $\Pt\comp \equiv \Dscr_-((M\,\rhoB)^{-1})\comp = [-M\,\rhoB, M\,\rhoB]$ and $\Po\comp \equiv \Dscr_-(\delta)\comp =[-\delta,\delta]$.

Given $\hat x\in \Pt \comp $ and $N\in\N$, consider the event
$$\Omega_N:=[\, \exists\, 1\leq j\leq N\, \text{ such that }
\, \xi^{\hat x}_j \in \Po\,] .$$
Our goal is to prove that $\Pp(\Omega_N)<r/4$\, for all $N\in\N$.

Consider the following family of events indexed in $0\leq i\leq N$
$$ \Escr^N_i:= \left[\, \xi^{\hat x}_i\notin\Pt \, \text{ and }\,
\forall\, i<l\leq N,\; \xi^{\hat x}_l \in\Pt\,  \right] .$$
Because $\hat x\notin \Pt$, these events partition the space of sequences $X$.

As before, let $S_n :=\sum_{j=0}^{n-1}\eta\circ T^{j}$ be the sum process associated with the i.i.d. process
generated by the random variable $\eta\colon X\to\R$, $\eta(\omega):=\log \lambda_{\omega_0}^\ast$.
Given $\omega\in\Omega_N\cap \Escr^N_i$, let 
$t_i:=  \sabs{\psi(\xi^{\hat x}_{i}(\omega))}<M\,\rhoB$  and $n\in [1, N - i]$ be the first integer such that 
$\xi^{\hat x}_{i+n}\in \Po$.
Proceeding inductively,  from inequality~\eqref{EE lambda*} we get that 
\begin{align*}
\delta \leq \sabs{\psi(\xi^{\hat x}_{i + n} (\omega))} &=\hat B_{\omega_{i+n-1}}\, \ldots \, 
B_{\omega_{i+1}}\, B_{\omega_{i}}(t_i)
\leq 
\hat B_{\omega_{\tilde\i+n-1}}^\ast\, \ldots \, 
B_{\omega_{i+1}}^\ast\, B_{\omega_{i}}^\ast(t_i)\\
&\leq M_0\,\rhoB + e^{-S_n(T^{ i}\omega)}\,(t_i-M_0\,\rhoB) \\
&< M_0\,\rhoB + e^{-S_n(T^{ i}\omega)}\,(M-M_0)\,\rhoB .
\end{align*} 
By N6 this implies that $e^{-S_n(T^{i}\omega)}>1$
and hence that
$$ \delta  <  e^{-S_n (T^{i}\omega)}\, M\,\rhoB . $$
In other words we have established the inclusion
$$ \Omega_N\cap \Escr^N_i\subset T^{-i}\left[\,\exists\, n\geq 1\, \text{ such that }\,  e^{-S_n}\,M\,\rhoB > \delta \,\right] , $$ 
which implies that
$$ \Pp[\,\Omega_N \, \vert \,\Escr^N_i\, ]\leq \Pp\left[\, \exists\, n\geq 1\, \text{ such that }\,
e^{-S_n}\,M\,\rhoB > \delta \,\right] . $$
Next, define 
$l=l(A):=\max\{ \log (\lambda_j^\ast)^{-1}\colon 1\leq j\leq k\}$ so that  
$S_n\geq - n\, l$ for all $n\geq 0$.
Note also that $l=\kappa/l_0>0$.
By assumption N6  we have
$M\,\rhoB^{1/2} <\delta$.
Since
\begin{align*}
e^{-S_n}\,M\,\rhoB > \delta \quad &\Leftrightarrow \quad  e^{S_n}  < \delta^{-1}\,M\,\rhoB <\rhoB^{1/2}\\
&\Rightarrow \quad 
-n\,l\leq   S_n <   \frac{1}{2}\,\log (\rhoB) <0 \\
&\Rightarrow \quad 
 n \geq  \frac{1}{2\,l}\,\log (\rhoB^{-1})\;
 \text{ and }\; S_n <  n\, L(A) .
\end{align*}
it follows that 
$$ \Pp[\,\Omega_N \, \vert \,\Escr^N_i\, ]\leq  \Pp[ \; \exists \, n\geq \frac{1}{2\,l}\,\log (\rhoB^{-1})\; \text{ such that }\; 
S_n <  n\, L(A)\; ] $$
which by~\eqref{P[Sn<nL(A)]}, ~\eqref{def l0} and assumption N11  has probability
$$ \frac{e^{- \frac{\hat c}{2\,l}\,\log (\rhoB^{-1})} }{1-e^{-\hat c}}
=
\frac{e^{- \frac{\hat c\, l_0}{2\,\kappa}\,\log (\rhoB^{-1})} }{1-e^{-\hat c}}
< \frac{e^{-  \hat c\, l_0 } }{1-e^{-\hat c}} <\frac{r}{8} <\frac{r}{4}.
$$
Finally, applying the Law of Total Probabilities
\begin{align*}
\Pp(\Omega_N) &= \sum_{i=0}^{n-1} \Pp[\,\Omega_N \, \vert \,\Escr^N_i\, ]\,
\Pp(\Escr^N_i) \\
&< \frac{r}{4}\, \sum_{i=0}^{n-1}  \Pp(\Escr^N_i) = \frac{r}{4} .
\end{align*}
\end{proof}

\subsection*{Conclusion}

We now apply the general prison break result to our setting.


\begin{proposition}
\label{prison break main}
Let $A\in \cocycles$ be a diagonalizable cocycle such that $L(A)>0$. 
There exists 
$c_0=c_0(A)<\infty$ such that 
if $B\in \Vscr_A$ with $\rho_B=\rho_-(B) >0$  and 
$n_B :=   c_0\,  \log ( \rho_B^{-1})  $ then for all $\hat x\in\Pp$
\begin{align}
\label{prison break main bound}
\Pp\left[ \, \exists j \geq n_B,\quad \xi^{\hat x}_j(B) \in \Po   \; \right] < r .
\end{align} 
\end{proposition}

\begin{proof}
Take ${n_0}\in\N$ and $b_0>0$ as given by 
Proposition~\ref{step 1 prop}. Next choose $q=q(A)\in \N$
such that $q\geq \frac{1}{b_0}\, \log(4/r)$,
which implies that $(1-b_0)^q < \frac{r}{4}$.
Take $n_1=n_1(B)$ and $n_2=n_2(B)\in\N$
as given by  propositions~\ref{step 2 prop} and~\ref{step 3 prop}.

By Proposition~\ref{step 4 prop},
$\Pp_{\hat x}\left[ \, \forall \, j\geq 0 \,, \;\xi^{\hat x}_j \notin \Po\, \right]\geq 1-r/4$, for all  $\hat x \in\Pt\comp$.

Therefore all assumptions (A1)-(A4) of Proposition~\ref{abstract prison break} are satisfied, the first three  within the times $n_0$, $n_1$ and $n_2$ respectively.
Reducing the neighborhood $\Vscr_A$ we can assume that
$\log \rho_B^{-1}>n_0$ for all $B\in \Vscr_A$. Since the time $N=q\,n_0+n_1+n_2$ is bounded above by
$c_0\, \log(\rhoB^{-1})$ with  
$c_0=q+\frac{3}{L(A)}$, $c_0$ being a constant depending only on  $A$,
the conclusion of this proposition follows  from that of Proposition~\ref{abstract prison break}.
\end{proof}

\bigskip

\begin{proof}[Proof of Proposition~\ref{N-rho comparison proposition}]
The second inequality reduces to applying
the first one to the inverse cocycle $B^{-1}$.
Thus, from now on  we will focus on the first.

Since $A\in \Diagast$ we may assume without loss of generality that  $A$ is a diagonal cocycle,
$$ A_i=\begin{bmatrix}
a_i & 0 \\ 0 & a_i^{-1}
\end{bmatrix}\qquad \text{ for }\; 1\leq i\leq k $$
and so $\hat e_+(A)=(1:0)$ and $\hat e_-(A)=(0:1)$.
In the coordinate system $\psi_+$ associated with the diagonal cocycle $A$, the projective action of the matrix $A_i$ is given by $\hat A_i \, x = a_i^{-2}\, x$. 
In these coordinates, the random walk
$\xi_n^{\hat v}= \xi_n^{\hat v}(A)$ becomes a  random walk $\xi^x_n$ on the real line with starting point 
$x$, where $\hat v = (1\colon x)$. 

Next consider the i.i.d. process
$\eta_n$, where $\eta_n= \log \sabs{a_i}$ with probability $p_i$, and denote by $S_n=\sum_{j=0}^{n-1}\eta_j$ the corresponding sum process.
A simple verification shows that  
$$ \EE[\frac{1}{n}\, S_n]=\EE[\eta_0]=L(A) . $$

By Hoeffding's inequality there exist positive constants $n_0\in\N$ and $c_1$  depending only on $A$ such that for all $n\geq n_0$,
\begin{align}
\label{L(A) bound}
\Pp\left[ \,  \exists j\geq n , \;  \frac{1}{j}\,S_j<\frac{9}{10}\, L(A)  \, \right] < e^{-c_1\,n}   .
\end{align}

Let  $c_0=c_0(A)$ and $n_B=c_0\, \log \rho_B^{-1}$ be the constants provided by  Proposition~\ref{prison break main}. Making $\Vscr_A$  small enough, since by Proposition~\ref{prop d(B, Diag) < rho(B)}  $\rho_B\lesssim d(A,B)$,  we can assume that
$n_B\geq n_0$. Then  take the union $\Bscr=\Bscr_1\cup \Bscr_2$  of the sets defined in~\eqref{prison break main bound} and
\eqref{L(A) bound}, i.e.,
\begin{align*}
\Bscr_1 &:= \left\{ \omega\in X\colon \exists j \geq n_B \quad \xi^{\hat x}_j(B)(\omega)  \in \Po \right\}, \\
\Bscr_2 &:= \left\{ \omega\in X\colon \frac{1}{n}\,S_n(\omega) < \frac{9}{10}\, L(A) \,  
\text{ for some }\,  n \geq n_B\right\}.
\end{align*}

For each matrix $B_j$ consider the projective function
$\varphi_{B_j}(\hat v):= \log \norm{B_j\,v}$.
By N5 and N9, for all $\hat v\in\Dscr_+(A,2\delta)$ and
$1\leq j\leq k$ 
$$ \abs{
  \varphi_{B_{j}}(\hat v) 
-\log \sabs{a_{j}} }
= \abs{
\log  \norm{ B_{j} v } 
-\log \norm{A_j\,e_+(A)}  }  \leq \frac{L(A)}{12} . $$
Given  $\omega=\{\omega_j\}_{j\in\Z}\notin \Bscr$ and $\hat v\in \Pp$,
because $\xi^{\hat v}_j(B)\in \Po\comp =\overline{\Dscr_+(\delta)}\subset \Dscr_+(A,2\delta)$ for all  $n\geq n_B$,
\begin{align*}
\abs{\frac{1}{n}\,\log \norm{B^{(n)}\,v} - \frac{1}{n}\, S_n } &\leq 
\frac{1}{n}\,\sum_{j=0}^{n-1} \abs{
 \varphi_{B_{\omega_j}}(\xi^{\hat v}_j(B)) 
-\log \sabs{a_{\omega_j}}} \\
\qquad & \leq \frac{n_B}{n}\,  2\,\log L 
+ \frac{n-n_B}{n}\, \frac{L(A)}{12}\\
\qquad & \leq  2\,\frac{n_B}{n}\, \log L 
+  \frac{L(A)}{12}\leq \frac{L(A)}{10}
\end{align*}
where the last inequality holds for $n\geq \frac{120\,\log L}{L(A)}\,  n_B$.

Hence, since $\omega\notin \Bscr_2$ and   $n\geq n_0$,
\begin{align*}
\frac{1}{n}\,\log \norm{B^{(n)} v} \geq
\frac{1}{n}\, S_n -\frac{L(A)}{10}
\geq
\left( \frac{9}{10} -\frac{1}{10}\right)\,L(A)
>\frac{7}{10}\, L(A) .
\end{align*}
Therefore, integrating we have for every $\hat v\in\Pp$ and
$n\geq \frac{120\,\log L}{L(A)}\,  n_B$,
\begin{align*}
\EE\left[ \frac{1}{n}\,\log \norm{B^{(n)} v} \right] &\geq -(\log L)\, \Pp(\Bscr)
+(1-\Pp(\Bscr))\,\frac{7\,L(A)}{10}\\
&\geq \frac{2\,L(A)}{3} >\frac{L(B)}{2}
\end{align*}
where the last inequality holds by assumption N2,
and the one next-to-last is also valid provided
\begin{align}
\label{P(B) bound}
\Pp(\Bscr) <  \frac{L(A)}{30\,\log L + 21\, L(A)} = 2 r.
\end{align} 
Note from the above that in particular we also have
\begin{equation}
\label{E(log ||B||)>=2L(A)/3}
L(B)\geq \frac{2 L(A)}{3}
\end{equation}
 for all $B\in\Vscr_A$.
By~\eqref{L(A) bound} and Proposition~\ref{prison break main}, the inequality
$$ \Pp(\Bscr) < r +   e^{-c_1 n_B}   < 2\,r$$
is equivalent to
\begin{align*}
  n_B  &\geq  	\frac{1}{c_1}\,  
  \log \left(42 + \frac{60\,\log L}{L(A)}\right)   . 
\end{align*}
Since the right-hand-side of this inequality depends only on $A$,
reducing  the neighborhood $\Vscr_A$,
we can assume that it holds for all $B\in\Vscr_A$.
Therefore, setting  $K=  \frac{120\,\log L}{L(A)}$, for all $n\geq K\, n_B$ and $\hat v\in \Pp$,
 relation ~\eqref{P(B) bound} holds which in turn implies that
$$ \EE\left[ \frac{1}{n}\,\log \norm{B^{(n)} v} \right]   >\frac{L(B)}{2} .$$
By Definition~\ref{def N(A)} this proves that
$$N(B)\leq K\, n_B
\leq K\, c_0\,  \log (\rho_-(B)^{-1} ) . $$
The proposition is established with  
$c=K^{-1}\,c_0^{-1} $.
\end{proof}

\begin{remark}
\label{continuity}
As a bi-product of the previous proof we derive the fact (to be used later) that  $L(B)\geq 2 L(A)/3$. 

An improvement of this 
argument  leads to the lower semi-continuity of the LE, which when combined with the upper-semicontinuity shows that  the LE function $A\mapsto L(A)$ is continuous at any
diagonalizable cocycle $A\in \SL_2(\R)^k$  with $L(A)>0$.
As remarked in~\cite{BD}, E. Le Page's Theorem 
on  the H\"older continuity of the LE holds for  quasi-irreducible $\SL_2$-cocycles $A$ such that $L(A)>0$.
For any non diagonalizable $\SL_2$-cocycle  $A$ such that $L(A)>0$ either $A$ or else
the inverse cocycle $A^{-1}$ are quasi-irreducible. Finally, since the continuity of the LE at $A$ is equivalent to the continuity at $A^{-1}$, combining these facts we get an alternative proof of
C. Boker-Neto and M. Viana's Theorem~\cite{BV} on the general
continuity of the LE for 
$\SL_2(\R)$ and $\GL_2(\R)$ random cocycles in the finite support setting.
\end{remark}

\section{Quantitative LDT for irreducible cocycles}\label{irredldt}

In this section we establish uniform LDT estimates
for non-diagonalizable cocycles in a neighborhood
of a diagonalizable one.


\begin{theorem}\label{quant LDT Berboulli}
Given a cocycle $A\in\Diag$ with $L(A)>0$
there exists a neighborhood $\Vscr_A$ of $A$ in $\cocycles=\SL_2(\R)^k$  such that for any cocycle $B\in\Vscr_A$  with $\rho_-(B)>0$ 
\begin{equation*}
\Pp \left[ \, \left| \frac{1}{n} \, \log \norm{\Bn{n} } - L (B)  \right| > n^{-1/6} \, \right] <  e^{- n^ {1/6}} 
\end{equation*}
for all $n \ge \left[ \, c^{-1}\, \log \rho_-(B)^{-1}   \, \right]^{54}$.
\end{theorem}

\bigskip

The proof will be based on a functional analytic argument.

\smallskip

Let $(\Bscr,\norm{\cdot}_{\Bscr})$ be a Banach space
and $\Qscr\colon \Bscr\to\Bscr$   a bounded linear operator on $\Bscr$. We say that $\Qscr$ is {\em quasi-compact and simple} if there exists a $\Qscr$-invariant decomposition $\Bscr=E_0\oplus E_1$ into closed subspaces $E_0$, $E_1$ such that for some constants
$0<\sigma<\lambda$:
\begin{enumerate}
\item $\dim E_1=1$ and  $\Qscr(\phi)=\lambda\,\phi$ for all $\phi\in E_1$;
\item The operator $\Qscr\vert_{E_0}\colon E_0\to E_0$ has spectral radius $<\sigma$.
\end{enumerate}
Denoting by $P\colon \Bscr\to E_1$ the spectral projection  there is a constant $C<\infty$ such that 
for all $\phi\in \Bscr$,
\begin{enumerate}
\item $\Qscr( P(\phi)) =P( \Qscr(\phi))=\lambda\,P(\phi) $;
\item $\norm{P(\phi)}_\Bscr\leq C\,\norm{\phi}_{\Bscr}$;
\item $\norm{\Qscr^n(\phi)-\lambda^n\,P(\phi)}_{\Bscr}
\leq C\,\sigma^n\, \norm{\phi}_{\Bscr}$, for any $n\in\N$.
\end{enumerate}
We will refer to $\sigma$, $\lambda$ and $C$ as the {\em spectral constants} of the quasi-compact and simple operator $\Qscr$.

\bigskip

Let $\Sigma=\{1,\ldots, k\}$ and $\Pp=\Pp(\R^2)$.
Define the space $L^\infty(\Sigma\times \Pp)$ of all complex bounded measurable functions $\phi\colon \Sigma\times \Pp\to\C$ and let $\norm{\phi}_\infty$ denote the usual sup norm
of an observable $\phi\in L^\infty(\Sigma\times \Pp)$.

Given $0<\alpha\leq 1$ define also
$$v_\alpha(\phi):= \max_{i\in \Sigma}
\sup_{\hat p\neq \hat q}\frac{\abs{\phi(i,\hat p) - \phi(i,\hat q) } }{d(\hat p,\hat q)^\alpha} . $$

A function $\phi\in L^\infty(\Sigma\times\Pp)$ is said to be $\alpha$-H\"older continuous if $v_\alpha(\phi)<+\infty$.
Denote by $\Hscr_\alpha(\Sigma \times\Pp)$ the space of $\alpha$-H\"older continuous observables
$$ \Hscr_\alpha(\Sigma\times\Pp):=\left\{
\phi\in L^\infty(\Sigma\times\Pp)\, \colon
\, v_\alpha(\phi)<+\infty \,  \right\}   $$
and consider on this space the norm
$$ \norm{\phi}_\alpha := \norm{\phi}_\infty + v_\alpha(\phi). $$
These spaces form a scale of Banach algebras~\cite[Propositions 5.10 and 5.19]{DK-book}. 
Denote by 
$\Hscr_\alpha(\Pp)$, respectively
$L^\infty(\Pp)$, the Banach sub-algebras of the previous algebras formed by functions $\phi(i,\hat x)$
which do not depend on the variable $i$.

\bigskip

Throughout the rest of this section let $A\in \SL_2(\R)^k$ be a diagonali\-zab\-le cocycle such that $L(A)>0$ and $p=(p_1,\ldots, p_k)\in\R^k_+$ be the associated probability vector underlying the Bernoulli measure $\Pp_p=p^\Z$ on $\Sigma^\Z$.
Let $\Vscr_A$ be the neighborhood of  $A$
fixed in Section~\ref{irredm}. Then
Proposition~\ref{N-rho comparison proposition} holds for all
cocycles $B\in\Vscr_A$. Throughout the proofs of this section the neighborhood $\Vscr_A$ of $A$ may be shrunk in order to ensure that some relations regarding certain measurements of $A$ hold.

Given a cocycle $B\in\Vscr_A$, its projective action determines the Markov operator
$\Qscr_B\colon  L^\infty(\Sigma\times\Pp)
\to L^\infty(\Sigma\times\Pp)$ defined by
\begin{equation}
\label{Markov operator}
 \Qscr_B(\phi)(i,\hat x) = \int_\Sigma  \phi( l, \hat B_i\,\hat x)\, d p(l) . 
\end{equation}
This operator is associated to the kernel
$K_B\colon \Sigma\times\Pp\to \Prob(\Sigma\times\Pp)$,
$  K_B(i,\hat x):= \int_\Sigma  \delta_{(l,\hat B_i\hat x)}\, dp(l)$.
The cocycle $B$ also determines the observable $\xi_B:\Sigma\times\Pp\to\Pp$,
$\xi_B(i,\hat x):=\log \norm{B_i x}$, and the family of Laplace-Markov operators
$\Qscr_{B,t}\colon L^\infty(\Sigma\times\Pp)\to 
 L^\infty(\Sigma\times\Pp)$
\begin{equation}
\label{Laplace-Markov operator}  (\Qscr_{B,t} \phi)(i,\hat x):= \int_\Sigma  e^{t\,\xi_B(l, \hat B_i \hat x)}\, \phi(l, \hat B_i\,\hat x) \, dp(l), 
\end{equation}
which includes the Markov operator at $t=0$, i.e.,
$\Qscr_{B,0}=\Qscr_B$.

\bigskip

If $L(B)>0$ and $\rho_-(B)>0$ then the random cocycle  $B$
is quasi-irreducible. By~\cite[Proposition 4.6]{DK-31CBM}
there exists $\alpha>0$ small enough such that the Banach  sub-algebra $\Hscr_\alpha(\Sigma\times\Pp)$ is invariant under the Markov operator $\Qscr_B$, and moreover $\Qscr_B$ acts on this space as a quasi-compact and simple operator. We will  estimate the spectral constants of $\Qscr_B$ in terms of the irreducibility measurement $N=N(B)$ (see Definition~\ref{def N(A)}, Proposition~\ref{N-rho comparison proposition}).
 
By spectral continuity, for all small $t\approx 0$,
the Laplace-Markov operator $\Qscr_{B,t}$ is also a quasi-compact and simple. We will  derive bounds on the spectral constants of $\Qscr_{B,t}$ and on the size of the neighborhood
of $t=0$ where these bounds hold. Again, all these estimates are expressed in terms of the measurement $N=N(B)$.

Finally in the last part of this section we use the previous
bounds on spectral constants  to prove  Theorem~\ref{quant LDT Berboulli}.

\subsection*{Spectral constants of the Markov operator }
\label{spec MO}

Take any cocycle $B\in \Vscr_A$ such that $\rho_-(B)>0$ and write $N=N(B)$. 
Let $L$ be the upper bound introduced in Section~\ref{irredm} (not to be confused with the LE $L(B)$),   before the definition of the neighborhood $\Vscr_A$, for which we have
$$\max_{1\leq i\leq k}\,\max \{ \norm{B_i},\norm{B_i^{-1}}\} <L .$$

Given $0<\alpha\leq 1$ and $n\in\N$  define 
\begin{align*}
 \kappa^n_\alpha(B)  &:= \sup_{\hat x\neq \hat y}\,  \EE\left[ \left( \frac{d(\hat B^{(n)} \hat x, \hat B^{(n)}  \hat y) }{d(\hat x, \hat y)}\right)^\alpha  \right] \\
 & = \sup_{\hat x\neq \hat y} \sum_{j_n=1}^k\ldots
\sum_{j_1=1}^k
p_{j_n}\cdots p_{j_1}\,\left[ \frac{d(\hat B_{j_n}\cdots \hat B_{j_1} \hat x, \hat B_{j_n}\cdots \hat B_{j_1}  \hat y) }{d(\hat x, \hat y)}\right]^\alpha 
\end{align*}
This measurement 
plays a key role in the determination of the spectral constants of the operator $\Qscr_B$.
First, a straightforward calculation uncovers its sub-multiplicative behavior, i.e., for all $n,m\in\N$,
\begin{equation}
\label{kappa multip}
\kappa^{n+m}_\alpha(B)\leq \kappa^n_\alpha(B)\, \kappa^m_\alpha(B) .
\end{equation}
When $\alpha$ is small these quantities are bounded.

\begin{lemma}
\label{kappa bounded}
If \, $0<\alpha<\frac{1}{4n}$ \, then
 \; $ \kappa_\alpha^n(B) \leq L$.
\end{lemma}

\begin{proof}
See~\cite[Lemma 5.7]{DK-book}.
\end{proof}

Next lemma discloses the spectral character of the numbers $\kappa_\alpha^n(B)$.

\begin{lemma}
\label{kappa spectral}
For all $\phi\in \Hscr_\alpha(\Sigma\times\Pp)$,
$$ v_\alpha(\Qscr_B^n \phi) \leq L^{2 \alpha}\, \kappa^{n-1}_\alpha(B)\, v_\alpha(\phi) .$$\
\end{lemma}

\begin{proof}
A simple computation shows that
$$ \Qscr_B^n\phi (i,\hat x)=
\sum_{j_n=1}^k\ldots
\sum_{j_1=1}^k
p_{j_n}\cdots p_{j_1}\, \phi( j_n, \hat B_{j_{n-1}} \cdots \hat B_{j_{1}} \hat B_i \hat x) .  $$
Therefore
\begin{align*}
& \frac{\abs{(\Qscr_B^n \phi)(i,\hat x) - (\Qscr_B^n \phi)(i,\hat y)} }{d(\hat x, \hat y)^\alpha}  \leq \\
&\qquad \leq v_\alpha(\phi)\, 
\sum_{j_{n}=1}^k\ldots
\sum_{j_1=1}^k
p_{j_n}\cdots p_{j_1}\, \left[ \frac{d(\hat B_{j_{n-1}}\cdots 
\hat B_{j_{1}} \hat B_i \hat x, \hat B_{j_{n-1}}\cdots 
\hat B_{j_{1}} \hat B_i  \hat y ) }{d(\hat x, \hat y)}\right]^\alpha\\
&\qquad \leq v_\alpha(\phi)\, 
\EE\left[  \left( \frac{d( \hat B^{(n-1)} \hat B_i \hat x, \hat B^{(n-1)} \hat B_i  \hat y ) }{d(\hat   B_i \hat x, \hat   B_i \hat y)}\right)^\alpha \right]\, \left( \frac{d(  \hat B_i \hat x,  \hat B_i  \hat y ) }{d( \hat x,  \hat y)}\right)^\alpha  .
\end{align*}
Hence by Remark~\ref{d(g x, g y)=L^2 d(x,y)}
$$    d(\hat B_i \hat x, \hat B_i \hat y)  \leq \norm{B_i^{-1}}^2\, d(\hat x, \hat y) , $$
which implies that \,
$ v_\alpha(\Qscr_B^n \phi) \leq L^{2 \alpha}\, \kappa^{n-1}_\alpha(B)\, v_\alpha(\phi)$.
\end{proof}

The spectral bound $\kappa^n_\alpha(B)$ can be estimated
as the following one variable maximum.
\begin{lemma}\quad 
\label{kappa=max norm}
$\displaystyle  \kappa_\alpha^n(B) = \max_{\| x\|=1} \EE\left[ \norm{B^{(n)}\, x}^{-2\,\alpha} \,\right] $ 
\end{lemma}

\begin{proof}
See~\cite[Lemma 4.5]{DK-31CBM}.
\end{proof}

\begin{proposition} 
\label{prop: MO qcs constants}
Let $\alpha=\frac{1}{4\,L\,(\log^2 L)\,N^2}$ and 
$\sigma=\left(1-\frac{1}{8\,L\,(\log^2 L)\,N^2} \right)^{1/N}$.
Then the  operator $\Qscr_B\colon \Hscr_\alpha(\Sigma\times\Pp)\to
 \Hscr_\alpha(\Sigma\times\Pp)$ is quasi-compact and simple with 
 spectral constants 
$\lambda=1$, $\sigma$ and  $C \asymp 2\,L$ as
$N\to+\infty$.
\end{proposition}

\begin{proof}
As we decrease the size of $\Vscr_A$, $N=N(B)$ tends to $+\infty$ and by~\eqref{E(log ||B||)>=2L(A)/3}
  $L(B)^{-1}$ is bounded away from $+\infty$. Hence we can assume that
$\Vscr_A$ is small enough so that $N(B) > L(B)^{-1 }$ for all $B\in\Vscr_A$.
By Definition~\ref{def N(A)}, for any unit vector $v\in\R^2$,
$$ \frac{1}{N}\,\EE\left[\,\log\norm{\Bn{N}\, v}^{-2}\, \right]
=- \frac{2}{N}\,\EE\left[\,\log\norm{\Bn{N}\, v} \, \right] < -L(B) $$
which implies that
\begin{equation}
\label{EE[log ||BNv||(-2)<=-1}
\EE\left[\,\log\norm{\Bn{N}\, v}^{-2}\, \right]\leq -N\, L(B) < -1 .
\end{equation} 

A standard calculation (see  the proof of~\cite[Proposition 4.6]{DK-31CBM}) shows that 
for any unit vector $v\in\R^2$ and for any $\alpha>0$ 
letting $x= \alpha\, \log ( \norm{\Bn{N}\, v}^{-2})$
\begin{align*}
 \norm{\Bn{N}\, v}^{-2\,\alpha } &= e^x\leq 1+ x+\frac{x^2}{2}\,e^{\sabs{x}} \\
&\leq  1+\alpha\, \log  \norm{\Bn{N}\, v}^{-2}  + \frac{\alpha^2}{2}\, \norm{\Bn{N}}^{2\alpha} \, \log^2 \norm{\Bn{N}}^2  \\
&\leq  1+\alpha\, \log  \norm{\Bn{N}\, v}^{-2}  + 2\,\alpha^2\,L^{2\alpha N}\,  N^2 \,\log^2 L \;.
\end{align*}
Hence taking the expected value and using~\eqref{EE[log ||BNv||(-2)<=-1} we get
\begin{align*}
\EE\left[\, \norm{\Bn{N}\, v}^{-2\,\alpha }\, \right] 
&\leq  1-\alpha + 2\,\alpha^2\,L^{2\alpha N}\,  N^2 \,\log^2 L \;.
\end{align*} 
Fix now $\alpha=\frac{1}{4\,L\,(\log^2 L)\,N^2}$. 
If $N$ is large enough  then
$L^{2\alpha N}=L^{\frac{1}{2 L \log^2 L}\,\frac{1}{N^2}}\ll L
$  and so
\begin{align*}
2\,\alpha^2\,L^{2\alpha N}\,  N^2 \,\log^2 L
< \frac{1}{8\,L\,(\log^2 L)\,N^2}=\frac{\alpha}{2} .
\end{align*}
Therefore
 $$ \EE\left[\,\log\norm{\Bn{N}\, v}^{-2\,\alpha }\, \right]\leq 1-\alpha + \frac{\alpha}{2} = 1-\frac{\alpha}{2} .$$
Thus, by Lemma~\ref{kappa=max norm} one  has
$\kappa_\alpha^N(B)\leq 1-\frac{\alpha}{2}$, which is equivalent to $ \kappa_\alpha^N(B)\leq \sigma^N$. Given $n\geq 1$ write
$n-1=q\,N+r$ with $q\in \N$ and $0\leq r <N$. Because $r<N$, we have $0<\alpha\ll \frac{1}{4N}<\frac{1}{4r}$,
which by Lemma~\ref{kappa bounded} implies that $\kappa_\alpha^r(B)\leq L$.
Hence, by Lemma~\ref{kappa spectral} and the sub-multiplicative relation~\eqref{kappa multip}
 we have
\begin{align*}
v_\alpha(\Qscr_B^n\phi) &\leq v_\alpha(\phi)\, L^{2\alpha} \kappa_\alpha^{n-1}(B) \\
&\leq  v_\alpha(\phi)\, L^{\bigo(\frac{1}{N^2})}\, \kappa_\alpha^N(B)^q\, \kappa_\alpha^r(B)\\
&\leq v_\alpha(\phi)\, L^{1+\bigo(\frac{1}{N^2}) }\, \sigma^{N\,q}
\leq v_\alpha(\phi)\, L^{1+\bigo(\frac{1}{N^2}) }\, \sigma^{n-N}\\
&\leq C_N\, \sigma^n\, v_\alpha(\phi) ,
\end{align*}
with $C_N:=\frac{L^{1+\frac{1}{2\,L\,(\log^2 L)\, N^2}}}{\sigma^N} = \frac{L^{1+\bigo(\frac{1}{N^2})}}{1-\frac{\alpha}{2}}\asymp L$
as $N\to +\infty$.

Let now $\nu=\nu_B\in\Prob(\Pp)$ be the (unique) stationary measure
of the cocycle $B$ and define $\mu_B:=(\sum_{l=1}^k p_l\,\delta_l)\times\nu_B\in \Prob(\Sigma\times\Pp)$, i.e.,
$$ \int \phi\,d\mu_B := \sum_{l=1}^k\, p_l \, \int_\Pp \phi(l,\hat x)\, d\nu_B(\hat x) .$$
Consider the operator $P\colon \Hscr_\alpha(\Sigma\times\Pp)\to \Hscr_\alpha(\Sigma\times\Pp)$ defined by
$P(\phi) := (\int \phi\,d\mu_B)\,\one$, where $\one$ stands for the constant function $1$. This operator is the spectral  projection onto the eigen-space $E_1\subset \Hscr_\alpha(\Sigma\times\Pp)$ of constant functions. It  has norm $1$ because
$$ \norm{P(\phi)}_\alpha =\norm{P(\phi)}_\infty
=\abs{\smallint \phi\,d\mu} \leq \norm{\phi}_\infty\leq \norm{\phi}_\alpha.$$ 

By the spectral character of the projection $P$,
$\Qscr_B\circ P = P\circ \Qscr_B=P$. Hence, denoting the kernel of $P$ by $E_0:= \{\, \phi\in \Hscr_\alpha(\Sigma\times\Pp)\colon P(\phi)=0 \,\}$ we get a direct sum decomposition $\Hscr_\alpha(\Sigma\times\Pp)=E_0\oplus E_1$
which is invariant under the Markov operator $\Qscr_B$.

To finish note that
\begin{align*}
v_\alpha( \Qscr_B^n\phi  - P(\phi) ) &= v_\alpha( \Qscr_B^n\phi   ) \leq C_N\, \sigma^n\, v_\alpha(\phi) \leq C_N\, \sigma^n\, \norm{\phi}_\alpha .
\end{align*}
We need a similar bound on the decay
of $\norm{\Qscr_B^n\phi  - P(\phi)}_\infty$.
For that  we introduce two seminorms on $L^\infty(\Sigma\times\Pp)$.
\begin{align*}
v_0(\phi)& :=\max_{1\leq i\leq k} \, \sup_{\hat x \neq \hat y}  \, \abs{\phi(i,\hat x) - \phi(i,\hat y)} , \\
 \osc(\phi) &:=\max_{1\leq i,j\leq k} \, \sup_{\hat x , \hat y\in\Pp}  \, \abs{\phi(i,\hat x) - \phi(j,\hat y)} .
\end{align*} 
It is easily  seen that 
for all $\phi\in L^\infty(\Sigma\times\Pp)$
\begin{enumerate}
\item  $v_0(\phi)\leq v_\alpha(\phi)$\, for all $\alpha>0$
\item $\osc(\Qscr_B\phi)\leq v_0(\phi)$,
\item $\norm{\phi}_\infty\leq \osc(\phi)$ if $\int\phi\,d\mu=0$.
\end{enumerate}
The first inequality holds because $\Pp$ has diameter $1$.  
For the second note that
\begin{align*}
\abs{\Qscr_B(i,\hat x) - \Qscr_B(j,\hat y) } &\leq
\sum_ {l=1}^k p_l\, \abs{ \phi(l,\hat B_i \hat x) 
-   \phi(l,\hat B_j \hat y) }\\
& \leq \sum_ {l=1}^k p_l\, v_0(\phi) = v_0(\phi).
\end{align*}
Finally the third inequality holds because
\begin{align*}
\abs{\phi(i, \hat x)} &= \abs{ \phi(i, \hat x) - \smallint \phi\, d\mu }\\
&\leq \int \abs{ \phi(i, \hat x) - \phi(j, \hat y) }\, d\mu(j,\hat y) \leq \osc(\phi).
\end{align*}
Using these three inequalities we have
\begin{align*}
\norm{\Qscr_B^n\phi - P(\phi) }_\infty &\leq
\osc( \Qscr_B^n\phi - P(\phi) )
= \osc( \Qscr_B( \Qscr_B^{n-1} \phi - P(\phi)) )\\
&\leq v_0( \Qscr_B^{n-1} \phi - P(\phi))
\leq v_\alpha( \Qscr_B^{n-1} \phi - P(\phi))\\
&\leq C_N\,\sigma^{n-1} \, v_\alpha(\phi)
=  \sigma^{-1}\,C_N \,\sigma^n\, v_\alpha(\phi) .
\end{align*}
Therefore
\begin{align*}
\norm{\Qscr_B^n\phi - P(\phi) }_\alpha
&\leq C_N\,(1+\sigma^{-1})\,\sigma^n\, v_\alpha(\phi)\\
&\leq C_N\,(1+\sigma^{-1})\,\sigma^n\, \norm{\phi}_\alpha  ,
\end{align*}
with $C_N \,(1+\sigma^{-1}) \asymp 2\,L$.
\end{proof}

\bigskip

\subsection*{Spectral constants of the Laplace-Markov operator }
\label{spec LMO}
Consider a cocycle $B\in \Vscr_A$ such that $\rho_-(B)>0$ and write $N=N(B)$. 
From the previous section,  the spectral constants of the Markov operator $\Qscr=\Qscr_B$ acting on $\Hscr_\alpha(\Sigma\times\Pp)$ are
$\lambda_{M}=1$, $\sigma_{M}=(1-\frac{\alpha}{2})^{1/N}$ and $C_{M}<3L$, where  $\alpha=\frac{1}{4\,L\,(\log^2 L)\,N^2}$.
The Laplace-Markov operator $\Qscr_t=\Qscr_{B,t}$
is a positive operator, i.e., $\Qscr_t(\phi)\geq 0$ whenever $\phi\geq 0$.
Hence for $t\approx 0$
it has a positive maximal eigenvalue
$\lambda(t)=1+\bigo(t)$ $(t\to 0)$.

In this section we use a continuity argument to derive the spectral constants for $\Qscr_t$ on the same space $\Hscr_\alpha(\Sigma\times\Pp)$. 
The subscript $M$ in the spectral constants $\lambda_M,\sigma_M$ and $C_M$ of the Markov operator $\Qscr_B$ stands for `Markov'.
Likewise we will use a subscript $L$ for `Laplace' for the 
spectral constants $\lambda_L,\sigma_L$ and $C_L$ of the Laplace-Markov operator $\Qscr_{B,t}$.
By simply applying~\cite[Proposition 5.12]{DK-book} with $\varepsilon=\frac{1-\sigma}{3}$, so that $\sigma<\sigma+\varepsilon <1-\varepsilon <1$,
one   gets spectral constants for $\Qscr_t$
of the following magnitude:
$\lambda_{L} > 1-\varepsilon$, $\sigma_{L}<\sigma+\varepsilon$ and
$C_{L}<  ( c_1\, N^4)^{c_0\, N^2}$, for some positive constants $c_0$ and $c_1$ depending only on the cocycle $A$. Note that the bound on $C_{L}$ grows super-exponentially as $N\to +\infty$.
Even worse, these bounds hold only for   exponentially small parameters $t$, i.e., for 
$\vert t\vert\ll ( c_1\, N^4)^{-c_0\, N^2}$.
Therefore these  bounds are far from enough to prove Theorem~\ref{quant LDT Berboulli}.
 The strategy to overcome this difficulty is a trade-off between the bounds $\sigma_{L}$ and $C_{L}$. Allowing $\sigma_{L}$ to be much closer to $1$  we obtain for $C_{L}$ a  polynomial bound in $N$, which holds for polynomially small parameters $t$.

 The distance from $\sigma$ to $1$ will be a reference measurement, asymptotically (as $N\to+\infty$) given by
\begin{equation}
\label{1-sigma asymp}
1-\sigma \asymp-\log \sigma=
-\frac{1}{N}\,\log(1-\frac{1}{8\, L\,(\log^2 L)\,N^2})\asymp \frac{1}{8  L (\log^2 L) N^3}   .
\end{equation}

\bigskip

The next proposition provides a  tool for implementing this  trade-off.
 
\begin{proposition}
\label{subexpon bounds lemma}
Given     $0<C < \infty $ and $\frac{2}{3}\leq \sigma<1$,
there exist  constants $C'\lesssim \frac{C^2}{(1-\sigma)^2}$ and  $\beta_0 \gtrsim \frac{(1-\sigma)^4}{C^2}$ such that given  $0<\beta <\beta_0$ if  $\{c_n'\}_{n\geq 0}$ is a sequence with $c'_0=C$ and for all $n\geq 1$
$$
 0 < c'_n    \leq    C\,\sigma^n + \beta\,\left(  c'_0 \,c'_{n-1}  + c'_1 \,c'_{n-2}  + \ldots +
c'_{n-1} \,c'_{0}   \right) 
$$
then  for all $n\geq 0$,
\begin{align*}
  c'_n \leq C\,\sigma^n +  C'\,\beta\, \sigma^{\sqrt{n}} 
  \, <\,  C'\,    \sigma^{\sqrt{n}} .
 \end{align*}
\end{proposition}

\begin{proof} The  proof is based on 
Lemma~\ref{sum sub expon} below.
Let us set 
\begin{align*}
& c_n := C\,\sigma^n , \\
& \hat c_n := c_n + 
\beta\,\left(  c_0\,c_{n-1} + c_1\,c_{n-2} + \ldots +
c_{n-1}\,c_{0}  \right) .
\end{align*}
Define also (recursively) $c_0^\ast :=C$ and for all $n\geq 1$,
$$c_n^\ast := c_n  + 
\beta\,\left(  c_0^\ast\,c_{n-1}^\ast + c_1^\ast\,c_{n-2}^\ast + \ldots +
c_{n-1}^\ast\,c_{0}^\ast  \right) .$$
By induction we have $c_n'\leq c_n^\ast$ for all $n\geq 0$.
Consider the constant
\begin{equation}
C':=\frac{6\,C^2}{(1-\sigma)^2} .
\end{equation}
We claim that for all $j\in\N$,
\begin{equation}
\label{claim ind hyp}
 \sabs{c_j^\ast-c_j}  \leq C'\beta\, \sigma^{\sqrt{j}}  .
\end{equation}

This implies that  
\begin{equation}
\label{claim2 ind hyp}
c'_j \leq c_j^\ast  \leq   \sabs{c^\ast_j-c_j} + 
\sabs{c_j} \leq (C'\beta + C)\sigma^{\sqrt{j}} 
<  C'\sigma^{\sqrt{j}}  ,
\end{equation}

where  the last inequality  holds because
\begin{equation*}
C'\beta + C <C' 
\end{equation*}
something obvious for $\beta>0$ small enough. Thus   claim~\eqref{claim ind hyp} implies this proposition.

Let us now prove the claim. Assume, by induction hypothesis, that~\eqref{claim ind hyp}, and in particular~\eqref{claim2 ind hyp}, hold  for all $0\leq j\leq n$.
Using   Lemma~\ref{sum sub expon},
  we have
\begin{align*}
\abs{c^\ast_{n+1} - c_{n+1} }&\leq  
\abs{c^\ast_{n+1} - \hat c_{n+1} } +
\abs{\hat c_{n+1} - c_{n+1} } \\
&\leq  
\beta\, \sum_{j=0}^{n} \abs{ c^\ast_j\,c^\ast_{n-j} - c_j\,c_{n-j} } +
\beta\, \sum_{j=0}^{n} c_j\,c_{n-j} \\
&\leq  
\beta\, \sum_{j=0}^{n} \left\{ c^\ast_j\,\abs{ c^\ast_{n-j} - c_{n-j} } +
 c_{n-j} \,\abs{ c^\ast_j  - c_j } \right\}  +
\beta\, \sum_{j=0}^{n} C^2\,\sigma^n \\
&\leq  
2\,\beta^2\, C'^2  \sum_{j=0}^{n}   \sigma^{\sqrt{j}+\sqrt{n-j}}   +
 \beta\, C^2\, (n+1)\,\sigma^n \\
&\leq  \frac{40\,\beta^2\, C'^2}{(1-\sigma)^2}\,\sigma^{\sqrt{n}}  +
 \beta\, C^2\, (n+1)\,\sigma^{n/8} \,\sigma^{\sqrt{n+1}} \\
 &\leq  \beta^2\, \frac{40\,  C'^2 }{\sigma\, (1-\sigma)^2}\,\sigma^{\sqrt{n+1}}  +
 \beta\, \frac{3\,C^2}{ \sigma^{1/8}\,(1-\sigma)}\,\sigma^{\sqrt{n+1}} \\
 &\leq  \beta \, \frac{40\, \beta\,  C'^2  + 3\, C^2}{\sigma\, (1-\sigma)^2}\,\sigma^{\sqrt{n+1}}  
 < \beta\, C'\,  \sigma^{\sqrt{n+1}}  .
\end{align*}
We have used above the following inequalities
\begin{enumerate}
\item $\displaystyle \frac{n}{8}+\sqrt{n+1}\leq n$,\quad $\forall\; n\geq 2$
\item $\displaystyle  \sigma^{\sqrt{n}} \leq \frac{\sigma^{\sqrt{n+1}} }{\sigma}$,\quad $\forall\; n\geq 0$,\\
\item $\displaystyle  (n+1)\, \sigma^{n/8}\leq \frac{-8}{e\,\sigma^{1/8}\, \log \sigma }
< \frac{3}{ \sigma^{1/8}\,(1-\sigma) }$, \quad $\forall\; n\geq 0 $, 
\item $\displaystyle \frac{40\, \beta\, (C')^2    + 3\, C^2}{(1-\sigma)^2 }    <\sigma\, C' $
\end{enumerate}

The first two inequalities are clear. Inequality (3) is obtained computing the maximum of the function $f(x)=(x+1)\,\sigma^{x/8}$ over the half-line $]0,+\infty[$.

Finally, since \, $\frac{2}{3}\leq \sigma<1$ we have
\begin{align*}
(4)\Leftrightarrow \;  40\, \beta\, C'^2 &  + 3\, C^2 <
6\,\sigma\,C^2 \,   \Leftarrow \;
40\, \beta\,  C'^2    <  C^2 \\
 & \Leftrightarrow \;
40\times 36\, \beta\, C^4    <
C^2\,(1-\sigma)^4 \\
 & \Leftrightarrow \;
 1440\,\beta\, C^2    <
(1-\sigma)^4 \\
 & \Leftrightarrow \; \beta < \beta_0:= \frac{(1-\sigma)^4}{1440\,C^2}  .
\end{align*} 
Hence inequality (4) holds for all $0<\beta <\beta_0$.
\end{proof}

\begin{lemma}
\label{sub expon series}
Given $0<\sigma<1$,
$$ \sum_{n=0}^\infty \sigma^{\sqrt{n}}\leq \frac{2}{(1-\sigma)^2} + \frac{1}{1-\sigma} . $$
\end{lemma}

\begin{proof}
Given $n\in\N$,
$\lfloor \sqrt{n}\rfloor=j$ \, $\Leftrightarrow$\,  
$j^2\leq n < (j+1)^2$, so that
$$\abs{ \{n\in\N\colon \lfloor \sqrt{n}\rfloor=j\} }=(j+1)^2-j^2= 2j+1.$$
Thus
\begin{align*}
\sum_{n=0}^\infty \sigma^{\sqrt{n}}  & \leq
\sum_{n=0}^\infty \sigma^{\lfloor \sqrt{n}\rfloor }  
= \sum_{j=0}^\infty \abs{ \{n\in\N\colon \lfloor \sqrt{n}\rfloor=j\} }\, \sigma^j \\
&= \sum_{j=0}^\infty (2j+1)\, \sigma^j 
= \frac{2\,\sigma}{(1-\sigma)^2} +
 \frac{1}{1-\sigma} .
\end{align*}
\end{proof}

\begin{lemma} 
\label{sum sub expon}
Given $0<\sigma<1$ \; 
$\displaystyle  \sum_{j=0}^n \sigma^{\sqrt{j}+\sqrt{n-j}-\sqrt{n}}  \leq \frac{20}{(1-\sigma)^{2}}$,\, for all $n\geq 1$.
\end{lemma}

\begin{proof}
For $n=1$ the sum equals $2$ and the inequality is obvious.
From now on  we assume that $n\geq 2$.
For   $0\leq j\leq n$, one has
$\sqrt{j}\,\sqrt{n-j}\leq \frac{n}{2}$, which implies that
$\sqrt{n+2\,\sqrt{j}\,\sqrt{n-j}}\leq \sqrt{2\,n}$, and hence  that
$$ \frac{1}{\sqrt{2\,n}}\leq  \frac{1}{\sqrt{n+2\,\sqrt{j}\,\sqrt{n-j}} } \leq
\frac{2}{\sqrt{n+2\,\sqrt{j}\,\sqrt{n-j}}+\sqrt{n}} \, . $$
Given $0<\sigma <1$ and using the previous inequality we get
\begin{align*}
\sum_{j=0}^n \sigma^{\sqrt{j}+\sqrt{n-j}-\sqrt{n}} &=
\sum_{j=0}^n \sigma^{\sqrt{( \sqrt{j}+\sqrt{n-j})^2} -\sqrt{n}}\\
 &=
\sum_{j=0}^n \sigma^{\sqrt{n+2\,\sqrt{j} \sqrt{n-j}} -\sqrt{n}}\\
 &=
\sum_{j=0}^n \sigma^{\frac{  2\,\sqrt{j} \sqrt{n-j}}{ \sqrt{n+2\,\sqrt{j} \sqrt{n-j}} + \sqrt{n} } } \\
 &\leq 
\sum_{j=0}^n \sigma^{\frac{  \sqrt{j} \sqrt{n-j}}{   \sqrt{2\,n} } }    \leq 
2\, \sum_{j=0}^{n/2} \sigma^{\sqrt{ \frac{   j\,(n-j)}{ 2\,n } }} .
\end{align*}

For  $0\leq j\leq \frac{n}{2}$, we have $\frac{n-j}{2\,n}\geq \frac{1}{4}$. Hence 
$$ \sqrt{\frac{j\,(n-j)}{2\,n}} \geq  \sqrt{\frac{j}{4}} . $$
Thus by Lemma~\ref{sub expon series}
\begin{align*}
\sum_{j=0}^{n/2} \sigma^{\sqrt{ \frac{   j\,(n-j)}{ 2\,n } }} \leq
\sum_{j=0}^{n/2} \sigma^{  \sqrt {\frac{j}{4}} }\leq
\sum_{j=0}^{\infty} \sigma^{  \frac{  \sqrt{j}}{2} } 
\leq 
\frac{2}{(1-\sigma^{\frac{1}{2}})^2} + \frac{1}{1-\sigma^{\frac{1}{2}}}
\end{align*}

On the other hand 
$$\frac{1}{1-\sigma^{\frac{1}{2}}}  =
\frac{1+\sigma^{\frac{1}{2}}}{1-\sigma}  
\leq \frac{2}{1-\sigma} .  $$
Therefore
\begin{align*}
\sum_{j=0}^n \sigma^{\sqrt{j}+\sqrt{n-j}-\sqrt{n}} 
  & \leq 2\,\sum_{j=0}^{n/2} \sigma^{\sqrt{ \frac{   j\,(n-j)}{ 2\,n } }} \\
 & \leq    
\frac{16}{(1-\sigma)^2} + \frac{4 }{1-\sigma} \leq    
\frac{20}{(1-\sigma)^2}  .
\end{align*}
\end{proof}

Finally we establish  the spectral constants
of the Laplace-Markov family of operators.

\begin{proposition}
\label{qcs}
There are positive constants $\sigma_L$, $\lambda_L$,
$C_L$ and $b_L$ with the following properties.
For every  $z\in\overline{\mathbb{D}}_{b_L}(0)$  
the space $\Hscr_\alpha(\Sigma\times\Pp)$ admits a one dimensional subspace  $E_z$ and a co-dimension $1$ subspace  $H_z$, 
there exist a number $\lambda(z)\in\C$ and a linear map $P_z\in \mathcal{L}(\Hscr_\alpha(\Sigma\times\Pp))$  
such that for all $f\in \Hscr_\alpha(\Sigma\times\Pp)$,
\begin{enumerate}
\item    $b_L = N^{-20}$,\, $1-\sigma_L\asymp N^{-8}$,\,  $1-\lambda_L\asymp N^{-10}$ \, and \, $C_L\asymp N^6$ as $N\to+\infty$.

\item  $\Hscr_\alpha(\Sigma\times\Pp)=E_{z}\oplus H_{z}$ is a $\Qscr_{z}$-invariant decomposition,
\item  $P_{z}$ is a projection onto $E_{z}$, parallel to $H_{z}$,
\item   $\Qscr_{z}( P_{z} f) = P_{z}( \Qscr_{z} f) =\lambda(z)\,P_{z} f$,
\item  $\Qscr_{z} f = \lambda(z)\, f$ if $f\in E_{z}$,
\item  $z\mapsto \lambda(z)$ is analytic in a neighborhood of $\overline{\mathbb{D}}_{b_L}(0)$,
\item  $\displaystyle \abs{\lambda(z)}\geq \lambda_L$, 

\item  $\displaystyle \abs{\lambda'(z)}\leq 2\,C_L$, 

\item  $\displaystyle \abs{\lambda''(z)}\leq 4\,b_L^{-1}\,C_L$,

\item  $\displaystyle \norm{ \Qscr_{z}^n f - \lambda(z)^n\,  P_{z} f }
\leq C_L \, \sigma^{\sqrt{n}}\,\norm{f}$,
\item  $\displaystyle \norm{ \Qscr_{z}^n f - \lambda(z)^n\,  P_{z} f }
\leq 2\,C_L\, (\sigma_L)^n  \,\norm{f}$,
 
\item  
there is a spectral gap: 
$\displaystyle \sigma_L  < \lambda_L^2$,

\item  $\displaystyle \norm{ P_z\,f } \leq C_L\,\norm{f}$,
\item  $\displaystyle \norm{ P_z\,f -P_0\,f } \leq C_L \, \vert z\vert\,\norm{f}$.
\end{enumerate}
\end{proposition}

\begin{proof}
Consider the abstract setting discussed in~\cite[Section 5.2.1]{DK-book}.  The scale of Banach algebras $\Hscr_\alpha(\Sigma\times\Pp)$, with $\alpha\in [0,1]$,
satisfies assumptions (B1)-(B7).

For each cocycle $B\in \Vscr_A$ consider the stochastic kernel $K_B$, the corresponding stationary measure
$\mu_B$ on $\Sigma\times \Pp$  and the observable $\xi_B\colon\Sigma\times \Pp\to \R$,
$$ \xi_B(i,\hat x):= \log \norm{B_i\,x}  $$
where $x$ stands for a unit representative of $\hat x\in\Pp$.
In~\cite{DK-book}
we refer to the tuple  $(K_B,\mu_B, \xi_B)$ as an observed Markov space on $\Sigma\times \Pp$. Denote by $\Xfrk$ the space of all observed Markov spaces of the form
$(K_B,\mu_B, \pm\,\xi_B)$ with $B\in \Vscr_A$ and $\rho_-(B)>0$.
This space satisfies assumptions (A1)-(A3) in~\cite[Section 5.2.1]{DK-book}.

Assumption (A1) follows automatically from the definition of $\Xfrk$. Assumption (A2) is a consequence of Proposition~\ref{prop: MO qcs constants} with the constants therein. Let us now focus on assumption (A3).
A simple calculation shows that $\norm{\xi_B}_\infty\leq \log L$
and 
\begin{align*}
\abs{ \log \norm{B_i \,x} - \log \norm{B_i \,y}  }&\leq
\frac{ \abs{  \norm{B_i \,x} - \norm{B_i \,y}  } }{L^{-1}} 
\leq L\,\norm{ B_i x - B_i y}\\
& \leq L^2 \, \norm{x-y} \leq \sqrt{2}\, L^2 \, d(\hat x, \hat y) 
\end{align*}
which implies that $v_\alpha(\xi_B)\leq v_1(\xi_B) \leq \sqrt{2}\, L^2 $. Therefore  $\norm{\xi_B}_\alpha <2\, L^2$.
The family of Laplace-Markov operators
$\Qscr_z\colon \Hscr_\alpha(\Sigma\times \Pp)\to \Hscr_\alpha(\Sigma\times \Pp)$ in~\eqref{Laplace-Markov operator} can be written as
$\Qscr_z(f) = \Qscr_0(f\,e^{z\,\xi_B})$ and  since $\Hscr_\alpha(\Sigma\times \Pp)$ is a Banach algebra,
the map $\C\ni z\mapsto \Qscr_z\in \mathscr{L}(\Hscr_\alpha(\Sigma\times\Pp))$ is an entire (analytic) function.
By Proposition~\ref{prop: MO qcs constants}
$\norm{\Qscr_0 f}_\alpha \leq (1+ C\, \sigma)\,\norm{f}_\alpha$
with $1+C\,\sigma\asymp 2\,L+1$.
Hence, for $i=0,1,2$,
\begin{align*}
\norm{\Qscr_z(f\,\xi_B^i)}_\alpha &=\norm{\Qscr_0(f\,\xi_B^i \, e^{z\,\xi_B})}_\alpha
\leq (2L+1)\,\norm{ f\,\xi_B^i\, e^{z\,\xi_B} }_\alpha\\
&\leq (2L+1)\,  \norm{\xi_B}_\alpha^i \, e^{\vert z\vert \,\norm{\xi_B}_\alpha} \, \norm{ f}_\alpha \\
&\leq (2L+1)\, (2\,L^2)^2\,  e^{2\,\vert z\vert \,L^2} \, \norm{ f}_\alpha .
\end{align*}
Therefore, setting $b=  (10 L^2)^{-1}$ and $M= 10\, L^5$, we have  for all $z\in \mathbb{D}_{b}$
$$\norm{\Qscr_z(f\,\xi_B^i)}_\alpha\leq 8\, e^{1/5}\,L^5\,\norm{f}_\alpha \leq M\, \norm{f}_\alpha  $$
which proves (A3).

Item (1) will follow  from the choices made below.

Except (7)-(12), all other items of this proposition 
follow from~\cite[Proposition 5.12 and Remark 5.3]{DK-book}.
Actually, apart from  (7)-(12), everything else follows from a general spectral continuity argument and the quantitative lemma~\cite[Lemma 5.2]{DK-book}. Items (13) and (14) for instance hold with $C_L =\frac{(3+3 C)^2}{(1-\sigma)^2}\, M$, with $C=2 L$ and
$M=10  L^5$. Using the asymptotic expression~\eqref{1-sigma asymp} for $1-\sigma$ we see that we can take $C_L= \kappa_0\, L^{10}\, N^6$, where $\kappa_0$ stands for some absolute constant.
Moreover, these bounds hold for all $\vert z \vert \leq L^{-{10}}\, N^{-6}$.

Let us now prove (10). From now on we will focus our  analysis of the  family of operators $\Qscr_z$ on the smaller disk $\mathbb{D}_{4\,b_L}(0)$ with  radius $4\, N^{-20}$, i.e., $b_L:=N^{-20}$.
Consider the family of operators \,
$\Nscr_z:=\Qscr_z-\lambda(z)\,P_z $.
The same lemma,~\cite[Lemma 5.2]{DK-book}, invoked to justify (14) shows that 
$\norm{\frac{d}{dz} \Nscr_z}\leq C_L$ for all $z\in \overline{\mathbb{D}}_{4 b_L}(0)$.
Defining
 $$ c_j':= \sup_{z\in\overline{\mathbb{D}}_{4 b_L}(0)} \norm{\Nscr_z^j}  \qquad (j\geq 0) $$
the previous inequality shows that
\begin{align*}
\norm{ \frac{d}{dz} ( \Nscr_z^n) } & \leq  \sum_{j=1}^n\norm{\Nscr_z^{n-j}}\,\norm{\frac{d}{dz}\Nscr_z}\, \norm{\Nscr_z^{j-1}}  \leq C_L\,\sum_{j=1}^{n}
c_j'\,c_{n-j}'.
\end{align*}
Next set $c_j:=C\,\sigma^j$, so that
$\norm{\Nscr_0^j}\leq c_j$.  Then, for all $z\in \overline{\mathbb{D}}_{4 b_L}(0)$,
$$
 \norm{ \Nscr_z^n }  \leq  \norm{ \Nscr_0^n } +  4\,C_L\,b_L\,\sum_{j=1}^{n}
c_j'\,c_{n-j}'  \leq  c_n  +  4\,C_L\,b_L\,\sum_{j=1}^{n}
c_j'\,c_{n-j}'  .
$$
Setting $\beta:= 4\,C_L\,b_L$, this  implies that  
$$ c_n' \leq c_n + \beta\,(c_0'\, c_{n-1}' + c_1'\, c_{n-2}' + \ldots + c_{n-1}'\, c_{0}' ). $$ 
By Proposition~\ref{subexpon bounds lemma}, since for large $N$ we have that
$$ \beta=4\,C_L\,b_L\asymp \frac{L^{10}\, N^6}{N^{20}} \ll
\frac{1}{2^{14}\, L^6 \, (\log^8 L)\, N^{12}}= \frac{(1-\sigma)^4}{ C^2 }=:\beta_0 $$
we have for all $n\geq 0$ and $z\in\overline{\mathbb{D}}_{4 b_L}(0)$, \,
$$\norm{\Qscr_z^n-\lambda(z)^n\, P_z } = \norm{\Nscr_z^n} \leq c_n' \leq  C'\, \sigma^{\sqrt{n}}
\leq  C_L\, \sigma^{\sqrt{n}}  $$
which proves (10).
Note that increasing if necessary the value of  $\kappa_0$ we ensure  that
$$ C'\asymp C^2\,(1-\sigma)^{-2}  
= 2^8\, L^4\, (\log^4 L)\, N^6  < \kappa_0\, L^{10}\, N^6=C_L . $$

Consider the numbers  
$$ \sigma_L:=\exp(-\frac{1}{N^8})\; \text{ and } \;
\lambda_L:=\exp(-\frac{1}{N^{10}}) $$
which satisfy
$1-\sigma_L\asymp N^{-8}$ and $1-\lambda_L\asymp N^{-10}$.
With these definitions item (12) holds because \,
$2\,N^{-10}\ll N^{-8}$.

To prove item (11) consider the integer $n_0=N^{8}$.
The inequality
\begin{equation*}
C'\,\sigma^{\sqrt{n_0}}=C'\,\exp\left(-\sqrt{n_0} \,\log \sigma^{-1} \right)
\leq \exp\left(- \frac{n_0}{N^8} \right) = \sigma_L^{n_0} 
\end{equation*}
holds because
$$ \sqrt{n_0} \,\log \sigma^{-1}\asymp \frac{N^{4}}{  N^3} 
\asymp  N  \gg 1 = \frac{n_0}{N^{8}} . $$
Hence $\norm{\Nscr_z^{n_0}}\leq \sigma_L^{n_0}$, for all $z\in \overline{\mathbb{D}}_{4 b_L}(0)$. Given  $n\geq 1$, performing an integer division we  write $n= q\,n_0 + r$ with $0\leq r <n_0$.
Thus, by item (10) we get for all $z\in \overline{\mathbb{D}}_{4 b_L}(0)$
\begin{align*}
\norm{\Nscr_z^n}  &\leq \norm{\Nscr_z^r}\, \norm{\Nscr_z^{n_0}}^q
 \leq   C_L\,(\sigma_L)^{q\,n_0}\\
 &\leq  C_L\,(\sigma_L)^{n-n_0}\lesssim C_L\,(\sigma_L)^{n} 
\end{align*}
which proves (11). Notice that $\sigma_L^{n_0}=e^{-1}\asymp 1$.

We  focus now on the proof of items (7)-(9).
Consider the family of  operators $L_z=\Qscr_z\circ P_z$.
By~\cite[Lemma 5.2]{DK-book}, for all $z\in \overline{\mathbb{D}}_{4 b_L}(0)$
$$ \norm{L_z-L_0}\leq C_L\, \abs{z} . $$
Let $\one$ denote the constant function $1$ on $\Sigma\times\Pp$
and $\mu_B$ be the stationary probability measure on $\Sigma\times \Pp$. Then for all $z\in\overline{\mathbb{D}}_{4 b_L}(0)$,
\begin{equation*} 
\lambda(z) = \frac{\langle L_{z}\one, \mu_B  \rangle}{\langle P_{z}\one, \mu_B  \rangle}\;.
\end{equation*}
The denominator above does not vanish because
$$ \langle P_{z}\one,\mu_B \rangle \geq 
\langle P_{0}\one,\mu_B \rangle -\norm{P_{z}\one -P_{0}\one}_\infty 
\geq 1-C_L\,b_L\asymp 1\;.$$
Hence, for all  $z\in \overline{\mathbb{D}}_{4 b_L}(0)$,
\begin{align*}
\abs{\lambda(z)-1} &\lesssim \abs { 
 \langle L_{z}\one,\mu_B \rangle -
 \langle P_{z}\one,\mu_B \rangle   } \\
&\leq  
 \abs{ \langle L_{z}\one-L_{0}\one ,\mu_0 \rangle} +
\abs{ \langle P_{z}\one-P_{0}\one,\mu_0 \rangle}     \\
&\leq  
  2\,C_L\,b_L\asymp  N^{-14}\ll N^{-10}\;.
\end{align*}
This proves (7).

By Cauchy's integral formula,  for all  $z\in \overline{\mathbb{D}}_{2 b_L}(0)$
\begin{align*}
\vert\lambda'(z)\vert &\leq  \frac{1}{2\pi}\, \int_{\vert \zeta-z\vert =4 b_L} \frac{\vert\lambda(\zeta)-1\vert}{\vert\zeta-z\vert^2}\, \vert d\zeta\vert  
\leq \frac{8 \pi b_L}{2\pi} \frac{2 C_L b_L}{(2 b_L)^2} =2\, C_L ,
\end{align*}
which proves (8). Similarly, for all  $z\in \overline{\mathbb{D}}_{b_L}(0)$
\begin{align*}
\vert\lambda''(z)\vert &\leq  \frac{1}{2\pi}\, \int_{\vert \zeta-z\vert =2 b_L} \frac{\vert\lambda'(\zeta)\vert}{\vert\zeta-z\vert^2}\, \vert d\zeta\vert  
\leq \frac{4 \pi b_L}{2\pi} \frac{2 C_L}{b_L^2} =4\, b_L^{-1}\, C_L ,
\end{align*}
which proves (9).
\end{proof}

\subsection*{Proving large deviation estimates}

\begin{proposition}\label{parametric LDT}
Given a cocycle $A\in\Diag$ with $L(A)>0$
there exists a neighborhood $\Vscr_A$ of $A$ in $\cocycles=\SL_2(\R)^k$ such that for any cocycle $B\in\Vscr_A$  with $\rho_-(B)>0$, if\,  $N=N(B)$ there are positive constants  $h\asymp N^{26}$ and $C\asymp N^6$ 
such that for all $0<\varepsilon< 1$ and  $n\geq 1$,
\begin{equation*}
\Pp \left[ \, \left| \frac{1}{n} \, \log \norm{\Bn{n} } - L (B)  \right| > \varepsilon \, \right] \leq C\,   e^{- \frac{\varepsilon^2}{2 h}\,n } .
\end{equation*}
\end{proposition}

\begin{proof}
The proof of this result reduces to~\cite[Theorem 4.1]{DK-31CBM} or~\cite[Theorem 5.3]{DK-book}. We outline it here for the reader's convenience.

As before let $X=\Sigma^\Z$ and $T:X\to X$
denote the shift homeomorphism. Define the skew-product map
$F_B:X\times \Pp\to X\times \Pp$ by
$F_B(x,\hat p):=(T x, \hat B_{x_0} \hat p)$.
Recall that $\nu_B\in\Prob(\Pp)$ stands for the (unique) stationary measure of the cocycle $B$ on $\Pp$.
Then $\Pp_B=(p^\Z)\times \nu_B$ is an $F_B$ -invariant probability measure on $X\times \Pp$. Next consider the random process
$e_n:X\times \Pp\to \Sigma\times\Pp$, 
$e_n(x,\hat p):= (x_n,\hat B^{(n)}(x)\, p)$, which satisfies
$e_n=e_0\circ F_B^n$ for all $n\geq 0$.
Given any probability  $\theta\in\Prob(\Sigma\times\Pp)$
there is a unique probability measure $\Pp_\theta \in \Prob(X\times\Pp)$ such that for any measurable set $E\subset \Sigma\times \Pp$,
\begin{enumerate}
	\item $\Pp_\theta[\, e_0\in E\,]=\theta(E) $,
	\item $\Pp_\theta[\, e_{n+1}\in E\, \vert e_n=(i,\hat p)\, ]
	=K_B((i,\hat p),E) $.
\end{enumerate}
We will refer to $\Pp_\theta$ as the Kolmogorov extension of $\theta$.
Because $\mu_B=(\sum_{l\in\Sigma} p_l\,\delta_l)\times \nu_B$ is the stationary measure of the stochastic kernel $K_B$,
the measure $\Pp_B$ is the Kolmogorov extension of $\theta=\mu_B$.
We denote by $\Pp_{(i,\hat p)}$ the Kolmogorov extension of 
a Dirac mass $\theta=\delta_{(i,\hat p)}$.
Averaging these Dirac masses we recover the invariant probability
\begin{equation}
\label{PB average}
\Pp_B = \int_{\Sigma\times \Pp}  \Pp_{(i,\hat p)}\, d\mu_B(i,\hat p)
.
\end{equation}

Let $\xi_B:\Sigma\times\Pp\to\R$ be the observable defined by
$\xi_B(i,\hat p):=\log \norm{ B_i\, p}$
where $p$ stands for a unit vector representative of $\hat p$.
This observable determines the random process $\eta_n:X\times\Pp\to\R$,
$\eta_n:= \xi_B\circ e_n$, and a corresponding sum process
$S_n=\sum_{j=0}^{n-1} \eta_j$. For all
$(x,\hat p)\in X\times\Pp$,
\begin{equation}
\label{Sn = log norm Bn}
S_n(x,\hat p)=\sum_{j=0}^{n-1}
\log \norm{ B_{x_j}\, \frac{B_{x_{j-1}}\ldots B_{x_0} \,p}{\norm{B_{x_{j-1}}\ldots B_{x_0} \,p}} } = \log \norm{ B^{(n)}(x)\, p }
 \end{equation} 
where $p$ is a unit vector representative of $\hat p$. This relates the additive process $S_n$ to the sub-additive one $\log \norm{B^{(n)}}$. The process $\{\eta_n\}_{n\geq 0}$ is stationary
w.r.t. $\Pp_B$ and by Furstenberg's formula we have 
$$ \EE[ \eta_n  ]=\int_{\Sigma\times\Pp} \xi_B\, d\mu_B = L(B) . $$
In particular $\frac{1}{n}\,\EE( S_n)=L(B)$ for all $n\geq 1$.
We may assume that $L(B)=0$ (otherwise we work with the $\GL_2$-cocycle $\tilde{B}$ obtained by 
multiplying  $B$ with the factor $e^{-L(B)}$).

By Chebyshev's inequality we get the following upper bound 
\begin{equation}
\label{Chebyshev}
 \Pp_{(i,\hat p)}\left[\, \frac{1}{n}\, S_n\geq \varepsilon \, \right]\leq e^{-n\, t\, \varepsilon}\, \EE_{(i,\hat p)}[ e^{t\, S_n} ]
\end{equation}
where $t$ and $\varepsilon$ are positive parameters and  $\EE_{(i,\hat p)}$ stands for  expectation w.r.t.
the probability measure $\Pp_{(i,\hat p)}$. The right-hand-side expected value can be expressed in terms of the  operator $\Qscr_{B,t}$:
$$ \EE_{(i,\hat p)}[ e^{t\, S_n} ] = (\Qscr_{B,t}^n\one )(i,\hat p). $$
For $t\in\R$, the Laplace-Markov operators $\Qscr_{B,t}$
are positive operators on the Banach lattice 
$\Hscr_\alpha(\Sigma\times\Pp)$ with $\alpha= (4\,L\,(\log^2 L)\, N^2)^{-1} $. By items (7) and (11) of Proposition~\ref{qcs} they are quasi-compact and simple operators for all
$0<t<b_L = N^{-20}$.
Hence $\Qscr_{B,t}$ has a unique positive maximal eigenvalue
$\lambda(t)$ which depends analytically on $t\in \overline{\mathbb{D}}_{b_L}(0)$. The uniform spectral decomposition
of this family of operators implies that
\begin{align*}
(\Qscr_{B,t}^n\one )(i,\hat p) &=\lambda(t)^n + 
\abs{(\Qscr_{B,t}^n \one)(i, \hat p)-\lambda(t)^n} \\
& \leq 
\lambda(t)^n +  \norm{ \Qscr_t^n \one-\lambda(t)^n\,\one}_\alpha\\
&\leq \lambda(t)^n  + \lambda(t)^n 
\norm{ P_t \one- \one}_\alpha +  \norm{ \Qscr_t^n \one-\lambda(t)^n\,P_t\one}_\alpha \\
&\leq \lambda(t)^n (1 +  C_L\,\vert t\vert) + C_L\,\sigma_L^n 
\leq 2\, C_L\,\lambda(t)^n .
\end{align*}
The last inequality holds because $\sigma_L< 1\leq \lambda(t)$, a claim proven below. Hence  
\begin{equation}
\label{EE exp t Sn}
 \EE_{(i,\hat p)}[ e^{t\, S_n} ] = (\Qscr_{B,t}^n\one )(i,\hat p) \leq 2\,C_L\,\lambda(t)^n = 2\,C_L\, e^{n\, c(t)}  
\end{equation}
with \, $c(t):=\log \lambda(t)$ \, for $t\in \overline{\mathbb{D}}_{b_L}(0)$.
The function $c(t)$ satisfies $c(0)=0$ because $\lambda(0)=1$,
and $c'(0)=0$ because $L(B)=0$. Now,    $\lambda(t)\geq 1$ is equivalent  to $c(t)\geq 0$ which follows by Jensen's inequality and the concavity of the logarithmic function.
In fact
\begin{align*}
c(t) &= \lim_{n\to +\infty} \frac{1}{n}	\,\log \norm{\Qscr_{B,t}^n} \geq \lim_{n\to +\infty} \frac{1}{n}	\,\log \norm{\Qscr_{B,t}^n\one } _\infty\\
& \geq \lim_{n\to +\infty} \frac{1}{n}	\,\log \EE_{(i,\hat p)}[ e^{t\, S_n} ] \geq \lim_{n\to +\infty} \frac{1}{n}	\, \EE_{(i,\hat p)}[  t\, S_n  ] =0 .
\end{align*}
Note that by~\cite[Lemma 4.3]{DK-31CBM},
 $\frac{1}{n}	\, \EE_{(i,\hat p)}[ S_n  ] = \frac{1}{n}	\, \EE[ \log \norm{B^{(n)}\,\frac{ B_i\,p}{ \norm{B_i\,p}}}  ]$ converges to $L(B)=0$  uniformly in $\hat p$.

By item (9) of Proposition~\ref{qcs} we have 
$$h:=\max_{t\in \overline{\mathbb{D}}_{b_L}(0)} \vert c''(t)\vert \leq 4\, b_L^{-1}\, C_L\asymp N^{26}.$$
For all $0<t< b_L$  one has $ c(t)\leq \frac{h}{2}\,t^2 $  and
combining~\eqref{Chebyshev} with~\eqref{EE exp t Sn}
$$ \Pp_{(i,\hat p)}\left[\, \frac{1}{n}\, S_n\geq \varepsilon \, \right]\leq 2\,C_L\,  e^{-n\, (t\, \varepsilon -\frac{h}{2}\,t^2)} .$$
Choosing the value of $t$ which minimizes the right-hand-side of the previous inequality leads to the following upper bound
$$\Pp_{(i,\hat p)}\left[\, \frac{1}{n}\, S_n\geq \varepsilon \, \right]\leq  2\, C_L\,  e^{-n\,\hat \varphi(\varepsilon) }\; \text{ with  } \; \hat \varphi(\varepsilon):=\max_{\vert t\vert<b_L} \left( t\, \varepsilon - \frac{h}{2}\, t^2\right)$$
where $\hat \varphi(\varepsilon)$ is the Legendre transform
of the quadratic function $\varphi:(-b_L,b_L)\to\R$, $\varphi(t):= \frac{h}{2}\, t^2$. A simple calculation shows that $\hat \varphi(\varepsilon)=\frac{\varepsilon^2}{2\,h}$
for all $\vert \varepsilon\vert <h\, b_L\asymp N^4$.
Averaging these inequalities in $(i,\hat p)$ w.r.t. $\mu_B$, through~\eqref{PB average} we get for all $n\geq 1$ and 
$0<\varepsilon<N^4$ 
$$ \Pp_{B}\left[\, \frac{1}{n}\, S_n\geq \varepsilon \, \right] \leq 2\,C_L\,  e^{-n\, \frac{ \varepsilon^2}{2\,h} } . $$
This provides an estimate for deviations of $\frac{1}{n}\, S_n$ above the average. Applying the same method to the symmetric process, the same bound holds for deviations below the average. We have assumed so far that $L(B)=0$. In the general case by~\eqref{Sn = log norm Bn} we have for all $n\geq 1$
$$ \Pp_{B}\left[\, \abs{ \frac{1}{n}\, \log \norm{B^{(n)}\, p} -L(B)  } \geq \varepsilon \, \right] \leq 4\,C_L\,  e^{-n\, \frac{ \varepsilon^2}{2\,h} } .$$
Finally, a standard argument allows us to substitute in the previous LDT estimate $n^{-1}\, \log \norm{B^{(n)}\, p}$ by $n^{-1}\, \log \norm{B^{(n)}}$, see the proof of~\cite[Theorem 4.1]{DK-31CBM}. This concludes the proof.
\end{proof}

\bigskip

\begin{proof}[Proof of Theorem~\ref{quant LDT Berboulli}]
Take $\varepsilon= n^{-1/6}$ with
$n \ge \left[ \, c^{-1}\, \log \rho_-(B)^{-1}   \, \right]^{54}$.
By Proposition~\ref{N-rho comparison proposition} this implies that $n\geq N^{54}$. Then by Proposition~\ref{parametric LDT}, the event
$\Delta_n:= \left[ \, \left|  n^{-1} \, \log \norm{\Bn{n} } - L (B)  \right| > n^{-1/6} \, \right]$ has probability
\begin{align*}
\Pp(\Delta_n) & \lesssim N^6\,   e^{- \frac{n^{-1/3}\, N^{-26}}{2}\,n }  \leq n^{1/9}\,   e^{- \frac{1}{2}\,n^{1-\frac{1}{3}-\frac{26}{54}}  }  \\
&\leq n^{1/9}\,   e^{- \frac{1}{2}\,n^{\frac{10}{54}}  }
\ll   e^{-  n^{\frac{1}{6}}  }
\end{align*}
The last inequality holds for all sufficiently large $n$,
so decreasing the size of the neighborhood $\Vscr_A$  we can always assume that $N$ is large enough that this holds for all $n\geq N^{54}$.
This proves the theorem. 
\end{proof}

\section{The bridging argument}
\label{bridging}
In this section we prove a uniform LDT estimate in the vicinity of a diagonalizable cocycle (for which an LDT estimate already holds at all scales) up to a certain {\em finite scale} that depends on the proximity to the set of diagonalizable cocycles. 

For every cocycle $A\in\SLtwoR^k$ and scale $n\in\N$, the number
$$\LE{n}(A) := \EE  \left[   \frac{1}{n} \, \log \norm{\An{n} } \right]$$
will be referred to as the  {\em finite scale} LE of $A$.
  
Note that as $n\to\infty$,  $\LE{n}(A) \to L(A)$ (the infinite scale LE).

\begin{proposition}\label{initial scale}
Let $A \in \Diag$ with $L(A)>0$. There are constants $\delta = \delta (A) > 0$ and $\nldt = \nldt (A) \in \N$ such that if $B \in \SLtwoR^k$ is any cocycle with $d (B, A) \le \delta$, then for all scales $n$ in the range $$\nldt \le n \le \left( \log d (B, \Diag)^{-1} \right)^{3/4}$$ we have 
\begin{align*}
 & \LE{n}(B) \ge \frac{L(A)}{3} ,\\
 & \abs{\LE{2n}(B)-\LE{n}(B)}  \le 2 n^{-1/5} ,\\
 & \Pp \left[  \left| \frac{1}{n} \, \log \norm{\Bn{n} } - \LE{n} (B)  \right| > n^{-1/5}  \right]  \le e^{- n^\frac{1}{3}} \, .
\end{align*}
\end{proposition}

\begin{proof}
By Theorem~\ref{ldt diag}, for every cocycle $D\in\Diagast$, if $\ep >0$ and $n\ge \frac{C(D)}{\ep}$ we have
\begin{equation*}
\Pp \left[  \left| \frac{1}{n} \, \log \norm{D^{(n)} } - L(D)  \right| > \ep  \right] \le e^{- c(D)\,\ep^2 \, n} \,,
\end{equation*}
where $C, c \colon \Diagast \to (0,\infty)$ are locally bounded functions. 

Therefore, there are $\delta_0=\delta_0 (A)>0$ and $\nldt=\nldt(A)\in\N$ such that if  $D\in\Diag$ with $d(A,D)<\delta_0$ and if $n\ge\nldt$, putting $\ep=n^{-1/4}$ above, we have
\begin{equation}\label{bridging eq1}
\Pp \left[  \left| \frac{1}{n} \, \log \norm{D^{(n)} } - L(D)  \right| > n^{-1/4}  \right] \le e^{- n^{1/3}} \,.
\end{equation} 

Since by Remark~\ref{continuity LE in Diag}, on $\Diag$ the LE is continuous, by possibly decreasing $\delta_0$ we may assume that $L(D) \ge \frac{L(A)}{2}$ for all cocycles $D\in \Diag$ with $d(A,D)\le\delta_0$.

Let $K=K(A)<\infty$ be such that $\log \norm{B}_\infty \le K$ for all cocycles $B\in\SLtwoR^k$ with $d(A,B)\le\delta_0$.

An easy calculation then shows that for any two cocycles $B$ and $D$ in this $\delta_0$-neighborhood of $A$, for all scales $m\in\N$ and all points $x\in X$ we have
\begin{equation}\label{proximity eq0}
\left| \frac{1}{m} \log \norm{\Bn{m} (x)} -  \frac{1}{m} \log \norm{D^{(m)} (x)}    \right| \le
e^{K m} \, d (B, D) \, .
\end{equation}

 

Let $\delta = \delta (A) > 0$ be such that $\delta < \min \{e^{-\nldt^2}, \frac{\delta_0}{3} \}$. 

Fix any cocycle $B \in \SLtwoR^k$ such that $d (B, A) \le \delta$. 

Note that since $A \in \Diag$, we have $d  (B, \Diag) \le d (B, A) \le \delta$, 
but $d  (B, \Diag)$ could be significantly smaller than $\delta$ (even zero). In order to get a (uniform) LDT estimate for $B$ by proximity, at scales $n$ in a reasonably long range, we will use the proximity of $B$ to a `quasi-closest' diagonalizable cocycle $B_\flat$, rather than to $A$.
We will distinguish between the cases $d (B, \Diag)=0$ and $d (B, \Diag)>0$.

If $d (B, \Diag)=0$ then for every $n$ there is $D_n \in \Diag$ such that $d (B, D_n) \le e^{- K n}\,n^{-1/4} < \delta_0$ if $n$ is large enough.

Then applying~\eqref{proximity eq0} with $D=D_n$ we get that for all $x\in X$,
 $$\left| \frac{1}{n} \log \norm{\Bn{n} (x)} -  \frac{1}{n} \log \norm{D_n^{(n)} (x)}    \right| \le
e^{K n} \, d (B, D_n)  < n^{-1/4} \,.$$ 

On the other hand, the LDT estimate~\eqref{bridging eq1} is applicable to the diagona\-lizable cocycle $D_n$ and for all $n\ge\nldt$ we have
$$\Pp \left[  \left| \frac{1}{n} \, \log \norm{D_n^{(n)} } - L(D_n)  \right| > n^{-1/4}  \right] \le e^{- n^{1/3}} \,.$$

Combining the previous two inequalities, we have that
$$\left| \frac{1}{n} \log \norm{\Bn{n} (x)} - L(D_n)  \right| < 2 n^{-1/4} < n^{-1/5}$$
with probability $ > 1 - e^{- n^{1/3}}$, and provided $n$ is large enough.

Integrating in $x$ we get
$$\abs{ \LE{n} (B) - L(D_n)} \le 2 n^{-1/4} + 2 \, K \,  e^{- n^{1/3}} < n^{-1/5} \, ,$$
 provided that $n\ge\nldt$ (if necessary, we slightly increase $\nldt$ so that asymp\-totic inequalities like the ones above hold for all $n \ge \nldt$).
 
 Then $B$ satisfies the following LDT estimate
 $$\Pp \left[  \left| \frac{1}{n} \, \log \norm{B^{(n)} } - \LE{n}(B)  \right| > n^{-1/5}  \right] \le e^{- n^{1/3}} $$
 for all $n\ge\nldt$, and we are done with the case $d (B, \Diag)=0$.

Now assume that $d (B, \Diag)>0$, so there is $B_\flat \in \Diag$ such that 
$$d (B, B_\flat) \le 2 \, d  (B, \Diag) \, (\le 2 \delta) \, .$$

Of course,  $d (B, B_\flat) \ge d  (B, \Diag)$, hence
$$d (B, B_\flat) \asymp d  (B, \Diag) \, .$$

Moreover, by the triangle inequality we also have
$$d (A, B_\flat) \le 3 \delta < \delta_0 ,$$
so \eqref{bridging eq1} applies to $B_\flat$ and we have that for all $n \ge \nldt$,
\begin{equation}\label{bridging eq2}
\Pp \left[  \left| \frac{1}{n} \, \log \norm{B_\flat^{(n)} } - L (B_\flat)  \right| > n^{-1/4}  \right] \le e^{- n^\frac{1}{3}} .
\end{equation}

We transfer the LDT~\eqref{bridging eq2} to $B$ by {\em proximity}, in a certain range of scales.
Put 
$$n_0 := \left( \log d(B, \Diag)^{-1}\right)^{1/4}, \ \text{ so } \ d(B, \Diag) = e^{-n_0^4} $$
and note that since  $\delta < e^{-\nldt^2}$, we have $\nldt \le n_0^2$. 
Now we derive an LDT for $B$ at any scale $n$ with $\nldt \le n \le n_0^3$, using the same procedure as above.
Applying~\eqref{proximity eq0} with $D=B_\flat$ we get that for all $x\in X$,
 $$\left| \frac{1}{n} \log \norm{\Bn{n} (x)} -  \frac{1}{n} \log \norm{B_\flat^{(n)} (x)}    \right| \le
e^{K n} \, d (B, B_\flat)  < e^{K n_0^3} \, e^{-n_0^4} <  n^{-1/4} .$$ 

Combining the above with~\eqref{bridging eq2} we get
$$ \left| \frac{1}{n} \, \log \norm{B^{(n)} } - L (B_\flat)  \right| <  2 n^{-1/4} < n^{-1/5}$$
with probability $ > 1 - e^{- n^{1/3}}$. 

Integrating in $x$ we have $$\abs{ \LE{n}(B) - L(B_\flat)} <  n^{-1/5} \,.$$ 

Then $\abs{\LE{2n}(B)-\LE{n}(B)}  \le 2 n^{-1/5}$ and   $B$ satisfies the LDT estimate
 $$\Pp \left[  \left| \frac{1}{n} \, \log \norm{B^{(n)} } - \LE{n}(B)  \right| > n^{-1/5}  \right] \le e^{- n^{1/3}} $$
 for all scales $n$ in the range  $\nldt \le n \le n_0^3 = \left( \log d(B, \Diag)^{-1}\right)^{3/4}$.
 
Finally, in the same range of scales,
$$\LE{n}(B)>L(B_\flat)-n^{-1/5}\ge\frac{L(A)}{2}-n^{-1/5}>\frac{L(A)}{3}$$
provided $\nldt$ is larger than a negative power of $L(A)$. 
\end{proof}

The next goal is to extend the range of the LDT (at the cost of weake\-ning the estimate), for scales up to an arbitrary power of $ \log d(B, \Diag)^{-1}$.

For that we use the Avalanche Principle (AP). Let us recall its statement (see for instance Theorem 2.1. in~\cite{DK-31CBM}).

\begin{lemma}\label{AP}
Let $0 < \kappa \ll \epsilon$. Given $g_0, g_1, \ldots, g_{n-1} \in \SL_2(\R)$ a sequence of matrices, if the geometric conditions 
\begin{align*} 
 & \text{(angle)}    & \frac{\norm{g_i \, g_{i-1}}}{\norm{g_i} \, \norm{g_{i-1}} }  \ge \epsilon  	\phantom{MMMMMMMMMMMMMMM} \\
 \\
 & \text{(gap)}    & \norm{g_i}   \ge \frac{1}{\kappa} 
\phantom{MMMMMMMMMMMMMMM}
\end{align*}
are satisfied for all indices $i$, then the following holds:
\begin{align*} 
\log \norm{ g^{(n)} }   & = -   \sum_{i=1}^{n-2} \log \norm{g_i} +  \sum_{i=1}^{n-1} \log \norm{ g_i   g_{i-1}}  + \bigo \left( n \, \frac{\kappa^2}{\ep^2} \right) . 
\end{align*}
\end{lemma}  

We now formulate the main result of this section.

\begin{theorem}\label{thm:bridging}
Let $A \in \Diag$ with $L(A)>0$ and fix any integer $p \in \N$. There are constants $\delta = \delta (A) > 0$ and $\nldt = \nldt (A) \in \N$ such that if $B \in \SLtwoR^k$ is any cocycle with $d (B, A) \le \delta$, then 
\begin{equation*}
\Pp \left[  \left| \frac{1}{n} \, \log \norm{\Bn{n} } - \LE{n} (B)  \right| > n^{-1/5}  \right] \le e^{- n^\frac{1}{20 p}}
\end{equation*}
holds for all scales $n$ with $\nldt \le n \le \left( \log d (B, \Diag)^{-1} \right)^p$.
\end{theorem}

\begin{proof} The constants $\delta$, $\nldt$, as well as the bound $K$ are the same as in Proposition~\ref{initial scale}. Fix any cocycle $B\in\SLtwoR^k$ with $d (B, A) \le \delta$. 

When $d (B, \Diag)=0$ there is nothing to prove, as the LDT in Proposition~\ref{initial scale} holds for all scales $n\ge\nldt$.

Assume that $d (B, \Diag)>0$ and put $n_0:=\left( \log d (B, \Diag)^{-1} \right)^{1/4}$.  Then by Proposition~\ref{initial scale}, for $m$ in the range $\nldt\le m\le n_0^3$ and for  $x\notin\B_m$, where  $\Pp [ \B_m] \le e^{- m^{1/3}}$,  we have that
\begin{equation}\label{bridge eq20}
\LE{m}(B) \ge \frac{L(A)}{3}
\end{equation}
and the following uniform LDT holds:
\begin{equation}\label{bridge eq10}
- m^{-1/5} \le  \frac{1}{m} \, \log \norm{\Bn{m} (x)} - \LE{m} (B)   \le m^{-1/5} \, .
\end{equation}


Let $N \in \N$ with $n_0^3 \le N \le n_0^{4 p} = \left( \log \imeas{B}^{-1} \right)^p$. 
Write $N = (n-1) \, n_0 + m_0$, with $n_0 \le m_0 < 2 n_0$, so that $n \asymp \frac{N}{n_0} \le n_0^{4 p -1}$.

The LDT~\eqref{bridge eq10} will allow us to apply the Avalanche Principle to matrix blocks of length $n_0$ and $m_0$, for a large enough set of phases $x$. 

Indeed, fix $x\in X$ and define
\begin{align*}
g_i  = g_i (x)  & :=   \Bn{n_0} (\transl^{i \, n_0} \, x) \quad \text{ if } \, 0\le i \le n-2 \quad \text{and}\\
g_{n-1}  = g_{n-1} (x)  & := \Bn{m_0} (\transl^{(n-1) \, n_0} \, x) 
\end{align*}

Then clearly $$g^{(n)} = \Bn{N} (x)$$ and for all $1 \le i \le n-2$,
\begin{align*}
g_i \, g_{i-1} &= \Bn{n_0} (\transl^{n_0} \ \transl^{(i-1) \, n_0} \, x) \, \Bn{n_0} (\transl^{(i-1) \, n_0} \, x) \\
&=  \Bn{2 n_0} (\transl^{(i-1) n_0}  x) , \quad \text{while}\\
g_{n-1} \, g_{n-2} &= \Bn{n_0+m_0} (\transl^{(n-2) n_0}  x) .
\end{align*}

The validity of the gap condition in the AP in Lemma~\ref{AP}, for a large number of phases $x$, is due to the left hand side of~\eqref{bridge eq10}  and the lower bound~\eqref{bridge eq20}, both applied at scales $n_0, m_0$. 

Indeed, if $x \notin \B_{n_0}$, then
\begin{align*}
\frac{1}{n_0} \log \norm{\Bn{n_0} (x)} & > \LE{n_0} (B) - n_0^{-1/5} > \frac{L(A)}{4} =: c > 0 \,,
\end{align*}
and the same holds  at scale $m_0$. 

Then for $x \notin \cup_{i=0}^{n-2} \, T^{- i n_0} \, \B_{n_0} \cup \transl^{- (n-1) \, n_0} \, \B_{m_0}$, which is a set of measure $\less n_0 \, e^{- n^{1/3}}$, and for all indices $0\le i\le n-1$, we have
$$\norm{g_i} = \norm{g_i (x)} \ge e^{c \, n_0} =: \frac{1}{\kappa} .$$ 

The validity of the angle condition in the AP is derived in a similar way. If $x \notin \B_{2 n_0} \cup \B_{n_0} \cup \transl^{- n_0} \B_{n_0} =: \Bt_{n_0}$, then
\begin{align*}
\frac{\norm{\Bn{2 n_0} (x)}}{\norm{\Bn{n_0} (\transl^{n_0} x)} \ \norm{\Bn{n_0} (x)}}  & \ge \frac{e^{2 n_0 (\LE{2 n_0} (B) - n_0^{-1/5})}}{e^{2 n_0 (\LE{n_0} (B) + n_0^{-1/5})}} \\
\\
& \kern-3em  = e^{2 n_0 (\LE{2 n_0}(B) - \LE{n_0}(B)) - 4 n_0^{4/5}} \ge e^{- 8 n_0^{4/5}} =: \ep .
\end{align*}

Then for all $x \notin \cup_{i=0}^{n-2} \, T^{- i n_0} \, \Bt_{n_0}$, and for all indices $1 \le i \le n-2$,
$$\frac{\norm{g_i \, g_{i-1}}}{\norm{g_i} \, \norm{g_{i-1}} }  \ge e^{- 8 n_0^{4/5}} =: \ep \,,$$

A similar estimate holds for the index $i=n-1$. 

Note also that
$$\frac{\kappa}{\ep} = \frac{e^{-c n_0}}{e^{- 8 n_0^{4/5}}} < e^{- c/2 \, n_0} \ll 1 \,.$$

We may then apply the AP to the list of matrices $g_0, \ldots, g_{n-2}, g_{n-1}$ and get
\begin{align}\label{linearization}
 \log \norm{\Bn{N} (x) } =  &  - \sum_{i=1}^{n-2} \, \log \norm{\Bn{n_0} (\transl^{i n_0} \, x) }  + \sum_{i=1}^{n-2} \, \log \, \norm{\Bn{2 n_0} (\transl^{(i-1) n_0} \, x) } \notag \\
 & + \log \, \norm{\Bn{n_0+m_0} (\transl^{(n-2) n_0} \, x) } +
 \bigo (n \, e^{- c \, n_0}) \, ,
 \end{align} 
 for all $x$ outside a set $\B_0$ with $\Pp [ \B_0 ] < e^{-n_0^{1/3}} \, n_0^{4 p} < e^{-n_0^{1/4}}$.


The right hand side of \eqref{linearization} can be broken down into three sums of i.i.d. random variables. 

Indeed, since the matrix blocks $\Bn{n_0} (\transl^{i n_0} \, x) $ use {\em disjoint} sets of $x$-coordinates for distinct values of the index $i$, the sequence of random variables $\left\{ \log \norm{\Bn{n_0} (\transl^{i n_0} \, x)}  \colon i = 1, \ldots, n-2 \right\}$ is already independent (and of course identically distributed). The common expected value of these random variables is $n_0 \, \LE{n_0}(B)$ and their common (and uniform) $L^\infty$ bound is $n_0 \, K$. Hoeffding's inequality (Lemma~\ref{Hoeffding}) is then applicable and it implies, putting $\ep=n_0 \, n^{-1/4}$,
\begin{align*}
\Pp \left[  \Big| \sum_{i=1}^{n-2} \, \log \norm{\Bn{n_0} (\transl^{i n_0} \, x) } - (n-2) \, n_0 \, \LE{n_0}(B) \Big| > n_0 \, n^{3/4} \right] \\
& \kern-7em < e^{- \frac{1}{2 K^2} \, n^{1/2}} < e^{- n_0} \, .
\end{align*}

Let us denote by $\B_1$ the event above, so $\Pp [ \B_1 ] < e^{- n_0}$.

The other sequence of random variables 
$$\left\{ \log \, \norm{\Bn{ 2 n_0} (\transl^{i n_0} \, x) }  \colon i = 1, \ldots, n-2 \right\}$$
is not independent, as consecutive matrix blocks defining  the terms of the  sequence depend on interlacing sets of $x$-coordinates. However, we may split it into two sequences, corresponding to the even and respectively the odd values of the indices $i$: 
\begin{align*}
\left\{ \log \, \norm{\Bn{ 2 n_0} (\transl^{i n_0} \, x) } \colon i = 1, \ldots, n-2, \ i \text{ odd } \right\},  \\
\left\{ \log \, \norm{\Bn{ 2 n_0} (\transl^{i n_0} \, x) }  \colon i = 1, \ldots, n-2, \ i \text{ even } \right\} ,
\end{align*}
each of which consisting now of independent random variables, as the corresponding matrix blocks depend of disjoint sets of $x$-coordinates.

The common expected value is now $2 n_0 \, \LE{2 n_0} (B)$ and the uniform $L^\infty$ bound is $2 n_0 \, K$. Applying Hoeffding's inequality with $\ep=2 n_0 \, n^{-1/4}$ to each of the two i.i.d. sequences above and then summing, we conclude as before that
$$\Pp \left[  \Big| \sum_{i=1}^{n-2} \, \log \norm{\Bn{2 n_0} (\transl^{i n_0} \, x) } - (n-2) \, 2 n_0 \, \LE{2 n_0}(B) \Big| > 4 n_0 \, n^{3/4} \right] 
< e^{- n_0} .$$

Let us denote by $\B_2$ the event above, so $\Pp [ \B_2 ] < e^{- n_0}$.

Combining~\eqref{linearization} with the previous two estimates derived via Hoeff\-ding's inequality, we obtain
\begin{align}\label{linearization 2}
 \log \norm{\Bn{N} (x) } = & - (n-2) \, n_0 \, \LE{n_0}(B) + (n-2) \, 2 n_0 \, \LE{2 n_0}(B) +  \bigo(n_0 \, n^{3/4}) \notag\\
 + & \log \, \norm{\Bn{n_0+m_0} (\transl^{(n-2) n_0} \, x) } +
 \bigo (n \, e^{- c \, n_0})
\end{align}
 for all $x\notin\B:=\B_0\cup\B_1\cup\B_2$, where $\Pp [ \B ] < e^{- n_0^{1/4}}+e^{- n_0}+e^{- n_0}<e^{- n_0^{1/5}}$.
 
 Note that since $N \asymp n \, n_0 \asymp n_0^{4 p} < e^{c \, n_0}$, then
\begin{align*}
 \frac{(n-2) \, n_0}{N} = 1 + \bigo(n^{-1}), &  \quad \frac{(n-2) \, 2 n_0}{N} = 2 + \bigo(n^{-1}), \\
 \frac{n_0 \, n^{3/4}}{N} = \bigo(n^{-1/4}), & \qquad \frac{n \, e^{- c \, n_0}}{N} = \bigo(n^{-1}) \, .
\end{align*}

Dividing both sides of~\eqref{linearization 2} by $N$ we then have: 
\begin{align*}
\frac{1}{N} \,  \log \norm{\Bn{N} (x) } = &  - \LE{n_0}(B)  + 2  \LE{2 n_0}(B) + \bigo(n^{-1/4})
\end{align*}
for all $x\notin\B$, where $\Pp [ \B ] < e^{- n_0^{1/5}} < n^{-1/4}$.

Integrating in $x$,
\begin{align*}
\LE{N}(B) & =  - \LE{n_0}(B)  + 2  \LE{2 n_0}(B) + \bigo(n^{-1/4}) + \bigo \left( K \, \Pp [ \B ] \right)\\
& =  - \LE{n_0}(B)  + 2  \LE{2 n_0}(B) + \bigo(n^{-1/4}) .
\end{align*}

We may then conclude that if $x\notin\B$ then
$$\frac{1}{N} \,  \log \norm{\Bn{N} (x) } = \LE{N}(B) + \bigo(n^{-1/4}) .$$

Thus
\begin{equation*}
\Pp \left[  \left| \frac{1}{N} \, \log \norm{\Bn{N} } - \LE{N} (B)  \right| > N^{-1/5}  \right] \le \Pp [ \B ] < e^{- n_0^{1/5}} <
e^{- N^\frac{1}{20 p}} ,
\end{equation*}
which completes the proof.
\end{proof}

\section{Proofs of the main results }\label{proofs}
We combine the results of the previous sections to derive a uniform LDT estimate in the space of all random cocycles over a finite set of symbols, and as a consequence of that, a modulus of continuity for the Lyapunov exponent.

\begin{proof}[Proofs of Theorems~\ref{main-uniform-ldt} and~\ref{main-continuity}]
By Remark~\ref{reduction diag remark}, it is enough to establish uniform (relative to the whole space $\SLtwoR^k$) sub-exponential LDT estimates for diagonalizable $\SLtwoR$ cocycles with positive LE.

Let $A\in\Diag$ with $L(A)>0$. 

By Theorem~\ref{ldt diag}, there are constants $\delta_1>0$, $c>0$ and $C<\infty$ depending on $A$, 
such that if $D\in\Diag$ is any diagonalizable cocycle with $d(D,A)\le\delta_1$, then
\begin{equation*}
\Pp \left[  \left| \frac{1}{n} \, \log \norm{D^{(n)} } - L(D)  \right| > \ep  \right] \le e^{- c(A)\,\ep^2 \, n}  
\end{equation*}
for all $\ep >0$ and $n\ge \frac{C(A)}{\ep}$. 

In particular, putting $\ep:=n^{-1/4}$, there is $\nldt_1=\nldt_1(A)\in\N$ such that 
for all $n\ge\nldt_1$, we have $$\Pp \left[  \left| \frac{1}{n} \, \log \norm{D^{(n)} } - L(D)  \right| > n^{-1/4}  \right] \le e^{- n^{1/3}} \,. $$

Integrating in $x$, we get (for a constant $K=K(A)<\infty$),
$$\abs{\LE{n}(D)-L(D)}\le n^{-1/4} + K \, e^{- n^{1/3}} < 2 \, n^{-1/4} \ll n^{-1/5} \, ,$$
so we can essentially replace $L(D)$ by $\LE{n}(D)$.

We conclude the following. There are constants $\delta_1>0$ and $\nldt_1\in\N$ depending on $A$, such that if $D\in\Diag$ and $d(D,A)\le\delta_1$, then for all $n\ge\nldt_1$ we have the following:
\begin{equation}\label{proofs diag}
\Pp \left[  \left| \frac{1}{n} \, \log \norm{D^{(n)} } - \LE{n}(D)  \right| > n^{-1/5}  \right] \le e^{- n^{1/3}} \, .
\end{equation}

If $B\notin\Diag$, then $\rho(B)=\max\{ \rho_-(B), \rho_+(B)\} > 0$, so either $\rho_-(B)>0$ or $\rho_+(B)>0$. But since $\rho_-(B^{-1})\asymp\rho_+(B)$, the latter case is equivalent to $\rho_-(B^{-1})>0$.

By Theorem~\ref{quant LDT Berboulli}, there are constant $\delta_2>0$ and $c>0$ depending on $A$, such that if $B \in \SLtwoR^k$ is any cocycle with $d (B, A) \le \delta_2$ and $\rho_-(B)>0$, then
\begin{equation}\label{proofs quantldt}
\Pp \left[ \, \left| \frac{1}{n} \, \log \norm{\Bn{n} } - L (B)  \right| > n^{-1/6} \, \right] <  e^{- n^ {1/6}} 
\end{equation}
holds for all $n \ge \left[ \, c^{-1}\, \log \left(1/\rho_-(B)\right)   \, \right]^{54}$. 

As before, we may essentially replace in~\eqref{proofs quantldt}  $L(B)$ by $\LE{n}(B)$. 

Moreover, $A^{-1}$ is also diagonalizable and $L(A^{-1}) = L(A)>0$, so \eqref{proofs quantldt} holds for $B^{-1}$ too, provided it is close enough to $A^{-1}$. 
Since the map $\SLtwoR^k \ni B \mapsto B^{-1}\in\SLtwoR^k$ is locally bi-Lipschitz, that proxi\-mity is ensured by the proximity of $B$ to $A$ (we may have to decrease $\delta_2$ slightly). Furthermore, by Lemma~\ref{inverse ldt}, any LDT that holds for a cocycle holds for its inverse. 

From these considerations we derive the following. 
There are cons\-tants $\delta_2>0$ and $c>0$ depending on $A$, such that if $B \in \SLtwoR^k$ is any cocycle with $d (B, A) \le \delta_2$ and $\rho_-(B^\pm)>0$, then
\begin{equation}\label{proofs eq2}
\Pp \left[ \, \left| \frac{1}{n} \, \log \norm{\Bn{n} } - \LE{n}(B)  \right| > n^{-1/7} \, \right] <  e^{- n^ {1/6}} 
\end{equation}
holds for all $n \ge \left[ \, c^{-1}\, \log \left(1/\rho_-(B^\pm)\right)   \, \right]^{54}$.

Recall from Proposition~\ref{diag rho} that there are constants $\delta_3>0$ and $C<\infty$  depending on $A$, such that for all cocycles $B$ with $d(B,A)\le\delta_3$, either $d(B,\Diag) \le C\, \rho_-(B)$ or $d(B,\Diag) \le C\, \rho_-(B^{-1})$.

Choose the sign $\varepsilon=\pm$ so that $d(B,\Diag) \le C\, \rho_-(B^\varepsilon)$. 

Let $\delta_4:=\min \{\delta_2, \delta_3\}$. If $B\notin\Diag$ (so in particular $d (B, \Diag)>0$) and if $d(B,A)\le\delta_4$, then~\eqref{proofs eq2} holds for all $n \ge \left[ \, c^{-1}\, \log \left(1/\rho_-(B^\varepsilon)\right)   \, \right]^{54}$.

But  $c^{-1}\, \log \left(1/\rho_-(B^\varepsilon)\right) \le c^{-1} \, \log C + c^{-1} \, \log d(B, \Diag)^{-1}$.

Since $d (B, \Diag)\le d(B,A)\le\delta_4\ll 1$, by possibly decreasing $\delta_4$, we may assume that 
$$c^{-1} \, \log C + c^{-1} \, \log d(B, \Diag)^{-1}\le\left( \log d (B, \Diag)^{-1}\right)^{1+1/9} .$$

Consequently, for any cocycle $B\in\SLtwoR^k$ with $d (B, A)\le\delta_4$, if $B\notin\Diag$ then
\begin{equation}\label{proofs eq3}
\Pp \left[ \, \left| \frac{1}{n} \, \log \norm{\Bn{n} } - \LE{n}(B)  \right| > n^{-1/7} \, \right] <  e^{- n^ {1/6}} 
\end{equation}
holds for all $n\ge\left( \log d (B, \Diag)^{-1}\right)^{60}$.

\smallskip

Finally, we apply the bridging argument in Theorem~\ref{thm:bridging} with $p=60$. There are constants $\delta_5>0$ and $\nldt_2\in\N$ depending on $A$,  such that if $B \in \SLtwoR^k$ is any cocycle with $d (B, A) \le \delta_5$ and if $B\notin\Diag$, then 
\begin{equation}\label{proofs bridging}
\Pp \left[  \left| \frac{1}{n} \, \log \norm{\Bn{n} } - \LE{n} (B)  \right| > n^{-1/5}  \right] \le e^{- n^\frac{1}{1200}}
\end{equation}
holds for all scales $n$ with $\nldt_2 \le n \le \left( \log d (B, \Diag)^{-1} \right)^{60}$.

From~\eqref{proofs eq3} and~\eqref{proofs bridging} we derive a uniform LDT for non diagonalizable cocycles, while~\eqref{proofs diag} provides the same for diagonalizable cocycles. We conclude: there are constants $\delta>0$ and $\nldt\in\N$ depending on $A$, so that for any cocycle $B\in\SLtwoR^k$ with $d (B, A)\le\delta$ and for all $n\ge\nldt$, we have
\begin{equation}\label{proofs eq5}
\Pp \left[  \left| \frac{1}{n} \, \log \norm{\Bn{n} } - \LE{n} (B)  \right| > n^{-1/5}  \right] \le e^{- n^\frac{1}{1200}} \, .
\end{equation}

This is the form
 of the uniform LDT required by our abstract continuity theorem (ACT) (see Section 3.2 in~\cite{DK-31CBM})\footnote{In~\cite{DK-31CBM} the uniform LDT is exponential, and as a result, the modulus of continuity of the LE and that of the Oseledets decomposition are H\"older. As alluded to at the end of Section 3.2 of~\cite{DK-31CBM} (or using our more general ACT in~\cite{DK-book}), a sub-exponential uniform LDT leads to a weak-H\"older modulus of continuity.}. Consequently, the continuity properties of the LE and of the Oseledets decomposition formulated in Theorem~\ref{main-continuity} follow from~\eqref{proofs eq5} by means of the ACT (Theorem 3.3 in~\cite{DK-31CBM}).
 The estimate~\eqref{proofs eq5} also implies (see Proposition 3.4 in~\cite{DK-31CBM}) a uniform rate of convergence of the finite scale LE 
 $\LE{n}(B)$ to $L(B)$. Therefore, we may essentially replace in~\eqref{proofs eq5} $\LE{n}(B)$ by $L(B)$, and derive a uniform LDT as formulated in Definition~\ref{def:weakLDT}, with $a=-1/6$ and $b=1/1200$.
\end{proof}

\begin{proof}[Proof of Theorem~\ref{cont-IDS}] Let $\{v_n\}_{n\in\Z}$ be an i.i.d. sequence of random variables with common probability distribution $\mu$ and let $\{w_n\}_{n\in\Z}$ be another i.i.d. sequence of random variables with distribution $\nu$, such that moreover $\{v_n\}$ and $\{w_n\}$ are independent from each other. Consider the 
random Jacobi operator $H_\om$ acting on $l^2(\Z)$ and given by
\begin{equation}\label{proofs Jacobi-op}
\big[ H_\om \psi \big]_n  := - (w_{n+1} (\om) \psi_{n+1} + w_n (\om) \psi_{n-1}) 
+ v_n (\om)  \psi_n  \, .
\end{equation}

We assume that the supports of the distributions $\mu$ and $\nu$ are finite,  and that $\EE ( \log \sabs{w_1} ) > - \infty$, that is, $0 \notin \text{ supp }  (\nu)$. 

The Schr\"odinger (or eigenvalue) equation $H_\om \, \psi = E \, \psi$ for $E\in\R$ and $\psi = \{\psi_n\}_{n\in\Z} \in \R^\Z$, is a second order finite difference equation. It can be written as the matrix first order finite difference equation
$$
\left[\begin{array}{@{}c@{}}
    \psi_{n+1} \\
    \psi_{n}  
\end{array} \right]
= g_n^E \, 
\left[\begin{array}{@{}c@{}}
    \psi_{n} \\
    \psi_{n - 1}  
\end{array} \right] \, ,
$$
 where
\begin{equation*}
g_n^E = \begin{bmatrix}  
\frac{v_n-E}{w_{n+1}} &  - \frac{w_n}{w_{n+1}}  \\
& \\
1 & 0
\end{bmatrix} \, .
\end{equation*}

The Lyapunov exponent $L^+(E)$ of the operator~\eqref{proofs Jacobi-op} at energy $E$ is thus the Lyapunov exponent $L^+ (g^E)$ of the i.i.d. multiplicative process  
$g^E := \{g_n^E\}_{n\in\Z}$.
This process can be conjugated with the i.i.d. $\SLtwoR$ process $\tilde{g}^E:=\{\tilde{g}^E_n\}_{n\in\Z}$, where
\begin{equation*}
\tilde{g}_n^E = \begin{bmatrix}  
\frac{v_n-E}{w_{n}} &  - w_n  \\
\\
\frac{1}{w_n} & 0
\end{bmatrix} \, ,
\end{equation*}
in the sense that
$$g^E_n =  
\begin{bmatrix}  
w_{n+1} &  0 \\
0 & 1
\end{bmatrix}^{-1} 
\,
\tilde{g}^E_n
\,
\begin{bmatrix}  
w_{n} &  0 \\
0 & 1
\end{bmatrix} \, .
$$

Thus for all $E\in\R$ we have 
$$L^+ (E)  = L^+(g^E) = L (\tilde{g}^E) \ge 0 \, .$$

Theorem~\ref{main-continuity} is applicable to $\tilde{g}^E$ and it implies that the function $E\mapsto L (\tilde{g}^E) = L^+(E)$ is continuous. In particular, if $E_0\in\R$ is such that $L^+(E_0)>0$, there is an open interval $\I\ni E_0$ such that $\restr{L^+}{\I}>0$ and $\restr{L^+}{\I}$ is weak H\"older continuous. 

By means of the Thouless formula, the  weak H\"older continuity of the IDS also follows.

Finally, let us discuss the dynamical localization of the operator~\eqref{proofs Jacobi-op}. Theorem 5.1 in~\cite{Chapman-Stolz} establishes such a result (for the more general se\-tting of block Jacobi operators) under the assumptions of Furstenberg's theory. However, these assumptions are only needed insofar as to guaran\-tee the positivity of the LE and the H\"older continuity of the IDS, where the latter is used to derive a Wegner estimate. In our case, the positivi\-ty of the LE is assumed, while the weak H\"older continuity of the IDS still implies a (good enough, albeit weaker) Wegner estimate. 
\end{proof}

\medskip

\subsection*{Acknowledgments}

The first author was supported  by Funda\c{c}\~{a}o para a Ci\^{e}ncia e a Tecnologia, under the projects: UID/MAT/04561/2013 and   PTDC/MAT-PUR/29126/2017.
The second author has been supported in part by the CNPq research grant 306369/2017-6 (Brazil) and by a research productivity grant from his current institution (PUC-Rio). He would also like to acknowledge the support of the Isaac Newton Institute for Mathematical Sciences in Cambridge, UK, where the work on this project first started, during the PEPW04 workshop in 2015.

\bigskip


\providecommand{\bysame}{\leavevmode\hbox to3em{\hrulefill}\thinspace}
\providecommand{\MR}{\relax\ifhmode\unskip\space\fi MR }
\providecommand{\MRhref}[2]{%
  \href{http://www.ams.org/mathscinet-getitem?mr=#1}{#2}
}
\providecommand{\href}[2]{#2}

\end{document}